\documentclass{memo-l}

\usepackage{amsmath}
\usepackage{amssymb}
\usepackage[all]{xy}
\usepackage{epsfig}
\usepackage{color}
\usepackage{mathabx}
\usepackage{tikz}
\usetikzlibrary{arrows,calc,matrix}
\usepackage[english]{babel}
\usepackage[original]{imakeidx}
\makeindex 

\title[Horocycle dynamics and the eigenform loci]{Horocycle dynamics:
  new invariants  and
  eigenform loci in the stratum 
${\mathcal H}(1,1)$}
\author{Matt Bainbridge}
\address{University of Indiana {\tt mabainbr@indiana.edu}}
\author{John Smillie}
\address{University of Warwick {\tt j.smillie@warwick.ac.uk}}
\author{Barak Weiss}
\address{Tel Aviv University 
{\tt barakw@post.tau.ac.il}}

\newcommand\pr{\mathrm{pr}}
\newcommand\Blo{\mathrm{Blo}}
\newcommand\dev{\mathrm{dev}}
\newcommand\hol{\mathrm{hol}}

\newcommand{\HH}{{\mathcal{H}}}

\newcommand{\Q}{{\mathbb {Q}}}

\newcommand{\R}{{\mathbb{R}}}

\newcommand{\Res}{{\mathrm{Res}}}
\newcommand{\Z}{{\mathbb{Z}}}

\newcommand{\C}{{\mathbb{C}}}
\newcommand{\BB}{{\mathcal{B}}}
\newcommand{\CC}{{\mathcal{C}}}

\newcommand{\N}{{\mathbb{N}}}

\newcommand{\GL}{\operatorname{GL}}

\newcommand{\Mod}{\operatorname{Mod}}

\newcommand{\SL}{\operatorname{SL}}

\newcommand{\Lie}{\operatorname{Lie}}

\newcommand{\SO}{\operatorname{SO}}

\newcommand{\End}{{\rm End}}

\newcommand{\conv}{{\rm conv}}

\newcommand{\EE}{{\mathcal{E}}}

\newcommand{\LL}{{\mathcal L}}

\newcommand{\df}{{\, \stackrel{\mathrm{def}}{=}\, }}

\newcommand{\til}{\widetilde}

\newcommand{\supp}{{\rm supp}}

\newcommand{\rel}{{\mathrm{Rel}}}

\newcommand{\sm}{\smallsetminus}

\newcommand{\vre}{\varepsilon}

\newcommand{\zed}{\mathbb{Z}}
\newcommand{\ratls}{\mathbb{Q}}
\newcommand{\reals}{\mathbb{R}}
\DeclareMathOperator{\Jac}{Jac}
\DeclareMathOperator{\RE}{Re}

\newtheorem{thm}{Theorem}[section]

\newtheorem{lem}[thm]{Lemma}

\newtheorem{prop}[thm]{Proposition}

\newtheorem{cor}[thm]{Corollary}

\newtheorem{claim}{Claim}

\newtheorem{remark}[thm]{Remark}

\newtheorem{dfn}[thm]{Definition}

\usepackage[colorlinks=true,linkcolor=blue]{hyperref}

\begin{document}
\date{\today}


\begin{abstract}
We study dynamics of the horocycle flow on strata of translation
surfaces, introduce new invariants for ergodic measures, and analyze the
interaction of the horocycle flow and real Rel surgeries. 
We use 
this analysis to complete and extend results of Calta and Wortman
classifying horocycle-invariant measures in the eigenform loci. In addition we classify the 
horocycle orbit-closures and prove that every orbit is
equidistributed in its orbit-closure. We also prove equidistribution
results describing limits of sequences of measures. Our results have applications
to the problem of counting closed trajectories on translation surfaces of genus 2.
\end{abstract}

\maketitle
\tableofcontents
\listoffigures

\section{Introduction}
Translation surfaces arise naturally in many different mathematical
contexts, e.g. complex analysis, geometric group theory, geometry and
polygonal billiard dynamics. 
See \cite{MT, zorich survey, Wright_survey} for surveys of translation surface theory
from a dynamical perspective.
A {\em stratum} is a moduli space of translation surfaces of a given 
topological type 
(detailed definitions will be given in \S\ref{sec:strata}). There is a natural geometric action of $\SL_2(\R)$ on strata.
The one-parameter subgroup $U$ of upper triangular unipotent matrices
plays an especially interesting role. Understanding the dynamics of
the action of $\SL_2(\R)$ and $U$ is very useful in understanding
translation surface dynamics and geometry.  
We will give an example of this shortly.

The simplest stratum is the stratum $\HH(0)$ of the torus with one marked point. 
This stratum has the structure of the homogeneous space
$\SL_2(\R)/\SL_2(\Z)$. In this case we can identify the flow $U$ with
the classical horocycle flow.  The dynamical analysis of the horocycle flow in this setting is due to
Hedlund \cite{Hedlund} and Dani \cite{Dani 78, Dani SL2R}. 
For
general dynamical systems we have a description for typical orbits,
but there are usually orbits with exceptional behavior, e.g. we may expect orbit closures to
be fractal sets. 
In stark contrast with a typical dynamical system, for the horocycle
flow, the results of Hedlund and Dani fully describe the distribution
of {\em every} orbit: every orbit is either closed or
dense, and in particular each orbit closure is a manifold and supports a
natural measure. In addition every
dense orbit is uniformly distributed with respect to the globally
supported measure $\mu$. Limits of
measures supported on closed horocycles converge to $\mu$, as do limits of
expanded circle measures.

In the case of the torus, the $U$ flow on a stratum falls can be
analyzed as a particular case of a unipotent flow on a homogeneous
space. The analogy between homogeneous dynamics and
moduli space 
dynamics has played a major role in the development of
the theory. 
The simplest stratum which is not a homogeneous space is that of
translation surfaces of genus 
two. We do not currently have a theory of the dynamics of the
$U$-action on the genus two strata but we can refine the 
question. Since $U\subset\SL_2(\R)$ any $\SL_2(\R)$ orbit closure is
$U$-invariant. We can ask about the dynamics of horocycle flows on
$\SL_2(\R)$ orbit closures which we will call {\em loci} and denote by
$\LL$. In fact we will be able to understand the horocycle dynamics on
{\em proper} loci in genus two stratum components. 

Some interesting loci in genus two were discovered by Veech. He found
3-dimensional loci corresponding to regular pentagon and decagon
billiard tables. These loci correspond to closed $\SL_2(\R)$-orbits
and have the form $\SL_2(\R)/\Gamma$ where  $\Gamma\subset \SL_2(\R)$
is a non-uniform lattice. We will call such loci {\em Veech loci}.
In particular they are examples of homogeneous spaces, and the
horocycle flow on these loci behaves like the horocycle flow in the
torus case. 
Five dimensional loci in genus two were discovered by Calta and McMullen and are called eigenform loci. 
We denote them as $\EE_D(1,1)$ where $D \geq 4$ is an integer congruent to either 0 or 1
mod 4. When $D$ is not a square such a $D$ is the discriminant of a unique order in a quadratic number field and these
orders play a central role in the construction of $\EE_D(1,1)$. When $D$ is a square $\EE_D(1,1)$ consists of surfaces
which cover tori. We refer to this setting as the arithmetic case. See Section \ref{section: eigenform locus}
for more detailed information.
The collection of all loci in genus two was classified by McMullen
\cite{McMullen-SL(2)}. He showed that they are all either Veech loci
or eigenform loci. 

The first description 
 of ergodic invariant measures for eigenform loci is due to Calta and
 Wortman, see 
\cite{CW}. In this paper we will 
show that every result for the horocycle flow in $\SL_2(\R)/\Gamma$
has an analog
 for 
eigenform loci. 
In Theorem \ref{thm: Calta Wortman revisited} we prove a
classification of ergodic invariant measures, extending \cite{CW}. In Theorem \ref{thm:
  orbit closures} we completely describe orbit-closures and we
describe the distribution of each orbit. In Theorem \ref{thm: teaser1}
we describe how families of closed horocycle orbits of increasing
length behave. In Theorem \ref{thm: circle averages} we prove that
circles of increasing radii equidistribute relative to the smooth
measure.

In various settings a good understanding of the horocycle action on a
locus $\LL$ has lead to a quantitative understanding of the rate of
growth of closed orbits on translation surfaces in $\LL$ and the rate
of growth of closed billiard orbits for rational polygonal billiard
tables corresponding to surfaces in $\LL$. In particular proving the
equidistribution of large circles leads to this result. See
\cite{Veechstrata, Veech - alternative, EMS, EMM} and Theorem
\ref{thm: Matt absolved} for more information.

The question discussed in this paper concerning $U$-invariant ergodic
measures and $U$ orbit-closures is finer than the similar questions for
$\SL_2(\R)$-invariant measures and orbit-closures. For the latter
question, recent breakthroughs were made by Eskin, Mirzakhani and Mohammadi
\cite{EM, EMM}, extending the earlier work of McMullen
\cite{McMullen-SL(2)}. We refer to \cite{Wright_survey} for a survey of
these developments.

\subsection{The horocycle flow on eigenform loci}
Before stating our results we point out some ways in which horocycle dynamics in the eigenform loci differ from those in Veech loci. For any locus $\LL$ we can distinguish two types of horocycle
invariant orbits, those that are invariant under $\SL_2(\R)$ and those
that are not. Horocycle measures which are not $\SL_2(\R)$-invariant
come in continuous families which we call {\em beds}. For example in the Veech
loci 
these are continuous families of closed horocycle orbits, and 
for each Veech locus there are
finitely many such beds. These can be described in terms of surfaces with
cylinder decompositions of a fixed combinatorial type.  
They also form orbits under the action of the upper triangular group,
which normalizes the horocycle flow. Both of these descriptions admit
an interesting generalization in the context of eigenform loci.

Surfaces with horizontal cylinder decompositions have horizontal saddle connections and these
are preserved by the horocycle flow. In the setting of eigenform loci there are surfaces with
horizontal saddle connections which do not have cylinder decompositions as well as those that do.
A useful invariant to describe these beds is the `horizontal data
diagram' $\Xi(\mu)$ (see \S \ref{sec: separatrix 
  diagram}), recording topological and geometric structures which are
invariant under the $U$-action. This includes the number and length of
horizontal saddle connections, the topology of the complements, and
the incidence of these saddle connections at singularities. 
Horizontal data diagrams play a central role in describing certain beds.

It is a general 
principle of homogeneous dynamics that if we are interested in a flow
then the normalizer and centralizer of that flow take orbit closures
to orbit closures. The normalizer of the horocycle flow includes the
geodesic flow 
and for Veech loci, each bed is just a union of closed horocycles
which are permuted by the geodesic flow. In eigenform loci 
we have additional centralizing elements. 
These  
involve the real Rel vector fields, which correspond to local 
motion in a direction in a stratum which moves one singularity
horizontally with respect to another, while fixing absolute periods of
the surface. As observed by Calta \cite{Calta}, these vector fields
commute with the $U$-action and can  
be used to create new $U$-orbits out of old $U$-orbits. Combining these vector fields
with the geodesic flow vector field provides directions a family of vector fields which normalize 
the $U$-action, sending $U$-orbits to
$U$-orbits, while inducing a time-change on the orbits.

There is a fundamental difference between homogeneous 
spaces and strata of translation surfaces, in connection with these
vector fields. On homogeneous spaces, natural vector fields can be integrated to
define group actions on the space, and the centralizer and
the normalizer subgroups of the $U$-action play an important role in
studying the dynamics. In the case of strata,
motion in the real Rel directions is not globally
defined; i.e. the solution curves for the differential equation
defined by these vector fields may not be defined for all times. This
implies that the centralizer and normalizer of the horocycle
flow make sense locally but do not correspond to globally defined 
actions of Lie groups.  In \cite{EMM} issues related to locally defined flows are discussed using the
terminology of `pseudo group-actions'.  We will avoid this terminology here
and instead use the more standard language of vector fields and the trajectories they generate. 

In the first part of this paper (\S \ref{sec:strata} -- \S \ref{sec:
  separatrix diagram}) we discuss invariants for the
$U$-action which are defined in terms of the horizontal data diagram
and the real Rel vector field. These are discussed for arbitrary
strata and arbitrarily sub-loci. Then (from \S \ref{section: eigenform
locus} onwards) we specialize to the eigenform
loci in the stratum $\HH(1,1)$, and use these invariants to fully
describe the horocycle dynamics. 
The following result is the following measure classification result,
which is the 
basis for all of our subsequent results on the eigenform loci.

\begin{thm}\label{thm: Calta Wortman revisited} 
Let $D$ be as above and let $\mu$ be a $U$-invariant $U$-ergodic Borel probability measure on
$\EE_D(1,1)$. Then one of the following holds: 
\begin{enumerate}
\item 
Every surface in $\supp \, \mu$ has a horizontal cylinder
decomposition and $\mu$ is the length measure on a periodic $U$-orbit.  
\item 
Every surface in $\supp \, \mu$ has a horizontal cylinder
decomposition into three cylinders and $\mu$ is the area measure on a 2-dimensional
minimal set for the $U$-action. In this case $\mu$ is invariant under
the real Rel operation. 
\item
For every $M \in \supp \, \mu$, the horizontal data diagram $\Xi(M)$
contains two saddle  
connections, joining distinct singularities, whose union divides $M$ into
two isogenous tori glued along a slit. In this case 
$\mu$ is the image of the $\SL_2(\R)$-invariant measure on a quotient
$\SL_2(\R)/\Gamma$ for some lattice 
$\Gamma$, via a Borel $U$-equivariant map. 

\item
For every $M \in \supp \, \mu$, $\Xi(M)$ contains one saddle connection
joining distinct singularities,  and $\mu$ is the image of the
$\widehat{G}$-invariant measure on a quotient 
$\widehat{G}/\Gamma$ for some lattice 
$\Gamma$ in the 3-fold connected cover $\widehat{G}$ of $\SL_2(\R)$, via a
Borel $U$-equivariant map.

\item
The set of surfaces with horizontal saddle connections has $\mu$ measure zero
 and $\mu$ is the
image of the $\SL_2(\R)$-invariant measure on a closed $\SL_2(\R)$-orbit in $\HH(1,1)$ via a
Borel 
$U$-equivariant map. In this case $D$ is either a square or is equal
to 5. 
\item
For every $M \in \supp \, \mu$, $\Xi(M)$ contains two saddle
connections joining distinct singularities, whose complement in $M$ is
a torus with two parallel slits of equal length, which are images of
each other under a translation by an exactly $d$-torsion element of the
torus. In this case $D=d^2$ is a square and $\mu$ is
the image, via a Borel $U$-equivariant map, of the $\SL_2(\R)$-invariant
measure on the space of tori. 
\item
$\mu$ is the canonical flat
measure on $\EE_D(1,1)$ obtained from period coordinates. 
\end{enumerate}
\end{thm}

We describe these measures 
 in detail in \S \ref{section: construction}. 
The construction of most of the measures in this list involves the
real $\rel$ operation on $\HH(1,1)$ described earlier.
Any $U$-invariant measure $\mu$, which is not preserved by real Rel,
gives rise to a one-parameter family of $U$-invariant translates of
$\mu$ by real Rel. This observation of Calta \cite{Calta, CW} is
crucial to our analysis, and the measures (3)--(6)  
all arise in this way as real Rel translates of $\SL_2(\R)$-invariant measures. The measures (5) are the
real Rel translates of the natural measures on closed $\SL_2(\R)$-orbits in
$\HH(1,1)$. Loosely speaking, the measures (3),(4), and (6) are all 
pushforwards of measures on closed $\SL_2(\R)$-orbits in a suitable boundary
component in a bordification of
$\HH(1,1)$, where the maps pushing the measure consist of the
composition of real Rel with a map passing from the
boundary component to $\HH(1,1)$. In (4), the closed $\SL_2(\R)$-orbit belongs to the
stratum $\HH(2)$, while in (3) it belongs to $\HH(0)\times\HH(0)$, and
in (6) it belongs to $\HH(0,0)$ (where $\HH(0)$ and $\HH(0,0)$ denote
respectively the moduli spaces of genus one 
translation surfaces with one or two marked points).

 The equivariant maps in cases (3)--(6) arise from the real Rel operation and are not defined for
every point in $\SL_2(\R)/\Gamma$. They define Borel isomorphisms between
the supports of these measures and homogeneous spaces (that is,
quotients $\SL_2(\R)/\Gamma$ for lattices $\Gamma \subset  \SL_2(\R)$). We emphasize
however that these maps are not everywhere 
defined, but rather on a dense open set of full measure, and thus
their existence does not imply the existence of 
homeomorphisms of the supports of 
these measures with homogeneous spaces of $\SL_2(\R)$.
In fact such 
homeomorphisms do not exist in general; 
in a forthcoming paper \cite{examples}, we show that the supports of these measures
are not homeomorphic to  homogeneous spaces. In particular the
supports of the measures appearing in case (5) above are
manifolds with nonempty boundaries and infinitely generated fundamental
groups.

Theorem \ref{thm: Calta Wortman revisited} extends a theorem of Calta and Wortman \cite{CW}. In
\cite{CW} it was assumed that $D$ is not a square, and thus case (6)
did not arise. Also \cite{CW} assumed that $\mu$ is not invariant
under the real Rel operation, and case (2) did not arise. Theorem
\ref{thm: Calta Wortman revisited} is proved in a more general form in
\S \ref{sec: classification}. The statement is inspired by Ratner's
measure classification theorem \cite[Thm. 3.3.2]{KSS handbook}, and
its proof is inspired by \cite{EMM, CW}, 
which in turn employs arguments of Ratner. We provide some
technical shortcuts and clarify some delicate 
steps. Our treatment
relies on the general results worked out in \S\ref{subsec: group of a
  measure}, \S\ref{subsec: generic points}
and \S\ref{subsec: explicit}. 

The measures in cases (1)--(6) of Theorem \ref{thm: Calta Wortman
  revisited} are naturally organized in beds. In case (5) the continuous parameter
involves a real Rel perturbation, and in case (2) there is a natural
two parameter family of perturbations.

Applying our measure classification, we classify all orbit-closures for the
$U$-action on $\EE_D(1,1)$. 
In fact we prove a stronger
statement. Recall that if $\mu$ is a $U$-invariant ergodic probability measure on 
a stratum $\HH$, then $M \in
\LL$ is said to be {\em generic for $\mu$}\index{generic for $\mu$} if for any continuous compactly
supported function $f$ on $\HH$, we have 
\label{eq: generic points}
\begin{equation}
\lim_{T\to\infty} \frac{1}{T} \int_0^T f(u_sM) \, ds =
\int f \, d\mu.
\end{equation}
The assertion that  $M$ is generic for $\mu$ constitutes a
quantitative strengthening of the assertion that $\overline{UM} =
\supp \, \mu$. The following statement is an analog of a
genericity theorem proved by Ratner \cite[Thm. 3.3.10]{KSS handbook}:

\begin{thm}\label{thm: orbit closures}
For  any $M \in \EE_D(1,1)$, there is a measure
$\mu$ described in  Theorem \ref{thm: Calta Wortman revisited} such that $M$
is generic for $\mu$, and belongs to the support of $\mu$. 
\end{thm}

In \S \ref{sec: generic} we deduce Theorem \ref{thm: orbit closures} from a more
explicit result Theorem \ref{thm: generic in orbit closure} which
explains, given $M$, how to detect the measure $\mu$ for 
which $M$ is generic.

Continuing the analogy with related results for homogeneous flows, we prove several equidistribution 
results. We mention 
four results of this type here, and refer the reader to \S
\ref{sec: more} for more results along these lines.

\begin{thm}\label{thm: teaser1}
Let $\mu$ be the length measure on a periodic $U$-orbit in
$\EE_{D}(1,1)$, let $\{g_t\} $ denote the geodesic flow, and suppose $\supp
\, \mu$ is not contained in a closed $\SL_2(\R)$-orbit. Then the
measures $g_{t*} \mu$ converge to the flat measure on $\EE_D(1,1)$ as
$t \to \infty$. 
\end{thm}

This could be viewed as a non-homogeneous counterpart of a theorem that Shah 
proves in a homogeneous setting \cite[Thm. 3.7.6]{KSS handbook}. The following result  also has an
analog in the homogeneous setting (see \cite[\S 3.7]{KSS handbook}):
 
\begin{thm}\label{thm: teaser2}
Let $\mu$ be the $\SL_2(\R)$-invariant measure on a closed $\SL_2(\R)$-orbit in
$\EE_D(1,1)$. Let $t \in \R$ and let $\mu_t$ be the 
measure obtained by applying the real Rel operation to surfaces in
$\supp \, \mu$, with rel parameter $t$. Then as $t \to \infty$ or $t \to -\infty$,
$\mu_t$ converges to the flat measure on $\EE_D(1,1)$. 
\end{thm}

In response to a question of Giovanni Forni, we prove: 
\begin{thm}\label{cor: for Giovanni}
Let $\mu$ be any ergodic $U$-invariant measure on
$\EE_D(1,1)$, then as $t \to +\infty$,  $g_{t*}\mu$
converges to a $\SL_2(\R)$-invariant measure, and as $t \to -\infty$, either $g_{t*}\mu$
converges to a $\SL_2(\R)$-invariant measure or $g_{t*}\mu$ is divergent in
the space of probability measures on $\HH(1,1)$. 
\end{thm}

We also prove an equidistribution result for large circles:

\begin{thm}\label{thm: circle averages}
For any $M \in \EE_D(1,1)$ which is not a lattice surface, 
let $\mu_{t}$ \index{m@$\mu_t$} be the measure on $\EE_D(1,1)$ defined by 
\begin{equation}\label{eq: defn circle averages}{
\int
\varphi \, d\mu_{t}=
\frac1{2\pi}\int_0^{2\pi} \varphi(g_tr_\theta M)\, d\theta, \text{ for
  all } \varphi \in C_c \left( \EE_D(1,1) \right). 
}
\end{equation}
Then $\mu_{t}$ converges to the flat measure on
$\EE_D(1,1)$ as $t \to \infty$. 
\end{thm}

Following a strategy of Eskin and Masur \cite{eskinmasur}, we use this result to
solve a counting problem. 
We obtain: 
\begin{thm}\label{thm: Matt absolved}
For any $M \in \EE_D(1,1)$, the limit 
$$C_D = \lim_{T \to \infty} \frac{\# \{\text{saddle connections on } M
  \text{ of length } \leq T\}}{T^2}$$
exists. 
\end{thm}
Theorem \ref{thm: Matt absolved} was proved in \cite{EMS} for the case
that $D$ is a square, and in \cite{Matt} for the case that $D$ is not
a square.   In both of these papers precise formulae for the constants $C_D$ were
given.  In fact, when $D$ is not square, $C_D = 4\pi$, which agrees with the value of this constant
for a generic surface in $\HH(1,1)$. The proof given in \cite{Matt} contained a gap in the case
$D=5$ which our results fill.

\subsection{Guide to the paper} \label{subsec: sections}
In section \S \ref{sec:strata} we define strata of translation surfaces and establish some of their
basic properties. 

In \S \ref{subsec: blowups} we give introduce `blowups of
translation
surfaces' along with their corresponding moduli spaces and mapping
class groups. 
These notions are useful throughout our discussion: for marking
singularities, marking distinguished horizontal prongs, distinguishing
horizontal data diagrams, resolving
orbifold issues for both the stratum and for the structure of
individual rel leaves, and for discussing surgeries involving a
stratum and nearby boundary strata obtained from it as limits of rel
operations. An interesting novelty of our approach is that certain
finite covers $\widehat{G}$ of $\SL_2(\R)$ 
appear as the groups which act
naturally on our covers (see \S \ref{subsec: action}). The finite covers
we consider have arisen in topological contexts (see
\cite{Boissy}), in the study of interval exchange
transformations (see \cite{Yoccoz survey}) and in computations of monodromy representations
(see e.g. \cite{MYZ}), as well as in our forthcoming
\cite{examples}. We believe that the terminology we introduce will be
useful in future work on translation surfaces. 

In the Sections \S\ref{section:
  rel} we describe certain natural vector fields on strata arising from relative cohomology which we
  call $\rel$ vector fields.
In subsection \S\ref{subsec: real rel} we describe certain subset of
Rel fields which we call real Rel vector fields 
and show that in the appropriate sense these commute with the horocycle flow. 
 In \S \ref{subsec: group of a measure} we make a connection between
 real Rel and stabilizers of measures. We also describe centralizers
 and normalizers of the horocycle flow in the context of vector
 fields. 
  
  The 
domains of well-defined motion in the $Z$ and $N$ directions are
invariants of a $U$-invariant ergodic measure $\mu$, which we denote by
$Z^{(\mu)}$ and $N^{(\mu)}$ (see \S \ref{subsec: group of a
  measure}). Moreover within these domains of well-defined motion, are
Lie groups 
$Z_\mu$  and $N_\mu$  which do act on $\supp \, \mu$ and are the
stabilizer subgroups of $\mu$ within the centralizer and normalizer 
(see \S \ref{subsec: group of a measure}). 
The analysis of generic points plays a
major role in the analysis of ergodic measures on homogeneous spaces. 
The Rel surgery depends on a parameter $T$ and 
as remarked above, the Rel surgery is not globally
defined. As a consequence, the set of surfaces for which it is
defined for all $T$ is not locally compact and it is not
a priori clear that the centralizer of the $U$-action should map
$U$-generic points to $U$-generic points. 
We clarify this in \S
\ref{subsec: generic points}.

In \S\ref{sec: separatrix diagram} we introduce invariants of beds
related to horizontal saddle connections. 
In \S \ref{subsec: explicit} we explicitly
identify the domain of definition of real Rel surgeries, continuing
work done in \cite{MW, McMullen-twists, Matt}.

In  \S\ref{section:
  eigenform locus}--\S\ref{sec: more}, we give a complete picture of the horocycle dynamics in the genus 2 eigenform loci.

Starting in \S \ref{section: eigenform locus}  we describe the eigenform
loci.

We describe ergodic invariant  measures for the horocycle flow
 in detail in \S \ref{section: construction}. 

In \S\ref{sec: classification} we revisit 
Ratner's argument for transverse divergence of nearby
horocycle orbits and apply it to eigenform loci to characterize invariant measures.

In \S\ref{sec: injectivity} we we develop a linearization technique for strata
to  analyze the behavior of $U$-orbits which are near beds. 
This is an analog of the `linearization technique' developed by Dani
and Margulis for homogeneous spaces.

In \S \ref{sec: generic},  we use ideas from the previous section to classify all orbit-closures for the
$U$-action on $\EE_D(1,1)$ and prove that every orbit is
equidistributed in its closure. 

In \S\ref{sec: more} we prove Theorem \ref{thm: circle averages} and equidistribution results
for several other naturally occurring sequences of measures  

The authors are grateful to Uri Bader, Elon Lindenstrauss, Yair Minsky and Saul Schleimer  for
useful conversations. 

\section{Strata}
\label{sec:strata}
In this section we define our objects of study: translation surfaces,
moduli spaces of translation surfaces and dynamics on moduli
spaces. For background and alternate treatments we
refer the reader to \cite{EMZ, Forni-Matheus, MT, MS,Wright_survey, zorich survey}. Our treatment
expands on previous work by discussing in detail
blowups of translation surfaces and associated orbifold covering spaces of
strata, which are needed in later sections. Our discussion also
has interesting consequences of independent interest about fundamental
groups of strata (see \S \ref{subsec: not simply connected}). 

\subsection{Translation surfaces}\label{subsec: translation surfaces}
A {\em translation surface,} or a {\em surface with translation
  structure} can be defined in several equivalent ways. 
We will
describe it by gluing polygons, in terms of an atlas using the
language of $(G,X)$ structures, or as a holomorphic differential. 

Let $M$ be a surface obtained from a finite collection of polygons in 
$\R^2$ which are glued together by isometries of the edges which are
restrictions of translations. Each polygon edge receives an
orientation when viewed as part of the boundary of a polygon and we
assume that the edges of the polygons are partitioned into pairs of
edges which are parallel, have the same length, and opposite
orientations, and these are identified by the gluing isometries.
This construction typically produces a finite set of {\em
  singular points} corresponding to the vertices of the polygons at which the cone
angle is larger than $2\pi$. The cone angle at a singularity is $2\pi(n+1)$ for some natural number $n$.
We call $n$ the order of the singularity. In addition to the singular
points that we have defined it is often useful to mark a finite set of
points at which the cone angles are $2\pi$. Let $\Sigma$ be a finite
subset of $M$ which contains all of the singular points and possibly
some additional `marked' points.

A translation structure on a surface determines an atlas of charts for
the surface $M \sm\Sigma$ taking values in $\R^2$ where the transition
maps are restrictions of translations.
  A translation structure can be determined by  
specifying this atlas. If we have a space $X$ equipped with an action
of a Lie group $G$ a {\em $(G,X)$-structure} \index{G,X@ $(G,X)$-structure}
is an
atlas of charts with overlap functions in $G$. Thus a translation
surface produces a $(G,X)=(\R^2,\R^2)$ structure on $M \sm  \Sigma$
where the first $\R^2$ represents a Lie group acting on the second $\R^2$ by
translation. See \cite{Thurston1} for more information about $(G,X)$-structures.

We can use the atlas of charts
on $M$ to define geometric structures on $M$ which are naturally
associated with the translation structure. 
 Since the one-forms $dx$ and $dy$ are translation invariant on $\R^2$ these 
charts allow us to build globally defined one-forms $dx$ and $dy$
on $M$. Similarly we can use the planar charts to define a metric, an area form, and an
orientation on $M$. A {\em saddle connection} on $M$ is a path with
endpoints in $\Sigma$, which is a straight line in each chart of this
atlas, and does not contain singularities in its interior. 
The one-forms $dx$ and $dy$ are closed and represent
cohomology classes in $H^1(M,\Sigma;\R)$. 
For an oriented path $\gamma$ connecting points in $\Sigma$  write
$\hol(M,\gamma)$\index{hol} \index{hol$(M, \gamma)$} for  
$\left(\int_\gamma dx, \int_\gamma dy \right)$. We can think of $\hol(M,\cdot)$ as
giving a homomorphism from 
$\pi_1(M \sm
  \Sigma)$ to $\R^2$ or as determining an element of 
$H^1(M,\Sigma;\R^2)$. This is the {\em holonomy homomorphism}.
If $\til M$ is the universal cover of $M \sm
  \Sigma$ then
a map from $\til M$ to $\R^2$ with derivative equal to the identity is
a {\em developing map}. Developing maps exist and are unique up to
translation. The developing map 
is equivariant with respect to the holonomy homomorphism. This is an
instance of the general principle that for any $(G,X)$-structure on a
manifold $M$, there is a developing map
$\dev:\til M\to X$ and a holonomy homomorphism $\pi_1(M)\to
G$.

The atlas for a translation surface gives a trivialization of the
tangent bundle at non-singular points.  Let us say that $f:M\to N$ is
a smooth map between translation  
surfaces if it 
preserves singular sets,
and is smooth away from the singular set of $M$.  
  Using the
trivialization of the tangent space, we can view the derivative of a smooth map $f$
as a $2\times 2$ real matrix-valued function.  We say that a smooth
map $f$ is a {\em translation 
  equivalence} if it is a homeomorphism and its derivative is the identity matrix. 
The notion of
translation equivalence gives a natural equivalence relation on
translation surfaces. 
We say that $M$ and $N$ are {\em affinely equivalent} if there is an
orientation-preserving 
smooth map $f$ between them which is a homeomorphism and for which
$Df$ is constant but not necessarily equal to 
the identity. The affine equivalence classes are orbits of a
$\GL^\circ_2(\R)$-action we will discuss shortly (where
$\GL^\circ_2(\R)$ is the group of orientation-preserving linear
automorphisms of $\R^2$).

If we identify $\R^2$ with $\C$ then the coordinate charts of a
translation structure induce a conformal 
structure on $M \sm \Sigma$. This conformal structure extends to $M$ so that the points
in $\Sigma$ correspond to punctures. We can define the complex-valued one-form
$dz=dx+i\, dy$. This is a holomorphic one-form or Abelian differential
with zeros at the singular points where the 
order of the zero is the order of the singular point. 
If we are given a surface with a conformal structure $X$ and we wish
to define a compatible translation surface then this is determined by
a nonzero holomorphic one-form $\omega$. We use the notation $(X,\omega)$ to
denote this pair. For the relations between these three points
of view, see \cite{zorich survey} and the references cited there.

\subsection{Strata as sets} \label{subsec: strata}
Strata are moduli spaces of translation
surfaces. In this section we describe the set of surfaces 
which comprise a given stratum and the relevant notion of
equivalence. In section \S\ref{subsection: structures} we will  
show that the set of these equivalence classes of surfaces can be given natural orbifold structures. Let $k \in \N$ and let $a_1,\dots, a_k$ be a sequence of
non-negative integers. 
Let $M$ be a translation surface and 
let $\Sigma$ be a finite subset of $M$ which consists of $k$ points
labeled $\xi_1,\dots,\xi_k$ and contains all singular points of
$M$. (Note that our conventions imply that $\Sigma$ has at least one point.)
We say that $M$ is a translation surface of type $(a_1,\dots,
a_k)$ if the cone angle  
at $\xi_j$ is $2\pi(a_j+1)$. We refer to $a_j$ as the order of $\xi_j$.
The points $\xi_j$ with order zero are not singular and we will refer to them as marked points. We want 
to construct a version of the stratum in 
which singular points have well-defined labels. 
To this end we say that two surfaces of type $(a_1,\dots, a_k)$ are label preserving
translation equivalent if 
they are translation equivalent by means of a translation equivalence
preserving the labelling of the singular points. (When no confusion will
result we will drop the expression `label preserving'.)

Let $\HH=\HH(a_1,\dots, a_k)$ \index{H@$\HH(a_1,\dots, a_k)$} denote the {\em stratum} of translation surfaces
of type $(a_1,\dots, a_k)$, considered up to label preserving translation
equivalence. The surfaces in $\HH$ have genus $g$ where 
$2g-2=\sum a_j$. It is sometimes convenient to denote this stratum as
 $\HH_g(a_1,\dots, a_k)$. In particular note that $g\ge 1$ since the $a_j$ are non-negative. 

\subsection{Strata of marked surfaces} \label{subsec: marked strata}
Let $S$ be an oriented surface of genus $g$ and let $\Sigma$ be a
subset consisting of $k$ points labeled $\xi_1,\dots,\xi_k$. The pair $(S,
\Sigma)$ will serve as a topological model for a translation surface. 
A {\em marked translation surface} 
is a translation surface $M$ of type $(a_1,\dots, a_k)$,
 equipped with 
an orientation-preserving homeomorphism $f: S \to M$ where $\Sigma$ maps to the
appropriate distinguished points of the translation structure, respecting
the labels. We will consider $S$ as fixed and sometimes write maps $f: S \to M$ as pairs
$(f,M)$. Two marked translation surfaces
$(f_1, M_1)$ and $(f_2 , M_2)$ are considered to be equivalent as marked translation surfaces if
there is a label preserving translation equivalence $\varphi: M_1\to M_2$ so that $\varphi \circ f_1$ and
$f_2$ are isotopic via an 
isotopy that fixes the points in $\Sigma$.

Let $\HH_{\mathrm{m}} =  \HH_{\mathrm{m}}(a_1, \ldots, a_k)$
\index{H@$\HH_{\mathrm{m}}(a_1, \ldots, a_k)$ } 
 denote the set of marked translation surfaces, up to
the equivalence relation described above.  As discussed above, a surface $M\in\HH$ determines a cohomology
class 
$\hol(M,\cdot)\in H^1(M,\Sigma;\R^2).$
A marked translation surface $f:S\to M$ in $\HH_{\mathrm{m}}$ determines a pullback cohomology class 
$f^*(\hol(M,\cdot))\in H^1(S,\Sigma;\R^2)$ where
$f^*(\hol(M,\gamma))=\hol(M,f_*(\gamma))$.
 It is not hard to check
that the pullback depends only on the equivalence class of $f$,
i.e. if $(f_1, M_1)$ and $(f_2, M_2)$ are equivalent then the
cohomology classes $f^*_i(\hol(M_i, \cdot)), \ i=1,2$ are the same. 
We denote by 
\begin{equation}\label{eq: dev}{
\dev: \HH_{\mathrm{m}} \to H^1(S, \Sigma ; \R^2)
}
\end{equation}
the map which takes a marked translation surface to the corresponding element of  
$H^1(S,\Sigma;\R^2)$.  This map  is often 
called the {\em period map}.

\subsection{Strata as spaces} 
\label{subsection: structures}
We now define a 
topology on $\HH_{\mathrm{m}}$ using an atlas
of charts which is defined via triangulations of
translation surfaces. 
Let $\tau$ be a triangulation of $S$ so that the vertices are points in $\Sigma$.
We do not require that edges have distinct endpoints in
$S$ though we do require that they have distinct endpoints in the
universal cover of $S$ or equivalently that no two edges are homotopic relative
to their endpoints.
This triangulation gives $S$ the structure of
a $\Delta$-complex (for a definition see \cite[p. 102]{Hatcher}). 
Let $\mathcal{U}_\tau \subset \HH_{\mathrm{m}}$ be the set of marked
translation surfaces 
containing a representative $f: S \to M$ which takes the edges of $\tau$ to saddle
connections in $M$. 

The function that maps an oriented edge to its holonomy vector is a
1-cochain in the cochain complex  
associated to the triangulation $\tau$. The condition that the sums
of vectors on the boundary of  
a triangle is zero means that this 1-cochain is a cocycle. Using the
fact that $\tau$ is a $\Delta$-complex  we can identify the space of 
such cocycles with $H^1(S,\Sigma;\R^2)$ (see \cite{Hatcher}).

Using the edges of the triangulation $\tau$ allows us to
define {\em comparison maps} between two marked translation
surfaces lying in a given chart $\mathcal{U}_{\tau}$. Namely, suppose
$(f,M)$ and $(f', M')$ are two representatives of marked translation 
surfaces in the same $\mathcal{U}_\tau$. Then $f' \circ f^{-1}: M \to
M'$ is a homeomorphism but it is only well defined up to isotopy. 
We define a map $F$ in this isotopy class by requiring that for each
triangle $\Delta$ in $\tau$, $F$ is an 
affine map when restricted to $f(\Delta)$. This requirement and the
triangulation $\tau$ determine 
$F=F(\tau, M, M')$, up to isotopy once isotopy classes of $f$ and $f'$ are fixed, and we call it the {\em comparison map} between $(f,M)$ and $(f',M')$.

\begin{prop} \label{prop: injective} The map 
$\dev|_{\mathcal{U}_{\tau}}:\mathcal{U}_\tau \to H^1(S,\Sigma;\R^2)$
is injective and its image is open.
\end{prop}

\begin{proof} Suppose $(f,M)$ and $(f', M')$ are in $\mathcal{U}_\tau$ and have the
same image under $\dev$. A priori the comparison map is affine on each triangle. In this case the comparison map $F(\tau, M, M'): M\to M'$ is a translation  
equivalence since the holonomy of corresponding edges are equal. 
 We conclude that $(f,M)$ and $(f', M')$ are 
equivalent and the developing map is injective on $\mathcal{U}_\tau$.

 A cohomology class $\phi$  is  in the image of
$\dev|_{\mathcal{U}_\tau}$ if the values of $\phi$ on the
sides of every triangle $\Delta$ in $\tau$ correspond to the
coordinates of a non-degenerate 
triangle in $\R^2$ with the appropriate orientation --- an open
condition in $H^1(S,\Sigma;\R^2)$. If this condition
is satisfied then an appropriate 
translation surface can be built by gluing together triangles in
$\R^2$ with edge coordinates given by $\phi$. In particular the
image 
$\dev(\mathcal{U}_\tau)$ is open in $H^1(S, \Sigma ; \R^2)$. 
\end{proof}

It was
shown in \cite{MS} that every translation surface $M$ with $\Sigma$ nonempty  admits a 
triangulation of this type. 
Thus the charts $\mathcal{U}_\tau$ cover
$\HH_{\mathrm{m}}$ as $\tau$ ranges over triangulations of $S$. 
The change of coordinate maps for this system of charts are 
linear. These charts on the sets $\mathcal{U}_{\tau}$ give
$\HH_{\mathrm{m}}$ an affine manifold structure. This affine manifold
structure can 
also be discussed using the terminology of
$(G,X)$-structures.  Specifically we can take $X$ to be $H^1(S, \Sigma; \R^2)$ and the
structure group $G$ can be taken to be the group of linear automorphisms of
$H^1(S, \Sigma; \R^2)$.  Our choice of dev as
notation for the map in \eqref{eq: dev} is motivated by the fact that
this map is the 
developing map for the affine structure. 

We have described triangulations of $(S,\Sigma)$ in terms of a homeomorphism
with a $\Delta$-complex. We note that such a homeomorphism is determined
up to isotopy by knowing the relative homotopy classes of the edges of the
triangulation (see \cite[Lemma 2.9]{FarbMargalit}). Furthermore distinct edges
are not homotopic to each other relative to their endpoints. In the sequel we
will use $\tau$ to denote the homotopy classes of edges of a triangulation.

We now make use of the marked stratum to put a topology on the stratum.
Consider the group of isotopy classes of orientation-preserving
homeomorphisms of $S$ fixing $\Sigma$ pointwise. This is sometimes
called the pure mapping class group (see 
e.g. \cite{Ivanov}). We will simply refer to it as the mapping class group and
we will denote it by $\Mod(S, \Sigma)$. Up to the action of
$\Mod(S, \Sigma)$ there are finitely many charts $\mathcal{U}_\tau$. 
It can be checked that the $\Mod(S, \Sigma)$-action on
$\HH_{\mathrm{m}}$ is properly discontinuous. This equips the quotient
$\HH = \HH_{\mathrm{m}} /\Mod(S, \Sigma)$ with the structure of an
affine orbifold, with respect to which the map $\HH_{\mathrm{m}} \to
\HH$ is an orbifold covering map (see \cite{Thurston1} for the definition and basic 
properties of properly discontinuous actions and  orbifolds). 

\subsection{The area form and the area one locus} 
Let $\til \HH^{(1)}$ and $\HH^{(1)}$
denote the subset of area-one surfaces in $\til\HH$ and $\HH$ respectively. With respect to the
charts afforded by the map $\dev$, $\til \HH^{(1)}$ is a
submanifold cut out by a quadratic equation. 

We can give these quadratic equations explicitly.
Identify the coefficients $\R^2$ with $\C$ and define a Hermitian form
on $H^1(S, \Sigma; \R^2)$ by 
\begin{equation}\label{eq: hermitian}{
\langle
\alpha,\beta\rangle=\frac{1}{2i}\int_M \alpha\wedge 
\bar\beta.}
\end{equation}
The area of $M$ is $\langle \omega,\omega\rangle$,
 where $\omega=\dev(M)$.

 We now explain how this Hermitian form can be obtained in 
 more topological terms, from the cup product and a
 particular choice of coefficient pairing. 
First note that if we take $z$
   and $w$ to represent two sides of a triangle, then the signed area
   of the triangle is equal to $\Re(z\bar w/2i)$. 
Now recall that the cup product of two
 $\R$-valued simplicial cochains is  
 defined on a simplex by taking the product of  
 values of the cochains on the simplex, and is then extended by
 linearity to chains. If we have coefficients which are not in $\R$,
 we can replace the operation of multiplying values of cochains by a
 bilinear pairing of the values of 
 cochains. Motivated by the above observation, we will use the
 coefficient pairing 
\begin{equation}\label{eq: specific pairing}{
\C\times\C \ni (z, w) \mapsto z\bar w/2i\in\C.}
\end{equation}
If we interpret the 1-forms in \eqref{eq: hermitian} as complex
 valued cohomology classes, interpret integration as
 evaluation on the fundamental class of $M$, and use evaluation on the
 fundamental class to identify $H^2(S, \Sigma; \C)$ with $\C$, then
 our Hermitian form 
is the cup product 
$$H^1(S, \Sigma; \C)\otimes H^1(S, \Sigma; \C)\to H^2(S, \Sigma; \C)
\cong \C,$$
and the particular choice of coefficient pairing 
\eqref{eq: specific pairing} is
 responsible for the connection with the
 area of translation surfaces. 

Since the Hermitian form on $\til\HH$ was defined purely
topologically, it is preserved by the $\Mod(S,\Sigma)$-action. 
Thus $\til \HH^{(1)}$ and $\HH^{(1)}$ have $(G,X)$-structures where
$X$ is a quadric in 
$H^1(S, \Sigma; \C)$ and $G$ can be taken to be the subgroup of the
general linear group which preserves this quadric.

\section{Blowups of translation surfaces}\label{subsec: blowups}
We now discuss a `blowup' construction that replaces a singularity 
$\xi$ by a  boundary circle. 
This is a special case of a more
general construction of a `real oriented blowup' (see e.g. \cite{HPV}).
The real oriented blowup of a point $\xi$ in $\R^2$ is a new space
$\Blo_\xi(\R^2)$ \index{Blo@$\Blo_\xi(\R^2)$} together with a collapsing map $c:
\Blo_\xi(\R^2)\to\R^2$ with the property that the inverse image of any
point other than $\xi$ is a single point while the inverse image of
$\xi$ is the circle of directions at $\xi$ which we can identify with
$(\R^2 \sm \{0\})/\R^+$ (or with the unit circle) and denote by $S^1$.
The space $\Blo_\xi(\R^2)$ has the property that a smooth path in
$\R^2$ landing at $\xi$ and with non-zero derivative at $\xi$ has a 
lift to the space $\Blo_\xi(\R^2)$ which takes
the endpoint of the path to a point in the circle of directions.

If we blow up the vertex $\psi$ of a polygon $P$ in $\R^2$ then we
obtain the space $\Blo_\psi(P)$ which is the result of replacing the
vertex $\psi$ in $P$ by an interval.  
We describe this construction explicitly. Applying a translation,
assume $\psi$ is at the origin. 
Say that the two edges of the polygon incident to $\psi$ are in directions $\theta_1,
\theta_2$ with $0< \theta_2 - \theta_1 <2\pi$, and there is $\vre>0$  such that  
$$\{r(\cos \theta, \sin \theta) : \theta \in [\theta_1, \theta_2], \, 
r \in [0, \vre)\}$$ 
parametrizes a neighborhood of $\psi$ in $P$. 
The blowup of this neighborhood in $\Blo_\psi(P)$ corresponds to the
rectangle 
$$\{(r,\theta): \theta \in [\theta_1, \theta_2], \,  
r \in [0, \vre]\}$$ 
and the collapsing map $c$ from $\Blo_\psi(P)$ to
$P$ takes $(r,\theta)$ to $r(\cos \theta, \sin \theta)$. In particular
the interval $\{(0,\theta): \theta \in [\theta_1, \theta_2]\}$ is
collapsed to the point $\psi$. Note that this interval has a natural `angular
coordinate' with values in the unit circle.

Given a translation surface $M$ with singular points $\Sigma$ we
construct a surface with boundary $\check{M}$ by blowing up all of the
points of $\Sigma$. We can do this as follows. Choose a triangulation
of $M$. Blow up each vertex of each triangle thereby creating a family
of hexagons where each hexagon contains edges of two types: those
corresponding to edges of triangles and those corresponding to
vertices of triangles.  Glue the (triangle edge) sides of the hexagons
together according to the gluing pattern of the original
triangles. The result is the surface $\check{M}$ which is in fact
independent of the particular choice of triangulation. At a singular
point $\xi_j$ of the surface the intervals mapping to $\xi_j$ glue
together to form a circle which we 
denote $\partial_j \check M$. The angular
coordinates glue together to give us a map $p_j:\partial_j \check M\to S^1$. 
The total angular measure of $\partial_j \check M$ is $2\pi(a_j+1)$ which is
the cone angle at $\xi_j$. We can choose an identification of the
circle $\partial_j \check M$ with the circle $\R/(2\pi(a_j+1))\Z$ so that the
angular coordinate of a point is equal to its circle coordinate modulo
$2\pi$.  This identification of $\partial_j \check M$ with the circle is well-defined
up to translation by $2\pi$ while the map $p_j$ is defined
independently of any choices.  

If we associate a point $\nu$ in the $j$-th boundary component
$\partial_j \check M$ with a
short ray heading away from $\xi_j$, then $p_j(\nu)$ is the direction
of the ray.  We define the {\em prongs} to be points on boundary circles corresponding
to  horizontal rays, i.e. points whose angular parameter is $\pi k$
with $k \in \Z$. We will call a prong {\em right-} or {\em
  left-pointing}, if $k$ is even (resp. odd), that is, 
according to the orientation that the prong
inherits from the plane. 
A particular choice of an identification of $\partial_j \check M$ with $\R/2\pi(a_j+1)\Z$ is
equivalent to the choice of a right-pointing prong. The boundary circles of $\check{M}$ inherit
boundary orientations as boundary components of the oriented manifold
$M$. 
With respect to the boundary orientation the maps $p_j$ are covering maps of degree 
$-(a_j+1)$ where the
negative sign reflects the fact that moving in the direction of the boundary
orientation corresponds to {\em decreasing} the angular coordinate.  

\subsection{Strata of boundary marked surfaces}\label{extended map}
In this subsection we define a notion of marked surface appropriate to
surfaces with boundary 
and a corresponding mapping class group.
To this end we 
 define a `model surface' $\check{S}$, \index{S@$\check{S}$} which will capture some of
 the structure common to the surfaces $\check{M}$  for
 $M\in\HH(a_1, \ldots,  a_k)$. Let $\check{S}$ be a surface with boundary
 which has genus $g$ where $2g-2=\sum a_j$ and $k$ boundary components
 labeled $\partial_1\check{S}, \ldots, \partial_k\check{S}$.  
 The circles inherit a boundary
orientation from $S$. 
We equip each boundary circle $\partial_j\check{S}$ with orientation reversing
homeomorphisms $q_j:\partial_j\check{S}\to\R/(2\pi(a_j+1)\Z)$ which
give angular coordinates on the boundaries. 
 A {\em marked
  translation surface rel boundary} is a surface $\check{M}$ which is
a blowup of a translation surface $M$ of
type $(a_1, \ldots, a_k)$, equipped with an orientation preserving homeomorphism $\check{f}:
\check{S} \to \check{M}$ respecting the labels, and such that on each boundary
circle $\partial_j\check{S}$ we have $p_j \circ
\check{f} \equiv  q_j\bmod 2\pi$. 
Note that a boundary marking of $M$ induces an explicit coordinate on
$\partial\check{M}$ 
and an explicit choice of a prong, namely the image under $\check f$
of the prong on $\partial_j \check S$ corresponding to angular
parameter zero.

We say that two boundary marked translation surfaces
rel boundary $(\check{f}_1, \check{M}_1)$ and $(\check{f}_2,
\check{M}_2)$ are  {\em equivalent} if there 
is a translation equivalence $g: M_1 \to M_2$ such that $\check{g} \circ \check{f}_1$ 
and $\check{f}_2$ are equal on $\partial\check{S}$ and are isotopic via an isotopy 
that fixes the boundary. Let $\til \HH=\til \HH(a_1,\dots, a_k)$ \index{H@$\til\HH$}  denote the
set of boundary marked translation surfaces, up to 
equivalence. We call $\til \HH$ a stratum of boundary marked translation surfaces.

There is a natural collapsing map $c:\check{M}\to M$ which collapses each boundary
component to a single point.
A map $\check{f}: \check{S} \to \check{M}$ induces a map $f: S\to M$,
where $c\circ\check{f}=f\circ c$.
If $\check{f}_1$ and $\check{f}_2$ are equivalent (in the sense of \S
\ref{extended map}), then so are $f_1$ and $f_2$ (in the sense of \S
\ref{subsec: strata}). We
say that $f$ is {\em obtained from $\check{f}$ by projection.}
The forgetful map which takes $(\check{f},\check{M})$ to the pair $(f,M)$ where $f$
is obtained by projection gives a map $\pr:\til\HH\to
\HH_{\mathrm{m}}$. 

We now define a {\em  mapping class group rel boundary}. \index{mapping class group rel boundary}
We say a homeomorphism
$\check f:\check{S}\to\check{S}$ is {\em admissible} if $\check f$ takes each boundary
circle $\partial_j\check{S}$ to itself and the restriction of $\check f$ to
$\partial_j\check{S}$ is a  
rotation by a multiple of $2\pi$ with respect to the circle coordinate
given by $q_j$. We say that two  
admissible homeomorphisms are {\em isotopic rel boundary} if they agree on
$\partial\check{S}$ and if they are isotopic 
by an isotopy which is the identity when restricted to $\partial\check{S}$. 
We denote  by $\Mod(\check{S}, a_1,\dots, a_k)$  the group of 
isotopy classes rel boundary of admissible homeomorphisms of $\check{S}$. When
there is no chance of confusion we will abbreviate this to 
$\Mod(\check{S}, \partial \check{S})$. \index{Mod(S)@$\Mod(\check{S}, \partial \check{S})$}
There is a natural right action of the mapping class group
$\Mod(\check{S}, a_1,\dots, a_k)$ on the stratum $\tilde\HH(
a_1,\dots, a_k)$ by precomposition.  

\subsection{Relative homotopy classes of paths}\label{subsec: paths}
A useful tool in the study of mapping class groups is the study of
their action on curves. 
In order to analyze $\til\HH$ and $\Mod(\check{S}, \partial
\check{S})$ we will consider the action on homotopy classes of  
paths with endpoints in the boundary.
We consider paths in $\check S$ with endpoints in
$\partial \check S$. We say that two paths are {\em homotopic} if there is a
homotopy between them that keeps the endpoints in $\partial \check S$. We say
that two homotopic paths are {\em relatively homotopic} if the homotopy can
be chosen to fix the endpoints of the path. The collection of relative
homotopy classes of paths is a natural set to consider but, unlike
homotopy classes of paths in $(M,\Sigma)$, it is not discrete. We
analyze its topology below. 
There is a well-defined action of $\Mod(\check{S}, \partial \check{S})$
on relative homotopy classes of paths.
We say that a path $\alpha$ is {\em peripheral} if both endpoints lie
in the same boundary component 
 and $\alpha$ is homotopic
to a path contained in that boundary component.  Let $\sigma: [0,1]\to\check{S}$ be a
non-peripheral oriented path with $\sigma(0)\in\partial \check S_j$ 
and $\sigma(1)\in\partial\check S_k$. Let
$\mathcal{C}_\sigma(\check{S})$ be the set of relative homotopy
classes of oriented paths in $\check{S}$ which are homotopic to $\sigma$. 

\begin{dfn}\label{def: epsilon close}
Let $\check{\sigma}$ and $\check{\sigma}'$ be elements of
$\mathcal{C}_\sigma(\check{S})$ going from 
$\partial_j\check{S}$ to $\partial_k\check{S}$, and let  $\vre>0$. 
We say that
$\check{\sigma}'$ is 
{\em $\varepsilon$-close to $\check{\sigma}$}\index{e@$\varepsilon$-close} if there are intervals 
$I_1\subset \partial_j\check{S}$ and 
$I_2\subset \partial_k\check{S}$ of length 
$\varepsilon$, containing the endpoints of $\check\sigma$ and
$\check\sigma'$, such that $\check{\sigma}'$ 
is homotopic to $\check{\sigma}$ through a family  
of paths each of which has one endpoint in $I_1$ and one endpoint in
$I_2$.  
\end{dfn}

Sets of $\varepsilon$-close homotopy classes of curves for given
intervals $I_1$ and $I_2$ form the basis for a topology on
$\mathcal{C}_\sigma(\check{S})$. 
We have an endpoint map  $\epsilon:
\mathcal{C}_\sigma(\check{S})\to\partial_j \check
S\times\partial_k\check S$ which takes a relative homotopy class of
curves to its endpoints. With respect to the topology on
$\mathcal{C}_\sigma(\check{S})$ the endpoint map is continuous. 

\begin{lem} \label{lem: twisting} Say that $S$ is a  surface with
  boundary with negative Euler characteristic.   
The endpoint map  $\epsilon:
\mathcal{C}_\sigma(\check{S})\to\partial_j \check
S\times\partial_k\check S$ is a covering map and
$\mathcal{C}_\sigma(\check{S})$ is the universal cover of $\partial_j
\check S\times\partial_k\check S$. 
It follows that the space of relative homotopy classes of paths in
this homotopy class is homeomorphic to the product of the universal
covers $\widetilde{\partial_j\check S}$ and
$\widetilde{\partial_k\check S}$.  
\end{lem}

We will apply this result to the surfaces with boundary $\check S$
arising as `model surfaces' corresponding to some family of translation
surfaces. According to our conventions these surfaces always have
negative Euler characteristic. This result captures the idea that the
distinction between homotopy and relative homotopy for paths is
measured by 
the amount of twisting around each boundary component. 

\begin{proof} Since the Euler characteristic of $\check S$ is negative
  we can give $\check S$ a hyperbolic structure so that the boundaries
  are geodesics. The universal
  cover $\widetilde{\check{S}}$ is 
  isometric to a convex subset of the hyperbolic plane with geodesic
  boundary.   Choose a boundary component $B_1$ of
  $\widetilde{\check{S}}$ corresponding to $\partial_j\check
  S$. Identify $B_1$ with $\widetilde{\partial_j\check S}$.  Choose a
  lift of $\sigma$ to a path $\tilde\sigma$ in
  $\widetilde{\check{S}}$ starting in $B_1$. The other endpoint of
  $\tilde \sigma$ lands in a component $B_2$ which maps to
  $\partial_k\check S$. Identify $B_2$ with
  $\widetilde{\partial_k\check S}$. Since the path is non-peripheral,
  $B_1$ and $B_2$ are distinct. Any path 
homotopic to
  $\sigma$ has a lift to a path from $B_1$ to $B_2$ and this lift is
  unique since the subgroup of the deck group that stabilizes $B_1$ and
  $B_2$ is trivial. This follows from the fact that a hyperbolic
  isometry that fixes four points is the identity.

We get a map from the space of relative homotopy classes of paths
homotopic to $\sigma$ to $B_1\times B_2=\widetilde{\partial_j\check
  S}\times\widetilde{\partial_k\check S}$ as follows. Given a path
homotopic to $\sigma$ we lift to a path from $B_1\times B_2$ and we
associate this path to its endpoints. Given a pair of points
$(p_1,p_2)\in B_1\times B_2$ we associate the projection to
$\check{S}$ of the unique geodesic from $p_1$ to $p_2$. The fact that
any path between $\widetilde{\partial_j\check S}$ and
$\widetilde{\partial_k\check S}$ is relatively homotopic to a unique
geodesic implies that these maps are inverses. 
This map is a continuous bijection with respect to the natural
topology on $\mathcal{C}_\sigma(\check{S})$. 
\end{proof}

We now describe explicitly some elements of $\Mod(\check{S}, \partial \check{S})$
which correspond to partial Dehn twists around boundary components.
Let $A_j$ be an annular neighborhood of $\partial_j\check{S}$ where
we choose coordinates $\{(t,\theta) : t \in [0,1], \,
\theta\in\R/2\pi(a_j+1)\Z \}$. Here $t=0$ corresponds to the boundary circle
$\partial_j\check{S}$, and the 
$\theta$ coordinate of $A_j$ is compatible at $t=0$ with the $\theta$
coordinate of the boundary circle.  
We will define a particular  homeomorphism
$\tau_j\in\Mod(\check{S}, \partial \check{S})$ as follows. 
On $A_j$ we define 
$\tau_j(t,\theta)=(t,\theta+2\pi(1-t ))$ so that $\tau_j$ rotates
$\partial_j\check{S}$ by $2\pi$ and is the identity on the other
boundary of $A_j$.  
We extend $\tau_j$ to a map $\tau_j:\check S\to\check S$ by setting it
to be the identity outside of  $A_j$. The map $\tau_j$ represents 
an element of $\Mod(\check{S}, \partial \check{S})$ which we call a
{\em fractional Dehn twist} by angle $2\pi$. In 
particular $\tau_j^{a_j+1}$ is a full Dehn twist around the boundary curve
$\partial_j\check{S}$. Note that $\tau_j^{a_j+1}$  is a right Dehn
twist and that it also makes 
sense to describe $\tau_j$ as a right fractional Dehn twist (see \cite{FarbMargalit}).
The collapsing map $c: \check{S}
\to S$ induces a map $c_*: \Mod(\check S, \partial\check{S}) \to
\Mod(S, \Sigma)$. Denote by $FT$ the group generated 
by the fractional Dehn twists $\tau_1, \dots, \tau_k$. 
Note that $FT$ depends on $(a_1, \dots , a_k)$.

\begin{lem}
\label{lem: boundary twist seq}
We have a  
short exact sequence
\begin{equation}\label{eq: new se sequence}{
1 \to FT \to \Mod(\check{S}, \partial\check{S}) \xrightarrow{c_*} \Mod(S, \Sigma)
\to 1,
}
\end{equation}
where the group $FT$ is the free Abelian group generated by the $\tau_j$ and is central in
$\Mod(\check{S}, \partial\check{S})$.  
\end{lem}

\begin{proof} We can see from the definition of a fractional
Dehn  twist that an element of $FT$ is isotopic to the identity by an isotopy which
moves points in the boundary of $S$. These isotopies descend to
isotopies of $(S,\Sigma)$. It follows that $FT$ belongs to the
kernel of $c_*$. The fact that the kernel of $c_*$ is exactly $FT$
follows from the arguments used in \cite[Prop. 3.20]{FarbMargalit}.

To see the surjectivity of $c_*$, let $h$ be a homeomorphism of $S$
fixing points of $\Sigma$. The homeomorphism $h$ is isotopic to a
diffeomorphism (see \cite[Thm. 1.13]{FarbMargalit}) 
which, using the properties of the real blowup, has a lift to a
homeomorphism $h'$ from $\check S$ to itself.  
This homeomorphism induces a homeomorphism of $\partial_j \check S$
for each $j$.  By 
applying an isotopy in the annular neighborhood $A_j$ of each $\partial_j S$ as
above, we can replace $h'$ by a map 
$h''$, which is the identity on $A_j$. 
 In particular $h$ and $h''$
represent the same element of $\Mod(S, 
\Sigma)$, and  $h''$ fixes each point of $\partial \check S$. This
shows that the equivalence class of $h$ contains a representative
which is in the image of $c_*$, proving surjectivity.

Since $h''$
is the identity on each $A_j$ and $\tau_j$ is the identity on the
complement of $A_j$ we see that $h''$ and  $\tau_j$ have disjoint support
so they commute which implies that
$FT$ is central. In particular $FT$ is abelian. 

We now show that there are no additional relations between elements of
$FT$. Consider a word in the  
collection of twists  that represents a relation. Since $FT$ is
abelian we can write it as 
$w=\tau_1^{m_1}\cdots \tau_j^{m_j}\cdots \tau_k^{m_k}$. Now since 
$\check S$ has negative Euler characteristic we can find a
non-peripheral path from a boundary component  
$\partial_j\check S$ to itself. The effect of $w$ on this path is to
shift both endpoints by $2\pi m_j$. 
By Lemma \ref{lem: twisting}, since $w$ acts trivially, $m_j=0$. Since
$j$ was arbitrary, $w=1$. 
\end{proof}

\subsection{Strata of boundary marked surfaces as spaces}
\label{subsec: marked blownups} 
In this subsection we define a topology on $\til\HH$ in a manner
somewhat analogous to the  
method used for defining the topology on $\HH_{\mathrm{m}}$. We construct
a cover of $\til\HH$ by sets on which the developing map is an
injection into $H^1(S,\Sigma;\R^2)$  and then  
we use these maps to endow each such set with the topology induced by
the developing map. 
We show that with respect to this topology the map $\pr$ is a covering
map, thus we conclude 
that these charts give not only a topology on $\til\HH$ but a
compatible affine structure. 

To construct these charts fix a point
$(\check{f},\check{M})\in\til\HH$ and a geodesic 
triangulation $\tau$ of $M$. We caution the reader that in contrast to
\S\ref{subsection: structures},  here $\tau$ denotes a geodesic
triangulation of $M$ rather than  
a topological triangulation of $S$. We have canonical lifts of the edges
$\sigma$ of the triangulation $\tau$ 
 to edges $\check{\sigma}$ in $\check{M}$ so that the endpoints of
 $\check{\sigma}$ lie in $\partial\check{M}$.  
Let $\check{\tau}$ be the collection of
paths in $\check{S}$  of the form $\check{f}^{-1}(\check{\sigma})$ for
$\sigma$ an edge of $\tau$.  These paths are embedded and we refer to them as arcs.
These arcs of $\check{\tau}$ decompose $\check{S}$ 
into hexagons where edges of the hexagons consist alternately of arcs
and intervals in boundary circles.

We now define what it means for two hexagon
decompositions to be {\em $\varepsilon$-close}. Firstly we require that the
arcs in the two decompositions are pairwise homotopic. Secondly we require that
these pairs of homotopic arcs are 
$\vre$-close in the sense of Definition \ref{def: epsilon close}. Note
that any triangulation 
of a translation surface has the property that distinct edges lie in
distinct homotopy classes, thus there is no ambiguity in comparing 
homotopy classes of arcs in the two decompositions. 

Given a geodesic
triangulation $\tau$ of $M$, let $\check\tau $ be its pullback under a
marking of blown up translation surfaces $\check S \to \check M$, and 
let $\mathcal{U}_{\check{\tau},\varepsilon}$ consist of $(\check{f},\check{M'})\in\til\HH$
for which there is a geodesic triangulation $\sigma$ of $M'$ which lifts to a hexagon decomposition 
of  $\check\sigma$ of $\check{M}'$ so that the pullback of $\check\sigma$ under $\check{f}'$ is 
$\varepsilon$-close to $\check\tau$.

\begin{lem} \label{lem: injectivity redux} The developing map is
  injective on the set $\mathcal{U}_{\check{\tau},\pi/2}$. 
\end{lem}

\begin{remark}
The developing map for $\til\HH$ is most naturally
defined to take values in $H^1(S, \partial\check{S}; \R^2)$ but the the
collapsing map $c:\check{S}\to S$ induces  a map $c^*:
H^1(S,\Sigma;\R^2)\to H^1(\check{S}, \partial\check{S}; \R^2)$  
which is an isomorphism. In the sequel we will make use of this
isomorphism to identify the two spaces. 
\end{remark}

\begin{proof}
Assume that $(\check{f}_1,\check{M}_1)$ and
$(\check{f}_2,\check{M}_2)$ 
map to the same point in $H^1(S,\Sigma;\R^2)$ and both lie in 
$\mathcal{U}_{\check{\tau},\pi/2}$. Let $\tau_1$ and $\tau_2$ be
triangulations of $M_1$ and 
$M_2$ so that $\check{\tau}_1= f_1^*(\check{\tau_1})$ and
$\check{\tau}_2= {f_2}^*(\check{\tau_2})$ are
$\pi/2$-close to $\check \tau$, and hence $\pi$-close to each
other. Since $M_1$ and  
$M_2$ have geodesic triangulations such that corresponding edges have
the same image in $\R^2$, the comparison 
map $F(\tau, M_1, M_2)$ is a translation equivalence, and we denote it
by $g:M_1\to M_2$. 
In order to show that $(\check{f}_1,\check{M}_1)$ and
$(\check{f}_2,\check{M}_2)$ represent the same element of $\til\HH$
we need to show that $\check{g}\circ \check{f}_1$ and $\check{f}_2$ agree on
$\partial \check S$ and that  they are isotopic via an isotopy which
fixes  $\partial \check S$. 
Both $\check{\tau}_1$ and $\check{\tau}_2$ produce collections
of arcs in $\check{S}$. There is a unique correspondence between these
collections of arcs so that corresponding arcs are homotopic.
We want to show that corresponding arcs are not just homotopic but in
fact relatively homotopic. 

We begin by working with a single edge.
 Let $\sigma$ be an edge of $\check \tau$ and let $\sigma_1$ and $\sigma_2$
 be corresponding oriented edges of $\check{\tau}_1$ and 
$ \check{\tau}_2$. Our first objective is to show that  
$\sigma_1$ and $\sigma_2$ have the same endpoints and are relatively homotopic.
 The relative homotopy classes of $\sigma_1$ and $\sigma_2$ determine
 points $[\sigma_1]$ and $[\sigma_2]$ in
 $\mathcal{C}_\sigma(\check{S})$. It suffices to show that these two
 points are the same. 

Say that $\sigma_1$ and $\sigma_2$ run from $\partial_i\check{S}$
to $\partial_j\check{S}$.
We have the endpoint map $\epsilon:
\mathcal{C}_\sigma(\check{S})\to\partial_i \check
S\times\partial_j\check S$. 
According to Lemma \ref{lem: twisting} this is a covering map.
We also have projection maps $p_k:\partial_k \check S\to
\R/2\pi\Z$  which are covering maps (in either $S$ or $M$). Consider the composition
$\Pi=(p_i\times p_j)\circ\epsilon:\mathcal{C}_\sigma(\check{S})\to
\R/2\pi\Z\times  \R/2\pi\Z$. This is a covering map. 
The oriented segments $\sigma_1$ and $\sigma_2$ have the same
holonomy  so they point in the same direction in $\R^2$ hence
$\Pi([\sigma_1])=\Pi([\sigma_2])$.
 Since both triangulations lie in $\mathcal{U}_{\check{\tau},\pi/2}$
 there are intervals 
 $I\subset \partial_i \check S$ and $J\subset \partial_j \check S$ of
 length $\pi$ and a homotopy 
 from  $\sigma_1$ to $\sigma_2$ for which the endpoints of the paths
 remain in $I$ and $J$. 
 This homotopy gives a path $\rho$ from $[\sigma_1]$ to $[\sigma_2]$ in $
 \mathcal{C}_\sigma(\check{S})$ which projects to  
 $I\times J\subset  \R/2\pi\Z\times  \R/2\pi\Z$ under the covering
 map $\Pi$. 
The image of $\rho$ is a loop, which is contractible
since it is contained in the contractible 
set $I\times J$. Covering theory implies that $\rho$ itself must
be a loop so $[\sigma_1]=[\sigma_2]$.  
It follows that $\check{g}\circ \check{f}_1(\sigma_1)$ and $\check{f}_2(\sigma_2)$
are relatively homotopic. 

Proposition \ref{prop: injective} shows that $g\circ \check{f}_1$ and  
$\check{f}_2$ induce 
isotopic maps on $(S,\Sigma)$.
It remains to show that they agree as maps on $(\check{S},\partial\check{S})$.
According to Lemma \ref{lem: boundary twist seq} this means that there is no boundary twisting but
boundary twisting is detected by the effect on relative homotopy classes of the paths $\sigma$.
We conclude that $g\circ \check{f}_1$ and  $\check{f}_2$ are isotopic relative to their boundaries.
\end{proof}

We can use the injectivity of the developing maps on the sets
$\mathcal{U}_{\check{\tau},\pi/2}$ to define a topology on $\tilde\HH$. 
With respect to this topology the developing map becomes a local
homeomorphism. The next proposition refers to this topology.

\begin{prop} The projection map $\pr$\index{pr@$pr$} from $\til\HH$ to
  $\HH_{\mathrm{m}}$ has the path 
lifting property. \label{prop: lifting}
\end{prop}

Recall that the path lifting property for the map $\pr$ means that if we are given a path $f:I\to \HH_{\mathrm{m}}$ and a lift $\tilde x_0$ of the endpoint 
$f(0)=x_0$ then there is a unique path $\tilde f:I\to \tilde X$  with
$\tilde f(0)=\tilde x_0$ which satisfies $\pr\circ \tilde f=f$ (see
\cite [p. 60]{Hatcher} for more information).

\begin{proof}
It suffices to construct lifts of paths locally. Consider an open set
$\mathcal{U}_{\tau} \subset\HH_{\mathrm{m}}$ corresponding to a triangulation $\tau$ as in \S \ref{subsection:
  structures}.  Let $\phi_t=(f_t,M_t)$ for $t \in [0,1]$ be a path taking
values in $\mathcal{U}_\tau$ and let 
$(\check{f_0},\check{M_0})$ represent a point in $\til \HH$ such that $\pr (\check{f_0},\check{M_0}) =
(f_0, M_0)$. 
We define a path $\check \phi_t = (\check f_t, \check M_t)$ for  $t \in
[0,1]$, satisfying 
\begin{equation}\label{eq: this is what we satisfy}{
\check \phi_0=(\check{f}_0,\check{M}_0) \ \ \mathrm{and} \ 
\pr(\check \phi_t)=\phi_t,} 
\end{equation}
as follows. 

Let  $F(\tau,M_0,M_t):M_0\to M_t$ be the comparison maps defined in
\S\ref{subsection: structures}. These maps are
piecewise linear and hence  they have unique extensions
 to the blowups, that is there are homeomorphisms
$\check{F}(\tau,M_0,M_t):\check{M}_0\to \check{M}_t$ of the 
blown up surfaces,  satisfying 
$$
c \circ \check{F}(\tau,M_0,M_t) = F(\tau,M_0,M_t) \circ c.
$$
We define $\check F_t =
\check{F}(\tau,M_0,M_t)\circ\check{f_0}:\check{S}\to\check{M}_t$, and
denote the restriction of $\check F_t$ to $\partial_j \check S$ by
$\partial_j \check F_t$. 
We would like to set
$\check \phi_t=(\check{F}_t,\check{M_t})$ but there is a
problem in that the maps $\check{F}_t$ need not preserve
the boundary coordinates given by the maps $p_j:\partial_j
\check{M}_t\to S^1$. In other words, it will not generally be the case that
$p_j\circ \partial_j\check{F}_t=p_j$. We claim that there is a unique
(up to isotopy) way to modify 
the maps $\check{F}_t$ by precomposition with a continuous family of
maps of $\check S$, so that they do
satisfy the boundary coordinate condition, and so that \eqref{eq: this
  is what we satisfy} holds. We will do this by precomposing
the maps $\check{F}_t$ with homeomorphisms $H_t:\check{S}\to\check{S}$
supported in a neighborhood of the boundary and then set
$\check{f}_t=\check{F}_t\circ H_t$. 

In order to prove the existence of the required
homeomorphisms $H_t$, we consider the condition that they will
need to satisfy. For each boundary component $\partial_j\check{S}$ the
restriction $\partial_j H_t$ should satisfy
\begin{equation}\label{eq: better commute}{p_j\circ\partial_j\check{F}_t\circ \partial_j
  H_t=p_j {\ \ \rm and\ \ } \partial_j H_0=\mathrm{Id}.} 
\end{equation}
We rewrite this as 
\begin{equation}\label{eq: better commute2}{p_j\circ\partial_j\check{F}_t=p_j\circ (\partial_j H_t)^{-1}.}
\end{equation}
Setting $\ell_{j,t}=(\partial_j H_t)^{-1}$, we see  that $\ell_{j,t} :
\partial_j \check S \to \partial_j \check S$
is a solution to the homotopy lifting problem
$$
p_j\circ\ell_{j,t}=p_j\circ\partial_j\check{F}_t
\ \ \ \mathrm{and} \ \ 
\ell_{j,0}=\mathrm{Id}
,$$
see Figure \ref{fig: lifting problem}. 

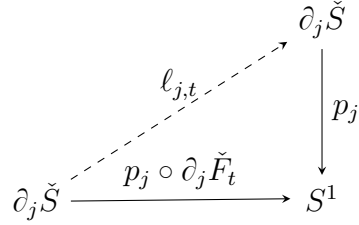
\begin{figure}[h]
\begin{tikzpicture}
  \matrix (m) [matrix of math nodes,row sep=4em,column sep=7em,minimum width=2em] {
      &  \partial_j \check{S}\\
\partial_j \check{S}  &S^1 \\
};
  \path[-stealth][dashed]
    (m-2-1) edge node [above] {$\ell_{j,t}$} (m-1-2);
  \path[-stealth]
    (m-1-2) edge node [right] {$p_j$} (m-2-2)
    (m-2-1) edge node [above] {$p_j\circ\partial_j\check{F}_t$} (m-2-2);
\end{tikzpicture}
\caption[The boundary map]{The boundary map $\ell_{j,t} = (\partial_j H_t)^{-1}$
  rectifies the discrepancies of $\partial_j \check F_t$ along boundary circles.
  } 
\label{fig: lifting problem}
\end{figure}    

Since $p_j$ is a covering map the Homotopy Lifting
Theorem \cite[Prop 1.30]{Hatcher} asserts that lifts $\ell_{j,t}$
exist and are unique. Thus $H_t=(\ell_{j,t})^{-1}$ is defined on the
boundary of $\check{S}$. It remains to extend $H_t$ to annular
neighborhoods of the boundaries. As in \S \ref{extended map}
we have a family of disjoint annuli $A_j$ in $\check{S}$ parametrized
by $\{(r,\theta): r\in [0,1], \, \theta \in \R/2\pi(a_j+1)\Z \}$ where
in these coordinates, $\{r=0\}$ is the $j$-th boundary component of
$\check{S}$. Then we define $H_t$ to be the identity outside the union
of annuli, to be equal to the prescribed map $\partial_j H_t$ on
$\{r=0\}$, and be given by the formula
$H_t(r,\theta)=(r,\psi_{(1-r)t}(\theta))$ for $0\le r\le 1$.

It is clear that with these definitions, the path $\check \phi_t = (\check
f_t, \check M_t)$ with $\check f_t = \check F_t
\circ H_t$ satisfies \eqref{eq: this is what we satisfy}.

It is not hard
to verify the 
homotopy invariance of the lift, that is, if 
we have two homotopic paths $\phi_t$ and
$\phi'_t$ (for $t \in [0,1]$), 
one can 
lift the homotopy to obtain 
a one-parameter family of maps
$\check{f_1}^s:\check{S}\to\check{M_1}$, which 
gives us
an isotopy from $\check{f_1^0}$ to $\check{f_1^1}$, fixing 
the boundary.  That is, the lifts of homotopic paths have the same
endpoint in $\til \HH$. 

The uniqueness (up to isotopy) of the maps $H_t$ (where $H_0$ is the
identity and \eqref{eq: better commute} holds), follows from the
uniqueness of the lifts $\ell_{j,t}$ and standard properties of annuli, and is left as an exercise.
\end{proof}

\begin{cor} The map $\pr:\til\HH\to\HH_{\mathrm{m}}$ is a covering map and $\til\HH$ has an affine structure
given by period coordinates.
\end{cor}

\begin{proof} A local homeomorphism to a sufficiently nice space
which has the path lifting property is a covering map. 
See \cite[p. 383]{do Carmo} for more details and a proof. \end{proof}

The group $\Mod(\check S, \partial \check S)$ acts on $\til \HH$ by
precomposition. It acts
continuously and properly discontinuously on 
$\til \HH$, the quotient is $\HH$ and according to Lemma \ref{lem: boundary twist seq} the
subgroup $FT$ acts simply transitively on each  
fiber of $\pr$. We warn the reader that the spaces $\tilde\HH$ and
$\HH_{\mathrm{m}}$ are not in general connected. 
Typically they have infinitely many components.

\subsection{The space of framed surfaces}\label{framing}
\label{subsec: framings} 

Our next objective is to define and analyze the covering space of framed
surfaces.
Since $\HH_{\mathrm{m}}$ is an affine manifold and $\pr: \til \HH \to
\HH_{\mathrm{m}}$ is a covering map, we have equipped $\til \HH$ with
the structure of an affine manifold. Since the action of  $\Mod(\check
S, \partial \check S)$ is properly discontinuous, for each
subgroup $\Gamma$ of $\Mod(\check S, \partial \check S)$, we can form
the quotient $\til \HH /\Gamma$.
By Lemma \ref{lem: boundary twist 
  seq} we have $\HH =\til \HH/\Mod(\check S, \partial \check S )$ 
and $\HH_{\mathrm{m}} = \HH/FT$. Moreover each $\til \HH /\Gamma$
is an orbifold cover of $\HH$. Note
that this is
not  the Galois correspondence relating connected covers to subgroups
of the fundamental group, since 
$\til \HH$ is not connected. Nevertheless we can define the space of
framed surfaces via a group-theoretic approach.

We start with a discussion of subgroups of
$\Mod(\check{S}, \partial\check{S})$. 
While an element of the group $\Mod(\check{S}, \partial\check{S})$ is only defined up 
to isotopy on the interior of $\check{S}$, it is well-defined on the boundary circles and acts 
on each circle $\partial_j \check S$ by rotations which are multiples
of $2\pi$. Let $\Mod(\check{S})$ be the subgroup of
$\Mod(\check{S}, \partial\check{S})$ 
represented by homeomorphisms that fix the boundary pointwise.
Let $PR$ be the {\em prong rotation group},
consisting of homeomorphisms of the boundary $\partial\check{S}$ which, on each boundary component $\partial_j\check{S}$, 
are rotations by an integral multiple of $2 \pi$. Since $\partial_j
\check S$ is
parametrized by a circle of length $2\pi (a_j+1)$, as a group $PR$ is
isomorphic to $\prod_{j=1}^{k} \Z/(a_j+1)\Z$. 
We have a short exact sequence
\begin{equation}\label{eq: another se sequence}{
1 \to \Mod(\check{S}) \to 
\Mod(\check{S}, \partial\check{S}) \to PR
\to 1. 
}\end{equation}
Surjectivity in \eqref{eq:
  another se sequence} follows from the fact that the fractional
twists 
$\tau_j\in \Mod(\check{S}, \partial\check{S})$ map to a
collection of generators for $PR$.

A {\em framed} translation surface is a translation surface $M$
equipped with a right-pointing horizontal prong at each singular
point. We will call this prong (considered as an element of
$\partial_j \check M$) the {\em selected 
  prong}. Equivalently a framed surface is equipped with a choice of
boundary coordinate on $C_j$ taking values in the circle $\R/2\pi
(a_j+1)\Z$ so that the map 
$p_j$ is reduction modulo $2\pi$ and the selected prong corresponds to the angle 0.

The space of framed translation surfaces is naturally a finite
cover of $\HH$ and we will denote it by $\HH_{\mathrm{f}}$. The
connected components of $\HH_{\mathrm{f}}$ were classified by Boissy
\cite{Boissy}. We recover $\HH_{\mathrm{f}}$
as the quotient of $\til \HH$ by the group $\Mod(\check{S})$ in \eqref{eq: another se
  sequence}:

\begin{prop} \label{prop: recovering framed} 
We have $\til \HH/\Mod(\check S) = \HH_{\mathrm{f}}$.
\end{prop}

\begin{proof} We define a map from $\til\HH$ to $\HH_{\mathrm{f}}$ as
  follows. For each boundary component $\partial_j \check S$, let $0_j$ denote the point in
  $\partial_j \check S$ with angular coordinate $0$. Given a marked blown-up
  surface  $(\check f,\check{M})$, define a framed surface by letting
  the point $\check f(0_j) \in \partial_j \check M$ be the selected
  prong. Say that two surfaces 
$(\check f_1,\check{M_1})$ and $(\check f_2,\check{M_2})$  map to the
same surface in $\HH_{\mathrm{f}}$. Thus there is a translation
equivalence $h:M_1\to M_2$ so that $\check h$ takes selected prongs in
each $\partial_j \check M_1$ to
selected prongs in $\partial_j \check M_2$. This
means that the restriction of the map $(\check f_2 )^{-1}\circ \check
h \circ \check f_1$ to $\partial_j \check{S}$ fixes $0_j$. Since this map is a rotation of the
circle with a fixed point it is the identity map. Thus $(\check
f_2)^{-1} \circ 
\check h \circ \check f_1\in \Mod(\check{S})$ so $(\check f_1, \check M_1)$
and $(\check f_2, \check M_2)$ are $\Mod(\check{S})$-equivalent.  
\end{proof}

 We summarize our constructions in Figure \ref{diagram: spaces}. 
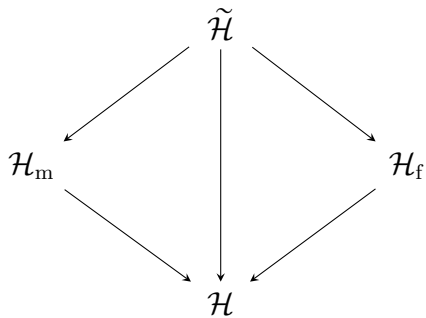
\begin{figure}[h]
\begin{tikzpicture}
  \matrix (m) [matrix of math nodes,row sep=3em,column sep=4em,minimum width=2em] {
    & \til{\HH } &  \\
\HH_{\mathrm{m}} & & \HH_{\mathrm{f}} \\
& \HH & \\};
  \path[-stealth]
    (m-1-2) edge node [left] {} (m-2-1)
            edge node [right] {} (m-2-3)
edge node [right] {} (m-3-2)
    (m-2-1) edge node [below] {} (m-3-2)
    (m-2-3) edge node [right] {} (m-3-2); 
\end{tikzpicture}
\caption[The relationship between covering spaces of a stratum]{The vertical arrow corresponds to factoring by the action of
  $\Mod(\check{S}, \partial \check{S}) $, and the pair of arrows on
  the left and right correspond to factoring by the groups
  appearing in the sequences \eqref{eq: new se sequence} and \eqref{eq:
    another se sequence} respectively. }
\label{diagram: spaces}
\end{figure}

\subsection[Action of $\SL_2(\R)$ on covers]{Action of $G=\SL_2(\R)$ and its covers on 
  $\til\HH$}\label{subsec: action}
 The affine equivalence classes of translation structures are orbits of a group
actions which we now define. 
Recall that $\GL^\circ_2(\R)$\index{GL@$\GL^\circ_2(\R)$} and
$\til{\GL^\circ_2(\R)}$\index{GL2@$\til{\GL^\circ_2(\R)}$} denote respectively the group of orientation
preserving invertible $2 \times 2$ real matrices, and its universal
cover group.  
Given $g\in \GL^\circ_2(\R)$ and a translation surface $M$ we
construct a new surface $gM$ as follows. As discussed in \S\ref{subsec:
  translation surfaces}, in the language of $(G,X)$-structures, a
translation surface can be given by an atlas of charts 
on $M$ with overlap functions taking values in the group of
translations $\R^2$. The element $g$ is a linear map 
$\R^2 \to \R^2$, and postcomposing each chart in an
atlas, we obtain a new translation atlas.  Let $gM$ be this new 
translation surface and let $\phi_g$ be the
identity map from $M$ to $gM$ (the underlying surfaces are the same). The map $\phi_g$ is an affine map with 
derivative $g$.

 The group $g\in \GL^\circ_2(\R)$ acts on $\til\HH_{\mathrm{m}}$ as follows. 
If $(f,M)\in\til\HH_{\mathrm{m}}$ then define $g(f,M)$ to be
$(\phi_g\circ f, gM)$. The condition 
that $g\in \GL^\circ_2(\R)$ insures that $\phi_g\circ f$ is orientation preserving.
Since the action of $\Mod(S,
\Sigma)$ on marked surfaces is by pre-composition, this action induces
a well-defined action on $\HH_{\mathrm{m}}$. 
If we let
$\GL^\circ_2(\R)$ act on $H^1(S, \Sigma ; \R^2) $ by acting on the
coefficients via the linear action on the plane, then
the map $\dev: \HH_{\mathrm{m}} \to H^1(S, \Sigma ; \R^2)$ will be
equivariant.

In the remainder of the paper we write $G$ for $\SL_2(\R)$.\index{G}

It is a general principle that if a  connected topological group acts on a
topological space $X$, then its universal cover acts
on any cover of $X$. Since $\GL_2^\circ(\R)$ and its subgroup $G$ act on
$\HH_{\mathrm{m}}$, and $\til \HH \to \HH_{\mathrm{m}}$ is a covering
map, we conclude that their universal covering groups
$\til{\GL^\circ_2}(\R)$ and $\til G$\index{G@$\til G$} act on $\til 
\HH$. For related discussions in the case of strata of meromorphic
quadratic differentials see \cite{BS} and \cite{HKK}. 

It will be useful for us to not only know that this action exists but
to have an explicit description of  
the action. 
An element of $\til g\in\til{\GL^\circ_2}(\R)$ can be represented by a
pair $(\rho,g)$ where $g\in\GL^\circ_2(\R)$ and 
 $\rho:[0,1]\to\GL^\circ_2(\R)$ is a path with $\rho(0)=\mathrm{Id}$ and
 $\rho(1)=g$. Two such representations 
are equivalent if the corresponding paths are homotopic relative to
their endpoints. Given an element 
of $\til\HH$, $(\check{f},\check{M})$ and an element $\til g=(\rho,g)$
of $\til{\GL^\circ_2}(\R)$ we have a 
path $\alpha(t)=\rho(t)(f,M)$ in $\HH_{\mathrm{m}}$ and a lift of the initial
point of this path $(f,M)$ to $(\check{f},\check{M})\in\til\HH$. 
According to Proposition \ref{prop: lifting} this path lifts uniquely
to a path $\tilde\alpha(t)\in\til\HH$ and we set 
$\tilde g (\check{f},\check{M})=\tilde\alpha(1)$. The resulting element is independent
of the choice of path $\rho$.

Let
\begin{equation}\label{eq: sect2 group elements}
{u_s = \begin{pmatrix}
1&s\\
0&1
\end{pmatrix},
\ 
g_t = \begin{pmatrix}
e^{t/2}&0\\
0&e^{-t/2}
\end{pmatrix} \text{ and }
r_\theta = \begin{pmatrix}
\cos \theta & -\sin \theta \\
\sin \theta & \cos \theta
\end{pmatrix}.
}\end{equation} 
\index{u@$u_s$} \index{g@$g_t$} \index{r@$r_\theta$}
We will refer to the action of $u_s$ as the {\em horocycle flow} and $g_t$ as 
the {\em geodesic} flow and write
$$
U = \{u_s: s \in \R\}, \ \ A = \{g_t: t \in \R\}, \ \SO_2(\R) =
\{r_{\theta} : \theta \in \R\}, \ B=AU.
$$
\index{U@$U$} \index{A@$A$} \index{B@$B$} 
 Note that $B$ is the connected component of the identity in the group of
upper-triangular matrices and it normalizes $U$. 
Since $U$
is simply connected, the connected component of the identity in the
pre-image of $U$ in $\til G$  is a subgroup isomorphic to $U$. Thus we can identify
$U$ with a subgroup of $\til G$, and the same is true for the groups $B$
and $A$. 
The groups $G$ and $\GL_2^\circ(\R)$ are
both homotopy equivalent to $\SO_2(\R) \cong S^1$, and $\til G$ and
$\til{\GL}_2^\circ(\R)$ are both homotopy equivalent to
$\til{\SO}_2(\R) \cong \R$.

As we have seen the fundamental groups of both groups $\til G$ and $\til{\GL}_2^\circ(\R)$ are isomorphic to
$\Z$. Thus for any positive integer $n$, there is a unique connected $n$-fold cover of $G$. Also, since $\SO_2(\R) \subset G \subset \GL_2^\circ(\R)$ and the induced maps on fundamental
groups are isomorphisms it will
cause no confusion to identify the fundamental groups of these three
groups. This group is infinite cyclic and we denote it by
$C$. We have three short exact
sequences: 
\begin{align}\label{eq: short sequence 2}
\notag
1 \to C \to \til{\GL^\circ_2} (\R) & \to \GL_2^\circ(\R) \to 1 \\
1 \to C \to \til G & \to G \to 1 \\
\notag
1 \to C \to \til{\SO_2}(\R) & \to \SO_2(\R) \to 1 
. 
\end{align}

We will write the element of $\til{\SO_2}(\R)$ corresponding to
$\theta \in \R$ as $\tilde{r}_\theta$, so that $\tilde{r}_\theta
\mapsto r_{\theta\bmod 2\pi}$ is  the projection $\til{\SO_2}(\R) \to \SO_2(\R)$. 
The group $C$ is central in  $\til{\GL^\circ_2} (\R)$ and in these
coordinates it is identified with
$\{\tilde{r}_{2\pi n}:n\in\Z\}$. 
 We will explicitly describe the action of
$\til{\SO}_2(\R)$ on $\til \HH$. 
  
\begin{prop}\label{prop: two pi} The left action of $\tilde{r}_{2\pi}$
  on $\til\HH$ is equal to the right action of  $\tau^{-1}$ where
  $\tau=\tau_1\cdots\tau_k\in FT.$
That is, for $(\check{f},\check{M})\in\til\HH$ we have
$\tilde{r}_{2\pi}
(\check{f},\check{M})=(\check{f} \circ \tau^{-1} ,\check{M}).$ 
\end{prop}

\begin{proof} Given a marked translation surface rel boundary
$\check{f}: \check{S} \to \check{M}$ we follow the definition
of the action of $\til{\GL}_2^\circ(\R)$ on $\til\HH$ given above. Let $D_\theta$ be the map 
supported in annuli around boundary components of $\check{S}$ that
performs a Dehn twist by $\theta\in\R$ on each boundary
component. 
Let $\theta_0 \in \R$, and assume for
concreteness that $\theta_0>0$. We will describe the 
action of $\tilde{r}_{\theta_0}$ by explicitly lifting the action of
$\{r_\theta : \theta \in [0, \theta_0]\}$ using the procedure
described in Proposition \ref{prop: lifting}. 

Let $(f, M) \in \HH_{\mathrm{m}}$ be a marked translation surface, and
let $(\check f, \check M)$ be an element of $\til \HH$ projecting to
$(f,M)$. We want to lift the path $\theta \mapsto r_\theta (M, f) =
(f, r_\theta M)$ to $\til \HH$. As in \S\ref{subsec:
  marked blownups}, fix a triangulation $\tau$ of $S$. By a
compactness argument it suffices 
to analyze the lift of the path $\theta \mapsto
r_\theta (M, f),$ for the subset 
$$
\{\theta \in [0, \theta_0]: r_\theta (M,
f) \in \mathcal{U}_{\tau} \},
$$
and we can assume with no loss of generality that $(f,M) \in
\mathcal{U}_\tau$. 

Let $F(\tau, M, r_\theta M)$ be the comparison map defined in \S
\ref{subsection: structures}. An important observation, which follows
immediately from the definition of the comparison maps, is that the
derivative of $F(\tau, M, r_\theta M)$  is everywhere equal to the
matrix $r_\theta$, and hence the comparison map is in fact independent of the
triangulation $\tau$.
Let $\check F(\tau, M, r_\theta M)$ denote the extension of $F(\tau,
M, r_\theta M)$ to blown-up surfaces, and let $\check F_\theta =
\check F(\tau, M, r_\theta M) \circ f$. Then as discussed in \S
\ref{subsec: marked blownups}, $\check F_\theta$ does not preserve
boundary coordinates, but the composition $D_{-\theta} \circ \check
F_\theta$ does. Thus the path $\theta \mapsto (\check{F}_\theta\circ D_{-\theta}
, r_\theta\check{ M})$ satisfies \eqref{eq: this is what we
  satisfy}, so by uniqueness, is the desired lift of the path $(f,
r_\theta M)$ to $\til \HH$. 

In particular, setting $\theta_0=2\pi$ and using the fact that
$r_{\theta_0} = \mathrm{Id}$, we get
$(\check{f} \circ D_{-2\pi}, \check{ M})$ and $D_{-2\pi} = \tau^{-1}$ with $\tau=\tau_1\cdots\tau_k$. 
\end{proof}

\subsection{Fundamental groups of covers of strata} \label{subsec: not simply connected}
We note some consequences of the above discussion. 

\begin{cor}\label{cor: not simply connected} Every path component of
  $\HH_{\mathrm{m}}$ has an infinite  fundamental group. In particular it is never the case
  that the space of marked surfaces is  the universal cover of the stratum.
\end{cor}

\begin{proof} 
Let $\mathcal{C}$ be a path component of $\HH_{\mathrm{m}}$ and let $M
\in \mathcal{C}$. 
Proposition \ref{prop: two pi} shows that lifting the closed path
$\{r_\theta M : \theta \in [0, 2\pi]\}$ to $\til \HH$, we get a
non-closed path whose endpoints differ by an application of the element
  $\tau^{-1}\in FT$. This element has infinite order in $FT$. Since
  the group $FT$ acts freely on the fiber 
  $\pr^{-1}(M)$, none of the lifted paths are closed.  It
  follows that ${\mathcal C}$ has an infinite fundamental group. 
\end{proof}

An instructive example is the case $\HH = \HH(0)$ (where the model surface $S$ is the torus with one
marked point and $\check S$ is the torus with an open disk
removed). The covering space $\HH_{\mathrm{m}}$ can be 
identified with $\GL^\circ_2(\R)$ (the connected component of the
identity in $\GL_2(\R)$), and is not simply connected,
whereas $\til\HH$ is identified with its universal cover group
$\til{\GL^\circ_2}(\R)$ and is simply connected. A generator of the fundamental group of
$\GL^\circ_2(\R)$ acts by a boundary Dehn twist on $\check S$.

We present another useful consequence of our lifting construction. 
Let
$\Gamma_0$ be a subgroup of $\Mod(\check S, \partial \check S)$, 
let $\HH_{\Gamma_0}$ denote the quotient $\til \HH /\Gamma_0$, let
$\mathcal C$ be a path component of $\HH_{\Gamma_0}$ and 
let $p\in\mathcal C$ and $q \in \til \HH$ such that $q$ projects to
$p$. We define a homomorphism $\rho_q: \pi_1(\mathcal
C,p)\to \Gamma_0$ as follows. Let
  $\phi$ be a loop based at $p$. Let $[\phi]$ be the element of
  $\pi_1$ that it represents. We can lift $\phi$ to a path starting at
  $q$. The endpoint of this lifted path maps to $p$ so it has the form
  $p\gamma$ for some $\gamma\in\Gamma_0$ (where our notation reflects
  the fact that $\Mod(\check
  S, \partial \check S)$ acts by precomposition, so defines a
  right-action). Define $\rho_q([\phi])$ to
  be $\gamma$. The homotopy lifting property shows that
  $\rho_q([\phi])$ depends only on $q$ and on the homotopy class $[\phi]$, and not the particular
  loop $\phi$ chosen to represent it. 

The construction of $\rho_q$
  depends on the choice of the point $q$. If we were to choose a
  different point $q'$ mapping to $p$ then $q'=q\alpha$ for some
  $\alpha\in\Gamma_0$. In this case the lift of $\phi$ starting at
  $p\alpha$ is the path $\phi\alpha$ and the other endpoint is
  $p\gamma\alpha=p\alpha(\alpha^{-1}\gamma\alpha)$. Thus
  $\rho_{q'}(\phi)=\alpha^{-1}\rho_q\alpha(\phi)$ so $\rho_{q'}$
  differs from $\rho_q$ by an inner automorphism of $\Gamma_0$. In
  other words we have constructed a preferred homomorphism $\rho :
  \pi_1(\mathcal{C}) \to \Gamma_0$, well-defined up to a choice of an
  inner automorphism of $\Gamma_0$.

Let $n \in \N$ and let 
$C_n$ denote the subgroup of $C$ generated by $r_{2\pi n}$. Then
$C_n$ is central in $G$ and $\widehat{G}_n= \til G/C_n$ is the unique
connected $n$-fold
cover of $G$. 

\begin{cor}
Let $\Gamma_0$ be a subgroup of $\Mod(\check{S}, \partial \check{S})$.
 Then $\widehat{G}_n$ acts on $\til \HH /\Gamma_0$
if and only if $\tau^n\in\Gamma_0$. 
In particular, suppose 
$n$ is the least common multiple of the numbers $\{a_i+1 :  i=1,
\ldots, k\}$. Then $\widehat{G}_n$ acts on $\HH_{\mathrm{f}}$, but
$\widehat{G}_m$ does not act when $m<n$.
\end{cor}

\begin{proof}
Let $\HH_{\Gamma_0} = \til \HH/\Gamma_0$, and let
$(\check{f},\check{M}) \in \til\HH$. 
 If the action of $\tilde r_{2\pi n}$ is well-defined on 
$\HH_{\Gamma_0}$ then $\tilde r_{2\pi n}(\check{f},\check{M})$ and
$(\check{f},\check{M})$ are equivalent in $\HH_{\Gamma_0}$ so by
Proposition \ref{prop: two pi}, 
$$(\check{f}, \check{M} ) \tau^{-n} = \tilde r_{2\pi n}(\check{f},\check{M})=(\check{f},\check{M})\gamma$$
for some $\gamma\in\Gamma_0$. That is, 
$\check{f}\circ \gamma$ and $\check{f}\circ\tau^{-n}$ are isotopic via
an isotopy fixing $\partial \check{S}$. In particular they represent
the same elements in $\Mod(\check S, \partial \check S)$ and 
$\tau^{-n}\in
\Gamma_0$. Conversely if $\tau^{-n}\in \Gamma_0$ then  
$(\check{f},\check{M})$ and  $\tilde r_{2\pi n}(\check{f},\check{M})$
represent the same surface in $\HH_{\Gamma_0}$.
 \end{proof}

\section{The Rel foliation and Rel vectorfields}\label{section: rel}
It is a general principle that geometric structures on $H^1(S,\Sigma;\R^2)$
which are invariant under the action of the mapping class group induce geometric
structures on the stratum. We will now see an example of this
principle. Consider the following  
exact sequence:
\begin{multline}
  \label{eq: exact}
 0 \longrightarrow H^0(S 
) \longrightarrow H^0(\Sigma 
)
 \longrightarrow H^1(S, \Sigma 
) 
  \stackrel{\mathrm{Res}} \longrightarrow H^1(S 
)\longrightarrow 0.
\end{multline}
Let us take the coefficients in cohomology groups to be $\R^2$. 

Let $\mathfrak{R}$\index{R@$\mathfrak{R}$} denote the image of $H^0(\Sigma;\R^2)$ in $H^1(S,\Sigma;\R^2)$. 
We call $\mathfrak{R}$ the rel subspace.
We can identify $\mathfrak{R}$
with the kernel of the restriction map $\Res: H^1(S,\Sigma;\R^2)\to
H^1(S;\R^2)$. We can identify  $H^0(\Sigma;\R^2)$ with the set of
functions from $\Sigma$ to $\R^2$ and we can identify the image 
of $H^0(S; \R^2)$ in $H^0(\Sigma;\R^2)$ with the subspace of constant
functions. Thus $\mathfrak{R}$ can be seen as $\R^2$-valued functions
on $\Sigma$ modulo constant functions. 
If $k$ is the cardinality of $\Sigma$ then the real dimension of $\mathfrak{R}$ is $2(k-1)$. 
Beginning with  \S \ref{section: eigenform locus}  we will work in a stratum for which $\Sigma$
consists of two points, so that $\dim_{\R}\mathfrak{R} =2$.

We will explicitly describe the action of $\mathfrak{R}$ on the period coordinate space
$H^1(S,\Sigma;\R^2)$.  Pick $v\in \mathfrak{R}$ and let $\gamma\in
H^1(S,\Sigma;\R^2)$. We will define $\gamma + v\in
H^1(S,\Sigma;\R^2)$. 
Explicitly, 
the elements $\gamma$ and $\gamma+v$ are determined by their values on oriented 
paths in $S$ with endpoints in $\Sigma$. Let $\sigma$ be one such
oriented path starting at $\xi_i$ and ending at 
$\xi_j$. Since $\mathfrak{R} \cong H^0(\Sigma; \R^2) /H^0(S; \R^2)$,
$v$ is an equivalence class of functions $\til v: \Sigma \to \R^2$,
where functions are equivalent if they differ by a constant. We define $(\gamma + 
v)(\sigma)=\gamma(\sigma)+\til v(\xi_j)-\til v(\xi_i)$. Since
representatives of $v$ differ by constants, the preceding formula does not depend on
the choice of $\til v$. Also $v(\sigma)$ gives the same value for any
$\sigma$ from $\xi_i$ to $\xi_j$. 

The group $G$ acts equivariantly on the terms of the exact sequence \eqref{eq: exact}.
If we think of the terms as vector spaces of $\R^2$-valued
functions then $G$ acts on these functions by acting on their
values. In particular there is a
natural action of $G$ on $\mathfrak{R}$ since it is the quotient of the first
two terms.  

A subspace $W$ of a vector space $V$ defines a linear foliation of $V$ where the leaves are 
the translates of $W$. In this way the subspace $\mathfrak{R}$ defines a foliation of $H^1(S,\Sigma)$. 
Since the mapping class group $\Mod(S, \Sigma)$ preserves the short
exact sequence it preserves this foliation and thus the foliation
descends to a well-defined 
foliation on $\HH$. We call this the $\rel$ foliation\index{rel@$\rel$ foliation}. The names
`kernel foliation' and `absolute period foliation' have also been used
in the study of this foliation, see \cite{zorich survey, schmoll, McMullen-leaves}.

\begin{prop}\label{prop: same area}
Two surfaces in the same {\rm Rel} leaf have the same area.
\end{prop}
\begin{proof}  
As we have seen in \eqref{eq: hermitian} the area of a surface $M$ can
be written as $\langle \omega,
\omega \rangle$ where $\omega = \dev(M)$ and the bilinear form is the
cup product with a certain choice of coefficient pairing.  
So it suffices to show that the cup product of two classes in $H^1(S,\Sigma)$
depends only on the image
of the classes in absolute cohomology $H^1(S)$; the latter statement
follows from the fact that the cup product is natural with respect
to the inclusion $(M,\emptyset)\to (M,\Sigma)$, i.e.  diagram
\eqref{eq: cup} below commutes.
We refer to \cite{Hatcher} for the definition of the cup product in
the relative case, and to \cite[
Prop. 3.10]{Hatcher} for a proof of naturality.
\begin{equation}\label{eq: cup}
\xymatrix{ 
 H^1(S,\Sigma)\ar[d] \times  H^1(S,\Sigma)\ar[r]^-\cup &H^2(S,\Sigma)
 \ar[d]^\simeq \\ 
  H^1(S)\times  H^1(S)\ar[r]^-\cup &H^2(S) 
} 
\end{equation}
\end{proof}

We will define a notion of parallel translation on the leaves
of the Rel foliation.
We can identify the elements of $\mathfrak{R}$ with constant vector
fields on the vector space $H^1(S,\Sigma; \R^2)$. 
Recall that we have insisted on labeling the points in $\Sigma$, and
that the mapping class group fixes $\Sigma$ pointwise. With these
conventions, $\Mod(S, \Sigma)$ acts trivially on $\mathfrak{R}$.
Thus the vector fields corresponding to $\mathfrak{R}$ are invariant
under $\Mod(S, \Sigma)$ and induce well-defined vector fields on 
$\HH_{\mathrm{m}}$ and $\HH$. The leaves of the $\rel$ foliation have natural 
translation structures and these are the coordinate vector fields.  

\subsection{Extending Rel paths}\label{subsec: paths}
The constant vector field associated with $v\in\mathfrak{R}$ can be integrated
on the vector space $H^1(S,\Sigma ; \R^2)$  to give a one-parameter
flow.  Our next 
objective is to lift this flow, to the extent possible, to $\HH_{\mathrm{m}}$.

\begin{dfn} \label{def: rel}
Let $M_0$ be point in $\HH$ and let $v\in\mathfrak{R}$. Let $\vec v$\index{vec@$\vec v$}
denote the rel vector field on $\HH$ corresponding to $v$. We say that
$\rel_v(M_0)$ \index{rel@$\rel_v$} is defined \index{rel is defined} and equal to $M_1$ if there is a smooth path $\phi(t)$
in $\HH$ with $\phi(0)=M_0$, $\frac{d}{dt}\phi(t)=\vec
v(\phi(t))$ and $\phi(1)=M_1$. 
\end{dfn}

 The translation
structures on the leaves of the $\rel$ foliation are not complete in
general and this means that the trajectories of the 
vector fields cannot always be defined for all time.

\begin{prop}\label{prop: continuous} Let $\Omega\subset
  \HH\times\mathfrak{R}$ be the set of pairs $(M,v)$ for which 
$\rel_v(M)$ is defined. Then $\Omega$ is open and the map
$(M,v)\mapsto \rel_v(M)$ is continuous when viewed as a map from
$\Omega$ to $\HH$.
\end{prop}

\begin{proof}This follows from properties of solutions of first order ordinary differential equations.
\end{proof}

 Our next result deals with the interaction between the natural actions
 of $G$ on $\HH$ and $\mathfrak{R}$, and the partially defined maps $\rel_v$.

\begin{prop}\label{prop: distributive} Let $M\in\HH$, $v\in \mathfrak{R}$ and $g\in G$.  If
 $\rel_v(M)$ is defined then $\rel_{gv}(gM)$ is defined and $g(\rel_v(M))=\rel_{gv}(gM)$. 
\end{prop} 

\begin{proof} Let $\vec v$ denote the vector field corresponding to
  $v \in \mathfrak{R}$. To say that $\rel_v(M)$ is defined means that
  there is a smooth path $\phi(t)$ 
with $\phi(0)=M$, $\frac{d}{dt}\phi(t)=\vec v (\phi(t))$, and in this
case $\phi(1)=\rel_v(M)$. Consider the path $t\mapsto g(\phi(t))$.
It has the property that $g(\phi(0))=gM$,
$\frac{d}{dt}g(\phi(t))=g(\frac{d}{dt}\phi(t))=g\vec{v}(\phi(t)) =
\vec{(gv)}(\phi(t))$ and $g(\phi(1))=g(\rel_v(M))$. 
The existence of this path shows that $\rel_{gv}(gM)$ is defined
and $g(\rel_v(M))=\rel_{gv}(gM)$.
\end{proof}

We can think of $\mathfrak{R}$ as a Lie group acting on $H^1(S,
\Sigma; \R^2)$.   The fact that 
we can lift elements of the Lie group action  on
$H^1(S,\Sigma;\R^2)$ to $\HH_{\mathrm{m}}$ does not imply that the relations
in the Lie group necessarily lift.  For example the transformations
$\rel_v$ and $\rel_w$ acting on $H^1(S,\Sigma;\R^2)$ commute but the
corresponding lifted transformations of $\HH_{\mathrm{m}}$ need not commute
where they are defined. 
The following result gives criteria for a composition law and for commutation.

\begin{prop} \label{prop: commuting} Let 
$$\Box = [0,1]^2 \text{ and } \triangle = \{(s,t) \in \R^2: 0\leq  s \leq t
\leq 1\}, $$ and let $v,w\in\mathfrak{R}.$ \index{$\Box$} \index{$\triangle$} Then:
\begin{itemize}
\item[(i)]
If $\rel_{sv+tw}(M)$ is defined 
for all $(s,t) \in \triangle$ then 
\begin{equation}\label{eq: what we want0}{
\rel_v \circ \rel_w(M) = \rel_{v+w}(M).
}\end{equation}
\item[(ii)]
 If
  $\rel_{sv+tw}(M)$ is defined for all $(s, t) \in \Box$ then  
\begin{equation}\label{eq: what we want}{
\rel_v\circ \rel_w(M)=\rel_w\circ \rel_v(M)=\rel_{v+w}(M).
}\end{equation}
\end{itemize}
\end{prop}

\begin{proof} 
It suffices  to prove the result in $\HH_{\mathrm{m}}$, since the maps $\rel_u$
are $\Mod(S, \Sigma)$-equivariant. Note that (ii) follows immediately
from (i), so we prove (i). 
Define 
$$\sigma : \triangle \to H^1(S,
\Sigma; \R^2) \ \text{ by } \ \sigma(s,t)=\dev(M)+sv+tw.$$ 
Recall
  that we have a developing map $\dev:\HH_{\mathrm{m}}\to H^1(S,\Sigma;\R^2)$.
  The developing map is a local homeomorphism. This does not imply
  that paths can be lifted but it does mean that when paths can be
  lifted the lifts are unique (see \cite{Ebeling} for more
  information). The hypothesis that $\rel_{sv+tw}(M)$ is defined for
  all $(s, t) \in \triangle$ 
  means that every path $r\mapsto \sigma(rs,rt)$ for $0\le r\le 1$
  lifts to $ \HH_{\mathrm{m}}$. Let
  $\tilde\sigma(s,t)=\rel_{sv+tw}(M)$ be the lift of $\sigma$ to $
  \HH_{\mathrm{m}}$. Arguing as in Proposition 1.11 in \cite{Ebeling} we see that $\tilde\sigma$ is a continuous map.
By construction 
  $\tilde\sigma(0,1)=\rel_w(M)$ and $\tilde\sigma(1,1)=\rel_{v+w}(M)$. The path $\rho_0(r)=
  \tilde\sigma(r,1)$ satisfies $\rho_0(0)=\rel_w(M)$ and $\rho'_0=v$, so by the definition of Rel,
  $\rho_0(1)=\rel_v(\rel_w(M))$. Also 
  $\rho_0(1)=\tilde\sigma(1,1)=\rel_{v+w}(M)$, and \eqref{eq: what
    we want0} follows.
\end{proof}

In the application of Ratner's techniques to strata in Proposition \ref{prop: CW main} we need to deal with the following situation. We have $M_j\to M^{(1)}$ in $\HH$ and 
$\rel_{v_j}(g_jM_j)\to M^{(2)}$. We know that $g_j\to g_{\infty}$ and that $v_j\to v_{\infty}$ and we would like to conclude that 
$\rel_{v_{\infty}}(g_{\infty}M^{(1)})=M^{(2)}$. This assertion does not hold in general but it does follow with one additional assumption.

\begin{prop}\label{prop: limit} Say that $M_j\to M^{(1)}$ in $\HH$, $v_j\to v_{\infty}$ in $\mathfrak{R}$, $g_j\to g_{\infty}$ in $G$ and $\rel_{v_j}(g_jM_j)\to M^{(2)}$.
If $\rel_{v_{\infty}}(g_{\infty}M^{(1)})$ exists then $\rel_{v_{\infty}}(g_{\infty}M^{(1)})=M^{(2)}$.
\end{prop}

\begin{proof}
The continuity of the $G$ action as a map from $G\times\HH$ to $\HH$ implies that $g_jM_j\to g_{\infty}M^{(1)}$. If $\rel_{v_{\infty}}(g_{\infty}M^{(1)})$ exists then according to Proposition \ref{prop: continuous}, there are neighborhoods of $U$ of $g_{\infty}M^{(1)}$ and $V$ of $v_{\infty}$ for which $\rel$ is defined and continuous as a function 
on $U \times V$. It follows that $\rel_{v_j}(g_jM_j)$ converges to $\rel_{v_{\infty}}(gM^{(1)})$. Since $\rel_{v_j}(g_jM_j)$ converges to $M^{(2)}$ by hypothesis, we have 
$\rel_{v_{\infty}}(g_{\infty}M^{(1)})=M^{(2)}$.
\end{proof}

\subsection{Real Rel}\label{subsec: real rel}
Let us write $\R^2$ as $\R_x\oplus\R_y$. We then write
\begin{equation}\label{eq: splitting}{
H^1(S, \Sigma; \R^2) \cong
H^1(S, \Sigma; \R_x) \oplus H^1(S, \Sigma; \R_y).
}\end{equation}
and we refer to $H^1(S, \Sigma; \R_x)$ \index{H1@$H^1(S, \Sigma; \R_x)$} as the {\em horizontal
  space}. \index{horizontal space}
Let $Z$ \index{Z@$Z$} denote the intersection of $\mathfrak{R}$ and the horizontal
  space.  We will refer to $Z$ as {\em real
   Rel}. Since the subgroup $B$ of  $\SL_2(\R)$ 
preserves the horizontal
directions $\R_x$, its action on $\mathfrak{R}$ leaves $Z$ invariant.

A special case which will concern us here are
strata with two singularities. In this case $\mathfrak{R}$ can be
identified with $\R^2$, and we make
the identification explicit. Label the singularities of the model surface $S$ by
$\xi_1$ and $\xi_2$, we will identify $\mathfrak{R}$ with $\R^2$ as follows: a
cochain $v:H_1(S, \Sigma) \to \R^2$ which vanishes on cycles
represented by closed curves is identified with the vector $v(\delta)$
for some (any) directed path $\delta$ from $\xi_1$ to $\xi_2$. 
In this case $Z$ is one dimensional, and we write
$\rel_t(M)$ for $\rel_v(M)$, where $v = (t,0) \in \R^2 \cong
\mathfrak{R}$ via the identification above.

Figures \ref{fig: dec1} and \ref{fig: tdec1} show the effect of
flowing along the real $\rel$
 vector field on a decagon with opposite sides identified. In Figure
 \ref{fig: dec1} the flow has the 
effect of shortening the top saddle connection. The flow cannot be 
continued past the point at which the length of the top saddle
connection shrinks to zero. In Figure 
\ref{fig: tdec1} the lengths of saddle connections are preserved since
they connect vertices of the 
same color, and hence represent a saddle connection from a singularity 
to itself. In this case the flow can be continued for all time.

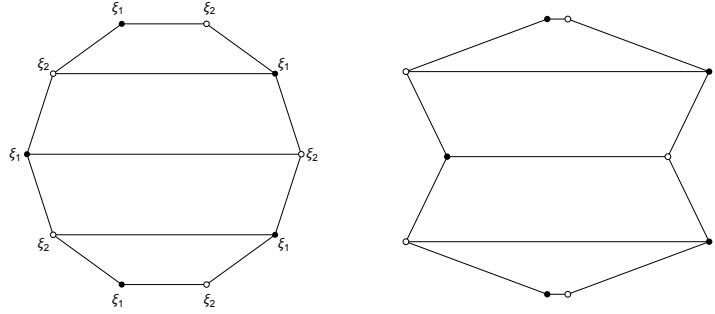
\begin{figure}[h]
\begin{tabular}{l r}
\begin{tikzpicture}[scale=2.0]
\def\a{0.809017};
\def\b{0.587785};
\def\c{0.309017};
\def\d{0.951057};

\node (A0) at (1,0) [circle,draw,inner sep=0pt,minimum size=1.0mm] {};
\node (A1) at (\a,\b) [circle,draw,fill=black,inner sep=0pt,minimum size=1.0mm] {};
\node (A2) at (\c,\d) [circle,draw,inner sep=0pt,minimum size=1.0mm] {};
\node (A3) at (-\c,\d) [circle,draw,fill=black,inner sep=0pt,minimum size=1.0mm] {};
\node (A4) at (-\a,\b) [circle,draw,inner sep=0pt,minimum size=1.0mm] {};
\node (A5) at (-1,0) [circle,draw,fill=black,inner sep=0pt,minimum size=1.0mm] {};
\node (A6) at (-\a,-\b)[circle,draw,inner sep=0pt,minimum size=1.0mm] {};
\node (A7) at (-\c,-\d) [circle,draw,fill=black,inner sep=0pt,minimum size=1.0mm] {};
\node (A8) at (\c,-\d) [circle,draw,inner sep=0pt,minimum size=1.0mm] {};
\node (A9) at (\a,-\b)[circle,draw,fill=black,inner sep=0pt,minimum size=1.0mm] {};

\draw (A0) -- (A1) -- (A2) -- (A3)--(A4) --(A5) -- (A6) -- (A7) -- (A8) --(A9) -- (A0);
\draw (A1) -- (A4); 
\draw (A0) -- (A5);
\draw (A9) -- (A6);
\end{tikzpicture}\ \ \ \ \ \ \ \ \
&
\begin{tikzpicture}[scale=2.0]
\def\a{0.809017};
\def\b{0.587785};
\def\c{0.309017};
\def\d{0.951057};
\def\rel{0.23}

\node (A0) at (1-\rel,0) [circle,draw,inner sep=0pt,minimum size=1.0mm] {};
\node (A1) at (\a+\rel,\b) [circle,draw,fill=black,inner sep=0pt,minimum size=1.0mm] {};
\node (A2) at (\c-\rel,\d) [circle,draw,inner sep=0pt,minimum size=1.0mm] {};
\node (A3) at (-\c+\rel,\d) [circle,draw,fill=black,inner sep=0pt,minimum size=1.0mm] {};
\node (A4) at (-\a-\rel,\b) [circle,draw,inner sep=0pt,minimum size=1.0mm] {};
\node (A5) at (-1+\rel,0) [circle,draw,fill=black,inner sep=0pt,minimum size=1.0mm] {};
\node (A6) at (-\a-\rel,-\b)[circle,draw,inner sep=0pt,minimum size=1.0mm] {};
\node (A7) at (-\c+\rel,-\d) [circle,draw,fill=black,inner sep=0pt,minimum size=1.0mm] {};
\node (A8) at (\c-\rel,-\d) [circle,draw,inner sep=0pt,minimum size=1.0mm] {};
\node (A9) at (\a+\rel,-\b)[circle,draw,fill=black,inner sep=0pt,minimum size=1.0mm] {};

\draw (A0) -- (A1) -- (A2) -- (A3)--(A4) --(A5) -- (A6) -- (A7) -- (A8) --(A9) -- (A0);
\draw (A1) -- (A4); 
\draw (A0) -- (A5);
\draw (A9) -- (A6);
\end{tikzpicture}
\end{tabular}
\caption[Action of Rel on a decagon]{Applying $\mathrm{Rel}_t$ (with $t<0$) to the
  decagon.  When $t<-a$ or $t>b$ then $\rm Rel_t$ fails to be defined,
  where $a$ is the 
  length of the top segment and $b$ is the length of the second
  segment from the top.}  
\label{fig: dec1} 
\end{figure}

\begin{figure}[h]
\begin{tabular}{l c}

\begin{tikzpicture}[scale=2.0]
\def\a{0.809017};
\def\b{0.587785};
\def\c{0.309017};
\def\d{0.951057};
\def\rel{0.1}

\node (A0) at (\d,\c) [circle,draw,inner sep=0pt,minimum size=1.0mm] {};
\node (A1) at (\b,\a) [circle,draw,fill=black,inner sep=0pt,minimum size=1.0mm] {};
\node (A2) at (0,1.0) [circle,draw,inner sep=0pt,minimum size=1.0mm] {};
\node (A3) at (-\b,\a) [circle,draw,fill=black,inner sep=0pt,minimum size=1.0mm] {};
\node (A4) at (-\d,\c) [circle,draw,inner sep=0pt,minimum size=1.0mm] {};
\node (A5) at (-\d,-\c) [circle,draw,fill=black,inner sep=0pt,minimum size=1.0mm] {};
\node (A6) at (-\b,-\a)[circle,draw,inner sep=0pt,minimum size=1.0mm] {};
\node (A7) at (0,-1.0) [circle,draw,fill=black,inner sep=0pt,minimum size=1.0mm] {};
\node (A8) at (\b,-\a) [circle,draw,inner sep=0pt,minimum size=1.0mm] {};
\node (A9) at (\d,-\c)[circle,draw,fill=black,inner sep=0pt,minimum size=1.0mm] {};

\draw (A0) -- (A1) -- (A2) -- (A3)--(A4) --(A5) -- (A6) -- (A7) -- (A8) --(A9) -- (A0);
\draw (A1) -- (A3); 
\draw (A0) -- (A4);
\draw (A9) -- (A5);
\draw (A8) -- (A6);
\end{tikzpicture}
&\ \ \ \ \ \ \ \
\begin{tikzpicture}[scale=2.0]
\def\a{0.809017};
\def\b{0.587785};
\def\c{0.309017};
\def\d{0.951057};
\def\rel{-0.2}

\node (A0) at (\d-\rel,\c) [circle,draw,inner sep=0pt,minimum size=1.0mm] {};
\node (A1) at (\b+\rel,\a) [circle,draw,fill=black,inner sep=0pt,minimum size=1.0mm] {};
\node (A2) at (0-\rel,1.0) [circle,draw,inner sep=0pt,minimum size=1.0mm] {};
\node (A3) at (-\b+\rel,\a) [circle,draw,fill=black,inner sep=0pt,minimum size=1.0mm] {};
\node (A4) at (-\d-\rel,\c) [circle,draw,inner sep=0pt,minimum size=1.0mm] {};
\node (A5) at (-\d+\rel,-\c) [circle,draw,fill=black,inner sep=0pt,minimum size=1.0mm] {};
\node (A6) at (-\b-\rel,-\a)[circle,draw,inner sep=0pt,minimum size=1.0mm] {};
\node (A7) at (0+\rel,-1.0) [circle,draw,fill=black,inner sep=0pt,minimum size=1.0mm] {};
\node (A8) at (\b-\rel,-\a) [circle,draw,inner sep=0pt,minimum size=1.0mm] {};
\node (A9) at (\d+\rel,-\c)[circle,draw,fill=black,inner sep=0pt,minimum size=1.0mm] {};

\draw (A0) -- (A1) -- (A2) -- (A3)--(A4) --(A5) -- (A6) -- (A7) -- (A8) --(A9) -- (A0);
\draw (A1) -- (A3); 
\draw (A0) -- (A4);
\draw (A9) -- (A5);
\draw (A8) -- (A6);
\end{tikzpicture}
\end{tabular}
\caption[Action of Rel on a tipped decagon]{Applying $\rel_t$ (with $t>0$) to the tipped
  decagon.}
\label{fig: tdec1} 
\end{figure}
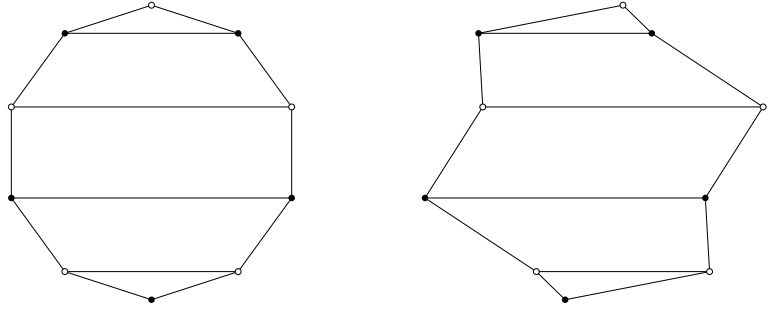

\begin{dfn}\label{dfn: domain of rel}
 Let $v \in \mathfrak{R}$. We denote by $\HH'_v$ \index{H'@$\HH'_v$}
the set
  of $M\in\HH$ for which $\rel_v(M)$ is defined.  
\end{dfn}
Proposition \ref{prop: commuting}(i), with $w =-v$, implies $\rel_{-v}
\circ \rel_v (M) =M$ for $M \in \HH'_v$, and this yields a useful equivariance property:
\begin{equation}\label{eq: equivariance property}{
\rel_{v} (\HH'_v) = \HH'_{-v}.
}\end{equation}
In the case of two singularities keeping in mind our convention
regarding  the identification
$\mathfrak{R} \cong \R^2$, we will
continue to use the notation $\HH'_v$ for $v \in \R^2$. 
Then Proposition~\ref{prop: distributive} implies 
that for $v\in\R^2$ and $g\in G$ we have $g(\HH'_v)=\HH'_{gv}$, where
$gv$ is the image of $v$ under the linear action of $g$ on
$\R^2$. 
We now introduce some more notation for discussing directions belonging to $Z$. 

\begin{dfn} 
We define $\HH'_\infty$ \index{H'infinity@$\HH'_\infty$} to be $\bigcap_{z \in Z}\HH'_z$ (i.e. the set of $M\in\HH$  on which
$\rel_z$ is defined for all $z \in Z$). 
\label{dfn: aitch prime sub infinity}
\end{dfn}

When $\HH$ is a stratum with two singularities
and $t$ is a real number  let $\HH'_t$ \index{H@$\HH'_t$} denote the set $\HH'_v$
  with $v=(t,0)$ (i.e. the subset of $M\in\HH$  on which $\rel_t$ is defined).

\begin{dfn}
For a fixed $M \in \HH$, let 
$$
Z^{(M)} = \{z \in Z: M \in \HH'_z\}
$$ 
\index{Z@$Z^{(M)}$}
(i.e., the subset of $Z$ corresponding to surgeries which are defined on $M$). 
Thus $\HH'_\infty = \{M \in \HH: Z^{(M)}=Z\}.$ 
\end{dfn}

Recall that if $V$ is a vector space and $V_0 \subset V$, we will say
that $V_0$ is a {\em star  body} if 
$$
v \in V_0, s \in [0,1] \implies sv \in V_0.
$$
We denote the convex hull of a subset $W \subset V$ by $\conv \, W$.

\begin{prop}\label{prop: star}
The set $Z^{(M)}$ has the following properties:
\begin{itemize}\item[(i)]
It is an open star body in $Z$.
\item[(ii)] If $b \in B$ then $Z^{(bM)}=b(Z^{(M)} )$. 
\item[(iii)]
If $z \in Z^{(M)}$ then $-z \in Z^{(\rel_z(M))}.$ 
\item[(iv)]
If $z,
z' \in Z$ and $\conv \{0, z, z+z'\} \subset Z^{(M)}$, then $\rel_z(M)$ and
$\rel_{z'}(\rel_z(M))$ are defined and $\rel_{z'}(\rel_z(M)) =  
\rel_{z+z'}(M)$.   
\end{itemize}
\end{prop}
\begin{proof}
The fact that $Z^{(M)}$ is open follows from Proposition \ref{prop:
  continuous}. The
fact that it is a star body is immediate from Definition \ref{def:
  rel}. This proves (i). Assertions (ii), (iii), (iv) follow respectively from Propositions \ref{prop:
  distributive}, \eqref{eq: equivariance property}, and \ref{prop: commuting}(i). 
\end{proof}

We will need a significant strengthening of Proposition \ref{prop:
  star}:
\begin{prop}\label{prop: convex}
For any $M$, the set $Z^{(M)}$ is convex. 
\end{prop}
Whereas the proof of Proposition \ref{prop:
  star} relies only on general principles, Proposition \ref{prop:
  convex} relies on additional information about Rel and will be
proved in \S \ref{subsec: explicit}.

Recall that $G$ acts on $H^1(S, \Sigma; \R^2)$ via its linear action
on the coefficients $\R^2$. Since the linear action of $U$ on $\R^2$ preserves
horizontal vectors, it fixes  elements of $\R_x$. This implies that
real Rel  commutes with the horocycle
flow. 
Namely, by
 Proposition \ref{prop: distributive},  if $z\in Z$ and $u \in U$ then
 $uz=z$, and hence $u(\HH'_z)
=\HH'_z$, and 
$$u(\rel_z(M))=\rel_{u(z)}(uM)=\rel_z(uM), \ \text{ for }  M \in \HH'_z.$$
\begin{dfn}
Now define $$N = \{(b, z): b \in B, z \in Z\}\index{N} \ \text{ and } L =
\{(g,v) : g \in G, v \in \mathfrak{R}\}.\index{L}$$ 
\end{dfn}

We equip $N$ and $L$ with the natural group structures as semi-direct
products $N = B \ltimes Z$ and  $L = G \ltimes \mathfrak{R}$, which
are compatible with their actions on period coordinates. For  the action of the first
factor on the second in this semi-direct product $L$, we take the natural action
of $G$ on $\mathfrak{R}$, and its restriction to the $B$-invariant
subspace $Z$. 
In particular in the case of two singularities,
$G$ acts on $\mathfrak{R} \cong \R^2$ via its standard linear
action, and $B$ acts on $Z \cong \R_x$ via the restriction of its
linear action on $\R^2$, to the horizontal axis. 
We will write this semidirect product group law explicitly as  
$$
(g_1, v_1) \cdot (g_2, v_2) = (g_1 g_2, v_1+ g_1 v_2) \text{ for }
\ell_i = (g_i, v_i) \in L, \,  i=1,2, 
$$
where in the expression $g_1 v_2$ we mean the action of $G$ on
$\mathfrak{R}$ described above.
We can associate a partially defined transformation of $\HH$ to
each element of $L$ as follows:
\begin{equation}\label{eq: on L}{
gM \pluscirc
v = \rel_v(gM) \ \text{ when  } (g,v) \in L \ \text{ and } gM \in
\HH'_v.
}\end{equation}
\index{gM@$gM \pluscirc v$}

We will use a different notation for the restriction to $N$. Namely we
write 
\begin{equation}\label{eq: on N}{
n M =\rel_z(bM) \ \text{ when  } n = (b,z) \in N \ \text{ and
} bM \in \HH'_z.}\end{equation}
\index{nM@$nM$}
 
Note that \eqref{eq: on L} and \eqref{eq: on N} give two different
notations for the same transformations. The reason for this is a
fundamental difference between the behavior of the same operation on
$N$ and on $L$. 
When dealing with all of $L$, these operations need not obey a 
group action law. Indeed, it may happen that  $(M\pluscirc
v)\pluscirc w \neq (M\pluscirc w)\pluscirc v$. However, for $N$ we have the
following weak form of a group action law:

\begin{prop} \label{prop: normalizer} 
Let $n_1 =(b_1,z_1)$ and $n_2 =(b_2,z_2)$ be two elements of  $N$.
Suppose that $n_2M$ and $n_1(n_2M)$ are defined. 
Then 
$(n_1
n_2) M $ is defined, and $n_1(n_2 M) = (n_1 n_2) M$. 
\end{prop}

\begin{proof}
Let $n_1=(b_1,z_1)$ and $n_2=(b_2,z_2)$. Since $n_2M = \rel_{z_2}(b_2M)$ is defined we
have $z_2 \in Z^{(b_2M)}$, and since $n_1(n_2M)$ is defined, we have $z_1
\in Z^{(b_1n_2M)}$. Using
Proposition \ref{prop: star}  we find that $Z^{(n_2M)}$
contains $-z_2$ and  $b_1^{-1} z_1$.
By Proposition \ref{prop: convex} we see that 
\begin{equation}\label{eq: third triangle}{
\conv\{0, -z_2, b_1^{-1} z_1 \} \subset Z^{(n_2M)}.
}\end{equation}

 Then
\begin{align}
n_1(n_2(M)) = & \rel_{z_1}(b_1 n_2(M)) \label{eqq-1}\\
= & \rel_{b_1^{-1}z_1}(n_2(M)) \label{eqq0} \\ 
= &\rel_{b_1^{-1}z_1 +z_2} \circ \rel_{-z_2} (n_2(M)) \label{eqq1} \\ 
= & \rel_{b_1^{-1}z_1 +z_2} (b_2(M)) \label{eqq1.5}\\
=& \rel_{z_1+b_1z_2}(b_1b_2(M)) \label{eqq2.5}\\ 
= & (n_1n_2)(M). \notag
\end{align}
Here Proposition
\ref{prop: distributive} is used to derive 
\eqref{eqq0} from \eqref{eqq-1} and to derive \eqref{eqq2.5} from
\eqref{eqq1.5}, and \eqref{eqq1} is obtained from \eqref{eqq0} by Proposition
\ref{prop: star} and \eqref{eq: third triangle}.
\end{proof}

The following
is immediate from Propositions~\ref{prop: continuous}, \ref{prop: star} and
\ref{prop: normalizer} (and was proved previously in \cite{MW}): 

\begin{cor}\label{cor: real rel via group}
The set $\HH'_\infty$ is $N$-invariant. The map $N\times \HH'_\infty \to \HH'_\infty$ defined by 
$(\ell, M) \mapsto \ell M$ 
defines a continuous action of
$N$ on $\HH'_\infty$. The $Z$-orbits in this action are the real Rel leaves in
$\HH'_\infty$.  
\end{cor}

The following will be useful. For $u \in U$ and $n = (b_n,z_n) \in
N$, write $u^{(b_n)} = bub^{-1}$, so that $n  u = u^{(b_n)}n$ as elements of
$N$. Then we have:

\begin{cor}\label{cor: special}
For any $u \in U$, $n \in N$ and $M \in \HH$, if $n(uM)$ is defined then
$n M$ is defined and $n(uM) = u^{(b_n)} (nM)$. 
\end{cor}
\begin{proof}
We apply Proposition \ref{prop: normalizer} with $b_2=u$ and $z_2=0$.
\end{proof}

\subsection{The stabilizer group of a measure}\label{subsec: group of a measure}
Proposition \ref{prop: star} implies the invariance property $Z^{(uM)} =
Z^{(M)}$. We now extend this to semi-direct products and use it to
define the stabilizer of a measure, within a collection of
partially defined transformations. 

For fixed $M \in \HH$ we write
$$
N^{(M)} = \left\{(b,z) \in N: z \in
Z^{(bM)} \right\}. 
$$
(i.e. the set of $n \in N$ for which $nM$ is defined). Then it follows
from Proposition \ref{prop: continuous} that $N^{(M)}$ is open for
each $M$, and Corollary \ref{cor: special} implies $N^{(uM)} =
N^{(M)}$.  

\begin{prop}\label{prop: Z continuous dependence}
Suppose $M$ is in the $U$ orbit-closure of $M'$. Then $N^{(M)} \subset
N^{(M')}$.  
\end{prop}
\begin{proof}
If $n \in N^{(M)}$ then by Proposition
\ref{prop: continuous} there is a neighborhood $\mathcal{W}$ of $M$ in
$\HH$ such that $n\in N^{(M_1)}$ for any $M_1 \in \mathcal{W}$. Let $u
\in U$ such that $uM' \in \mathcal{W}$. Then $z \in N^{(uM')} = N^{(M')}$.
\end{proof}

We also need the following:
\begin{prop} \label{prop: critical time} Given an ergodic $U$-invariant
  measure $\mu$ there is a subset $\Omega \subset \HH$ such that
  $\mu(\Omega)=1$ and for any $M_1, M_2 \in \Omega, \, Z^{(M_1)} =
  Z^{(M_2)}$ and $N^{(M_1)} = N^{(M_2)}$. 
\end{prop}

\begin{proof} To explain the idea, we first prove the assertion for
  $Z^{(M)}$ in case $\dim Z=1$. In this case, $Z^{(M)}$ is an open interval for every $M \in
  \HH$, which we can write as 
$$
Z^{(M)} = (a_M, b_M), \text{ for some } -\infty \leq a_M < 0 < b_M \leq \infty.
$$
The maps $M \mapsto a_M, M \mapsto b_M$ are measurable maps with
values in the extended real line so by ergodicity, are constant
$\mu$-a.e., and the statement follows. 

In the general case the proofs for $Z^{(M)}$ and $N^{(M)}$ are
identical, so we discuss the case of $Z^{(M)}$. The map $M \mapsto Z^{(M)}$ is a map from $\HH$ to the
collection of open subsets of $Z$. 
Rather than worry about
measurability issues, we proceed as follows. Let $Z_0$ be a dense
countable subset of $Z$. For each $z \in Z_0$, the set 
$$
\Omega_{z} = \left\{M \in \HH: z \in Z^{(M)} \right\}
$$
is a measurable invariant set, so has measure 0 or 1 by ergodicity. We define 
$$\Omega = \bigcap_{\mu(\Omega_{z})=1} \Omega_{z} \sm
\bigcup_{\mu(\Omega_{z'})=0} \Omega_{z'},$$
where $z, z'$ range over the countable set $Z_0$. 
Then $\mu(\Omega)=1$, and for any $M \in \Omega$, 
$$
Z^{(M)} \cap Z_0 = \{z \in Z_0: \mu(\Omega_{z})=1\},
$$
i.e., does not depend on $M$. 
Since an open set is the interior of the closure of its
intersection with a dense set, we see that $Z^{(M_1)} = Z^{(M_2)}$ for any
$M_1, M_2 \in \Omega$. 
\end{proof}

We will denote by $Z^{(\mu)}$ and $N^{(\mu)}$ the sets $Z^{(M)},
N^{(M)}$ which appear in Proposition \ref{prop: critical time} for $
M \in \Omega$. If $z \in Z^{(\mu)}$ we define a pushforward
$$
\rel_{z*} \mu(X) = \mu(\rel_{-z}(X \cap \HH'_{-z})), \text{ for all measurable
} X \subset \HH.  
$$
Note that
$
\mu(\HH'_z)=1, 
$
and now it follows from \eqref{eq: equivariance property} that
$\rel_{z*} \mu$ is a probability measure. 
Moreover, the partially defined map $\rel_{z}$ is a measurable conjugacy between
$(\HH,\mu)$ and $(\HH, \rel_{z*}\mu)$, thought of as dynamical systems
for the $U$-action. Thus $\rel_{z*}\mu$ is again an
ergodic $U$-invariant measure. The same observations are valid for $n
\in N^{(\mu)}$. Namely, if we denote 
\begin{equation}\label{eq: defn H'n}{
\HH'_{n} = \{M \in \HH: n(M) \text{ is defined} \}  = \left\{M \in \HH : z
\in Z^{(bM)} \right\} 
}\end{equation}
(where $n = (b,z)$),
then we have an equivariance property 
\begin{equation}\label{eq: another equivariance}{
n(\HH'_{n}) = \HH'_{n^{-1}}
}\end{equation}
 and we can define an ergodic $U$-invariant measure 
$$
n_* \mu(X) = \mu(n^{-1}(X \cap \HH'_{n^{-1}})), \text{ for all measurable
} X \subset \HH. 
$$
Corollary \ref{cor: special} now implies that $n_*\mu$ is a $U$-invariant
measure, and the partially defined map $M \mapsto n(M)$ is equivariant
for the action of $U$ on $(\HH, \mu)$ and the ``time-changed'' action
of $U$ on $(\HH, n_* \mu)$ via 
$$
u\cdot M = u^{(n)}M, \text{ for } n_*\mu \text{ a.e. } M
$$
(where $u^{(n)} = nun^{-1}$).

The collection of Borel probability measures on a locally compact
space $X$ can be given the weak-$*$ topology by embedding 
it in the dual space of the space $C_c(X)$ of continuous functions
with compact support. 

\begin{prop}
 Let $\mu$ be an ergodic $U$-invariant probability measure. The map
 which takes 
$
n \in N^{(\mu)} 
$ to $n_* \mu$ is continuous with respect to
  the weak-$*$ topology. 
\label{prop: continuity}
\end{prop}

\begin{proof} Let $n_j$ be a sequence of
  elements of $N^{(\mu)}$ converging to $n_\infty \in N^{(\mu)}$. In order to show that
  $n_{j*}\mu$ converges to $n_{\infty*}\mu$ we need to show that
  for any continuous compactly supported function $\varphi$ on $\HH$,
  we have 
$$
\lim_{j\to\infty}\int_\HH \varphi\, d(n_{j*}\mu)=\int_\HH \varphi\, d(n_{\infty*}\mu).
$$
For each $j$, the set $\HH'_{n_j}$ has full $\mu$-measure, and hence
so does $\HH'_0 = \bigcap_{1 \leq j \leq \infty} \HH'_{n_j}.$ Now we
have: 
\begin{align}
\lim_{j\to\infty}\int_\HH \varphi \,d(n_{j*}\mu)&=\lim_{j\to\infty}\int_{n_j(\HH'_0)} \varphi \,d(n_{j*}\mu)\label{eqa}\\
                                                               &=\lim_{j\to\infty}\int_{\HH'_0}
                                                               \varphi\circ
                                                               n_{j}\,d\mu\label{eqb}\\
                                                               &=\int_{\HH'_0}\lim_{j\to\infty}
                                                               \varphi\circ
                                                               n_{j}\,d\mu\label{eqc}\\
                                                               &=\int_{\HH'_0}
                                                               \varphi\circ
                                                               n_{\infty}\,d\mu\label{eqd}\\
                                                               &=\int_{n_\infty(\HH'_0)} \varphi \,d(n_{\infty*}\mu)\label{eqe}\\
                                                               &=\int_\HH \varphi \,d(n_{\infty*}\mu).\label{eqf}
\end{align}                                                               
The equalities \eqref{eqa} and \eqref{eqf} follow from the fact that each
$n_{j*} \mu$ assigns full
measure to $n_j(\HH'_{0})$ respectively. The lines \eqref{eqb} and
\eqref{eqe} follow from the definition of the pushforward of a
measure. Line
\eqref{eqc} follows from Lebesgue's Dominated Convergence Theorem
using the fact that, since $f$ has compact support, it is bounded and 
hence the family of functions $f\circ n_{j}$ is uniformly bounded
and also using the fact that 
constant functions are in $L^1(\mu)$ since $\mu$ is a probability measure.  Line \eqref{eqd} follows from
Proposition \ref{prop: continuous} and the continuity of $\varphi$. 
\end{proof}

\begin{dfn}\label{dfn: stabilizer}
For any ergodic $U$-invariant measure $\mu$ we define 
$$N_\mu = \left\{n \in N^{(\mu)}: n_*\mu=\mu \right\}.$$
\index{Nm@$N_\mu$}
\end{dfn}

\begin{cor}\label{cor: closed}
 $N_\mu$ is a closed subgroup of $N$. 
\end{cor}
\begin{proof}
The fact that $N_\mu$ is closed follows from Proposition \ref{prop:
  continuity},  and in order to prove that $N_\mu$ is closed under
compositions, it suffices to show that for $\mu$-almost every $M$,
$n_1(n_2M) = (n_1 n_2)M$. 
This follows from
Proposition \ref{prop: normalizer}. 
\end{proof}

\begin{prop}\label{prop: stabilizer connected}
If $N_{\mu}$ contains a non-unipotent element
(i.e.\ an element in  $N  \sm UZ$) and $\dim Z =1$ then $N_{\mu} \cap
Z$ is connected.  In particular, if $N_{\mu}$ contains a non-unipotent element 
of $N$ and a nontrivial element of $Z$ then it contains all of $Z$.
\end{prop}

\begin{proof}
Write  $Z_1 = U Z \cong \R^2$ and $N_1 =N_{\mu} \cap  Z_1, $ and let $a \in
N_\mu \sm N_1.$ By Corollary \ref{cor: closed}, $N_\mu$ is a closed subgroup of 
$N$, and hence $N_1$  is
a closed subgroup of $Z_1$ containing $U$.  
If it is not connected then 
there is a minimal positive distance between two distinct cosets 
for $N_1$ in $Z_1$. However $N_1$ is invariant under conjugation by the elements $a$ and $a^{-1}$, 
and one of these acts on $Z_1$ by contractions. This contradiction
proves the claim. 
\end{proof}

\subsection{Generic points and real Rel}\label{subsec: generic points}
Recall that if $\mu$ is a $U$-invariant ergodic probability measure on 
a closed subset $\LL$ of a stratum $\HH$, then $M \in
\LL$ is said to be {\em generic for $\mu$}\index{generic for $\mu$} if for any continuous compactly
supported function $f$ on $\LL$, we have 
$$
\lim_{T\to\infty} \frac{1}{T} \int_0^T f(u_sM) \, ds =
\int f \, d\mu.
$$
\begin{dfn} Let $\mu$ be an invariant $U$-ergodic measure on $\HH$. Then
$\Omega_\mu$ \index{O@$\Omega_\mu$} denotes the set of generic points for
$\mu$.
\end{dfn}

We collect some well-known facts about generic points.

\begin{prop}\label{prop: well known}
If $\mu$ is ergodic then $\Omega_\mu$ has full $\mu$ measure. If
$\mu\ne\nu$ then $\Omega_\mu$ and $\Omega_\nu$ are disjoint. If
$M\in\Omega_\mu$ then the support of $\mu$ is contained in the orbit
closure of $M$. 
\end{prop}

In light of Proposition \ref{prop: Z continuous dependence}, this implies:

\begin{cor}\label{cor: give a name}
If $M \in \Omega_\mu$ then $N^{(\mu)} \subset N^{(M)}$. 
\end{cor}

It will prove useful to reformulate
the notion of genericity in terms of convergence of measures. It is
important to keep in mind that we are dealing with spaces which are
locally compact but not compact. The weak-$*$ topology on the space of
Borel measures is defined in terms of integrals of continuous
functions of compact support. In this topology the total
mass of a measure need not be a continuous function on the space of
measures since the constant function 1 need not have compact support. 

\begin{dfn}
Let $\nu(M,T)$ \index{n@$\nu(M,T)$} be the probability measure defined by 
$$ \int \varphi\,d\nu(M,T)=\frac{1}{T}\int_0^T 
\varphi(u_sM) \, ds, \ \ \text{ for all } \varphi \in C_c(\HH).$$
\label{dfn: orbit measure}
\end{dfn}

 Then $M$ is generic for $\mu$ if 
\begin{equation}\label{eq: reformulate}
{\lim_{T \to \infty} \nu(M,T)=\mu
}\end{equation}
where the limit is taken with respect to the weak-$*$ topology on measures on $\HH$.

In preparation for Proposition \ref{prop: Katok Spatzier} below, we
collect some results 
about integrals of continuous bounded functions along
orbits. 
The space of probability measures on a locally compact space can be given another topology
by embedding it in the dual space of $C_b(X)$, the space of continuous
bounded functions on $X$. This topology is called the strict topology
(see \cite{Bour} for more information). Clearly convergence in the
strict topology implies convergence in the weak-$*$ topology. The following result gives a
simple criterion for showing that weak-$*$ convergence implies strict
convergence. 

\begin{prop}[\cite{Bour} Prop. 9, p. 61] \label{prop: convergence
    criterion} Suppose $X$ is a locally compact space, $\mu_j$ is a sequence of
  probability measures on $X$, and $\mu$ is also a measure on $X$. If
  $\mu_j \to \mu $ with respect to the weak-$*$ topology, and if $\mu$
  is a probability measure, then $\mu_j \to \mu$ with respect to the
  strict topology. 
\end{prop}

If the limit measure $\mu$ is not a probability measure
then it has total mass less than one. This phenomenon is referred to
as `loss of mass'. In \S \ref{sec: injectivity} we will give conditions that
establish that loss of mass does not occur for measures $\nu(M,T)$. 

We will repeatedly use the following:
\begin{prop}\label{prop: Katok Spatzier}
Let $\mu$ be an ergodic $U$-invariant measure.  If $n\in N^{(\mu)}$ then $n(\Omega_\mu)$ is the 
set of generic points for $n_*\mu$. 
\end{prop}

\begin{proof} Let $M\in \Omega_\mu$. According to Corollary \ref{cor:
    give a name}, $nM$ is defined. We use the reformulation of genericity in terms of weak-$*$
  convergence of measures given above. Thus we are given that
  $\lim_{T\to\infty}\nu(M,T)=\mu$ with respect to the weak-$*$
  topology on measures on $\HH$ and we want to show that
  $\lim_{T\to\infty}\nu(n M,T)=n_*\mu$ with respect to the same
  topology. 

Since $N$ normalizes $U$, there is $c>0$ such that 
$n u_s = u_{cs}n.$ We have $n_*\nu(M,T)=\nu(n M,cT)$. Thus it suffices to show:
$$
\lim_{T\to\infty} n_*\nu(M,T)=n_*\mu.\label{eq: line1}
$$

Let $\HH'_1 = \HH'_n$ as in \eqref{eq: defn H'n} and let $\HH'_2 =
n(\HH'_1) = \HH'_{n^{-1}}$. 
Let $\mathfrak{M}$ denote the set of
probability measures on $\HH$ which assign mass 1 to $\HH'_1$, and 
let $\mathfrak{N}$ denote the set of
probability measures on $\HH$ which assign mass 1 to $\HH'_2$.    
By Proposition \ref{prop:
  continuous}, $\HH'_1$ and $\HH'_2$ are open subsets of $\HH$, and in
particular are locally compact. 
Moreover any continuous compactly supported function on
$\HH'_1$ extends to a continuous function on $\HH$ by setting it
equal to zero on $\HH \sm \HH'_1$. It follows that
$\lim_{T\to\infty}\nu(M,T)=\mu$ with respect to the weak-$*$ topology
on $\mathfrak{M}$. 

Since $n$ is a homeomorphism from $\HH'_1$ to
$\HH'_2 $, the map $n_*: \mathfrak{M} \to \mathfrak{N}$ 
is continuous and $\lim_{T\to\infty}
n_*\nu(M,T)=n_*\mu$ with respect to the weak-$*$ topology on $\mathfrak{N}$.  
Since $n_*\mu$ is a probability measure 
which assigns full
measure to $\HH'_2$, and since 
$\HH'_2$ is locally compact,  we can use 
Proposition \ref{prop: convergence criterion} to conclude that 
$\lim_{T\to\infty} n_*\nu(M,T)=n_*\mu$ in the strict topology on
measures on $\HH'_2$.  

We need to show that convergence also holds with respect to the weak-$*$
topology on measures on $\HH$. Given a function $\varphi$ on $\HH$
which is  continuous and compactly supported,  its restriction  to
$\HH'_2$ is a bounded continuous function. Therefore
$$
\lim_{T\to\infty} \int_{\HH'_2} \varphi\, dn_*\nu(M,T)=\int_{\HH'_2} \varphi\, dn_*\mu
$$
by strict convergence which gives
 $$
\lim_{T\to\infty} \int_{\HH} \varphi\, dn_*\nu(M,T)=\int_{\HH} \varphi\, dn_*\mu,
$$
and establishes convergence with respect to the weak-$*$ topology on measures on $\HH$.
\end{proof}

\begin{cor}\label{cor: Katok Spatzier}
  If $\mu$ is an ergodic $U$-invariant probability measure and if
  $M_1$ and $M_2$ are in $\Omega_\mu$ and there is an element $n \in N^{(\mu)}$ 
of $N$ such that $M_2 = n M_1$, then $n \in N_\mu$. 
\end{cor}

\begin{proof}
Let $M_1$ and $M_2$ be generic for $\mu$ and $M_2 = n M_1$ where $n
\in N$. According to Proposition~\ref{prop: Katok Spatzier}, $\mu$ and
$n_* \mu$ share a generic point and hence, by Proposition
\ref{prop: well known}, they  must coincide. 
\end{proof}

\section{Horizontal equivalence of surfaces}\label{sec: separatrix diagram}
Given $M \in\HH$ with singularity set $\Sigma$,  we denote by $\Xi(M)$ \index{X@$\Xi(M)$} the set of
horizontal saddle connections for $M$. 
We would like to use $\Xi(M)$ to define two equivalence relations on
surfaces: {\em topological horizontal equivalence}  and {\em
  geometrical horizontal equivalence}.  These
will serve as invariants of ergodic $U$-invariant measures in the same way that the sets $Z^{(\mu)}$ and 
$N^{(\mu)}$ appearing in Proposition \ref{prop: critical time} do.
 Let $\check{M}$ be the blowup
of $M$ and let $(\check{f} , \check{M}) \in \til \HH$ be a marking
of $\check{M}$ rel boundary. Then $\check{f}^{-1}(\Xi(M))$ is a subset of
$\check{S}$, which we denote by $\check{f}^*(\Xi)$. We take $\Xi(M)$ to include all of the 
points of $\Sigma$ and hence $\check{f}^*(\Xi)$ contains  the boundary
$\partial \check{S}$. In addition, for each edge of $\Xi(M)$,
$\check{f}^*(\Xi)$ contains an edge in the interior of $S$, which
intersects the boundary  $\partial \check{S}$ at points with angular
parameters which are multiples of $\pi$, even or odd according as the
point is the initial or terminal point of the edge. 

\begin{dfn} We will say that $M_1$ and $M_2$ are {\em topologically
 horizontally equivalent} \index{topological horizontal equivalence} if there are markings rel boundary
$(\check{f}_i, \check{M}_i)$ such that
$\check{f}_2^*(\Xi)$ and $\check{f}_1^*(\Xi)$ can be obtained from
each other by an isotopy of $\check{S}$ that does not move points of
$\partial \check{S}$. We will say that $M_1$ and $M_2$ are
{\em geometrically horizontally equivalent} \index{geometrical horizontal equivalence} if they are topologically
horizontally equivalent, and if for any edge $\delta$ in $\check{f}_i^*(\Xi(M_i))$,
$\hol(M_1, \check{f}_1(\delta)) = \hol(M_2, \check{f}_2 
(\delta))$, where $\check{f}_1$ and $\check{f}_2$ are as in the definition
of topological equivalence. 
\end{dfn}

\begin{prop}
\label{prop: horizontal relations}
If $M \in \HH,$ $b \in B$ and $u \in U$ then
$M$ and $bM $ are  topologically  horizontally equivalent, and 
$M$ and  $ u M$ are geometrically horizontally 
equivalent. 
\end{prop}

\begin{proof}
Since $B$ acts by postcomposition on the charts, and
preserves the horizontal direction, it preserves the set $\Xi(M)$
(possibly changing the lengths of edges), and hence $M$ and $bM$ are
topologically equivalent. 
Moreover if $b = u \in U$ then the lengths
of edges are unaffected and so $M$ and $uM$ are geometrically
horizontally equivalent. 
\end{proof}

\begin{cor}\label{cor: horizontal relations}
If $\mu$ is an ergodic $U$-invariant measure on $\HH$, then there is a
subset $X \subset \HH$ of full $\mu$-measure such that for any $M_1,
M_2 \in X,$ $M_1$ and $M_2$ are geometrically horizontally equivalent.   
\end{cor}

We will define a combinatorial invariant of 
topological horizontal equivalence -- the {\em horizontal data
  diagram}. \index{horizontal data
  diagram} This diagram is an analog of the separatrix diagram
introduced by Kontsevich and Zorich \cite{KZ} and captures some of the
properties of $\Xi(M)$ which depend only on the class of $M$. Note
that a graph embedded on a translation surface has a ribbon graph structure,
namely at each vertex it inherits from the surface a cyclic order of
edges incident to the vertex. Its vertices are labeled by the labels
$\xi_1, \ldots, \xi_k$ of the corresponding singularities.  If the graph consists of
horizontal saddle connection, each edge inherits an orientation from
the translation structure and we take these orientations as part of the horizontal data diagram structure.

For any $(\check{f}, \check{M}) \in \til \HH$,
 $\check{f}^*(\Xi)$ is a graph embedded in $\check{S}$. By projecting it to $S$
we thus obtain a ribbon graph with labeled singularities and oriented edges. The graphs
$\check{f}^*(\Xi)$ carry additional information, namely the angular distance separating
consecutive 
edges incident at a vertex. This angle is an integer multiple of $\pi$, and
the orientations of consecutive edges agree if and only if this
distance is an even multiple of $\pi$.
It is clear that all of this information is an
invariant of geometric horizontal equivalence. 

The 
horizontal data
  diagram of $M$ 
records the graph
structure induced by $\check{f}^*(\Xi)$, as well as the orientation of
edges, labeling of vertices, 
and cyclic structure at each singularity. In order to
record the information of angular separation of edges at vertices, we
indicate as dashed lines additional left and right pointing
horizontals. Thus at the singularity $\xi_i$ there are $2(a_i+1)$
prongs, some of which correspond to edges of the
graph. Figures~\ref{fig: blown up}, \ref{fig: maximal} and \ref{fig: nonmaximal} 
depict graphs which can occur as $\Xi(M)$ for
some $M$ in $\HH(1,1)$.

  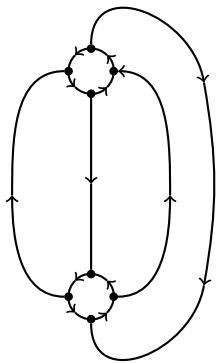
\begin{figure}[h]
  \begin{tikzpicture}[scale=1.5]
\coordinate (A) at (-.7,-.1);
\coordinate (B) at (.7,-.1);
\coordinate (C) at (1,.9);
\coordinate (D) at (1,-.9);
\coordinate (E) at (0,0);
   \node (black) at (0,1)  {};
  \node (white) at (0,-1) {};
   
\node (blacku) at (0,1.2) [circle,draw,fill=black,inner sep=0pt,minimum size=1mm] {};
\node (blackd) at (0,0.8) [circle,draw,fill=black,inner sep=0pt,minimum size=1mm]{};
\node (blackr) at (0.2,1) [circle,draw,fill=black,inner sep=0pt,minimum size=1mm]{};
\node (blackl) at (-0.2,1)[circle,draw,fill=black,inner sep=0pt,minimum size=1mm]{};

\node (whiteu) at (0,-.8)[circle,draw,fill=black,inner sep=0pt,minimum size=1mm]{};
\node (whited) at (0,-1.2)[circle,draw,fill=black,inner sep=0pt,minimum size=1mm]{};
\node (whiter) at (0.2,-1)[circle,draw,fill=black,inner sep=0pt,minimum size=1mm]{};
\node (whitel) at (-0.2,-1)[circle,draw,fill=black,inner sep=0pt,minimum size=1mm]{};

\draw[->,thick] (whitel) to [out=180,in=270,looseness=1.0] (A);
\draw[thick] (A) to [out=90,in=180,looseness=1.0] (blackl);
\draw[->,thick] (whiter) to [out=0,in=270,looseness=1.0] (B);
\draw[->,thick] (B) to [out=90,in=0,looseness=1.0] (blackr);
\draw[->,thick] (blackd) to [out=270,in=90,looseness=1.5] (E);
\draw[thick] (E) to [out=270,in=90,looseness=1.5] (whiteu);
\draw[->,thick] (blacku) to [out=90,in=100,looseness=1.5] (C);
\draw[->,thick] (C) to [out=280,in=80,looseness=1.0] (D);
\draw[thick] (D) to [out=260,in=270,looseness=1.5] (whited);

\draw[->,thick] (0,1)++(45:2mm) arc(45:135:2mm);
\draw[->,thick] (0,1)++(135:2mm)  arc(135:225:2mm);
\draw[->,thick] (0,1)++(225:2mm)  arc(225:315:2mm);
 \draw[->,thick] (0,1)++(315:2mm)  arc(315:405:2mm);


\draw[->,thick] (0,-1)++(45:2mm) arc(45:135:2mm);
\draw[->,thick] (0,-1)++(135:2mm)  arc(135:225:2mm);
\draw[->,thick] (0,-1)++(225:2mm)  arc(225:315:2mm);
 \draw[->,thick] (0,-1)++(315:2mm)  arc(315:405:2mm);

\end{tikzpicture}
\hspace{8 mm}
\caption{A horizontal data diagram on a blown up surface.}
\label{fig: blown up}
\end{figure}

In Figure~\ref{fig: maximal} the
angle between any two adjacent ends is exactly $\pi$, but this is not
the case in Figure~\ref{fig:
   nonmaximal}.  
 We capture this angle information in the diagrams in Figure~\ref{fig:
   nonmaximal} by inserting additional dotted lines at each vertex
 indicating the ends of horizontal separatrices which are not saddle
 connections. We
 can determine the angle between two prongs by counting the number of
 ends of separatrices between them. Note that it makes sense to give orientations to
 separatrices so that the orientations of ends at a given vertex
 alternate with respect to the cyclic ordering.

It follows from Corollary \ref{cor: horizontal relations} that for any
ergodic $U$-invariant measure $\mu$, there is a  subset $X$ of $\HH$
such that $\mu(X)=1$ and every  $M \in X$ has the same horizontal data
diagram. We call it the {\em horizontal data diagram of $\mu$} and
denote it by $\Xi(\mu)$. \index{X@$\Xi(\mu)$} \index{horizontal data diagram of $\mu$}

If a horizontal data diagram is {\em maximal}\index{maximal horizontal data diagram}, i.e. at each vertex,
all prongs are initial or terminal prongs of edges, then the
horizontal data diagram coincides with the separatrix diagram of
\cite{KZ}.

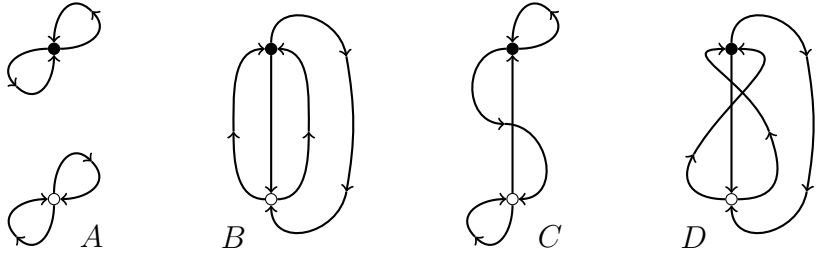
\begin{figure}[h]
\begin{tabular}{l c c c r}
\begin{tikzpicture}
\coordinate (A) at (.5,1.5);
\coordinate (B) at (-.5,-1.5);
\coordinate (AA) at (-.5,.5);
\coordinate (BB) at (.5,-.5);
 \node (black) at (0,1) [circle,draw,fill=black,inner sep=0pt,minimum size=1.5mm] {};
   \node (white) at (0,-1) [circle,draw,inner sep=0pt,minimum size=1.5mm] {};
\draw[->,thick] (black) to [out=0,in=315,looseness=1.5] (A);
\draw[->,thick] (A) to [out=135,in=90,looseness=1.5] (black);
\draw[->,thick] (AA) to [out=315,in=270,looseness=1.5] (black);
\draw[->,thick] (black) to [out=180,in=135,looseness=1.5] (AA);
\draw[->,thick] (white) to [out=270, in=315,looseness=1.5] (B);
\draw[->,thick] (B) to [out=135,in=180,looseness=1.5] (white);
\draw[->,thick] (BB) to [out=315, in=0,looseness=1.5] (white);
\draw[->,thick] (white) to [out=90,in=135,looseness=1.5] (BB);
\node (label) at ($ (white)+(.5,-.5) $) {$A$};
\end{tikzpicture}
\hspace{8 mm}
&
\begin{tikzpicture}
\coordinate (A) at (-.5,-.1);
\coordinate (B) at (.5,-.1);
\coordinate (C) at (1,.9);
\coordinate (D) at (1,-.9);
 \node (black) at (0,1) [circle,draw,fill=black,inner sep=0pt,minimum size=1.5mm] {};
   \node (white) at (0,-1) [circle,draw,inner sep=0pt,minimum size=1.5mm] {};
\draw[->,thick] (white) to [out=180,in=270,looseness=1.0] (A);
\draw[->,thick] (A) to [out=90,in=180,looseness=1.0] (black);
\draw[->,thick] (white) to [out=0,in=270,looseness=1.0] (B);
\draw[->,thick] (B) to [out=90,in=0,looseness=1.0] (black);
\draw[->,thick] (black) to [out=270,in=90,looseness=1.5] (white);
\draw[->,thick] (black) to [out=90,in=100,looseness=1.5] (C);
\draw[->,thick] (C) to [out=280,in=80,looseness=1.0] (D);
\draw[->,thick] (D) to [out=260,in=270,looseness=1.5] (white);
\node (label) at ($ (white)+(-.5,-.5) $) {$B$};
\end{tikzpicture}
\hspace{8 mm}
&
\begin{tikzpicture}
\coordinate (A) at (.5,1.5);
\coordinate (B) at (-.5,-1.5);
\coordinate (C) at (-.1,0);
 \node (black) at (0,1) [circle,draw,fill=black,inner sep=0pt,minimum size=1.5mm] {};
   \node (white) at (0,-1) [circle,draw,inner sep=0pt,minimum size=1.5mm] {};
\draw[->,thick] (black) to [out=0,in=315,looseness=1.5] (A);
\draw[->,thick] (A) to [out=135,in=90,looseness=1.5] (black);
\draw[->,thick] (white) to [out=270, in=315,looseness=1.5] (B);
\draw[->,thick] (B) to [out=135,in=180,looseness=1.5] (white);
\draw[->,thick] (white) -- (black);
\draw[->,thick] (C) to [out=0, in=0,looseness=1.5] (white);
\draw[->,thick] (black) to [out=180, in=180,looseness=1.5] (C);
\node (label) at ($ (white)+(.5,-.5) $) {$C$};
\end{tikzpicture}
\hspace{8 mm}
&
\begin{tikzpicture}
\coordinate (A) at (-.5,-.4);
\coordinate (B) at (.5,-.1);
\coordinate (C) at (1,.9);
\coordinate (D) at (1,-.9);
 \node (black) at (0,1) [circle,draw,fill=black,inner sep=0pt,minimum size=1.5mm] {};
   \node (white) at (0,-1) [circle,draw,inner sep=0pt,minimum size=1.5mm] {};
\draw[->,thick] (white) to [out=180,in=240,looseness=1.5] (A);
\draw[->,thick] (A) to [out=60,in=0,looseness=1.5] (black);
\draw[->,thick] (white) to [out=0,in=290,looseness=1.5] (B);
\draw[->,thick] (B) to [out=110,in=180,looseness=1.5] (black);
\draw[->,thick] (black) to [out=270,in=90,looseness=1.5] (white);
\draw[->,thick] (black) to [out=90,in=100,looseness=1.5] (C);
\draw[->,thick] (C) to [out=280,in=80,looseness=1.0] (D);
\draw[->,thick] (D) to [out=260,in=270,looseness=1.5] (white);
\node (label) at ($ (white)+(-.5,-.5) $) {$D$};
\end{tikzpicture}
  \\
\end{tabular}
\caption[Maximal horizontal data diagrams]{List of all maximal horizontal data diagrams
  in $\HH(1,1)$ up to 
  switching labels and orientations. See Figure \ref{fig: tdec1} for  a
  polygonal presentation of a cylinder decomposition of type A.
  \index{Types of horizontal data diagrams}
  }
 \label{fig: maximal}
\end{figure}

\begin{figure}[h]
\begin{tabular}{l c c c c r}
\begin{tikzpicture}
\coordinate (A) at (.5,1.5);
\coordinate (B) at (-.5,-1.5);
 \node (black) at (0,1) [circle,draw,fill=black,inner sep=0pt,minimum size=1.5mm] {};
   \node (white) at (0,-1) [circle,draw,inner sep=0pt,minimum size=1.5mm] {};
\draw[->,thick] (black) to [out=0,in=315,looseness=1.5] (A);
\draw[->,thick] (A) to [out=135,in=90,looseness=1.5] (black);
\draw[->,thick] (white) to [out=270, in=315,looseness=1.5] (B);
\draw[->,thick] (B) to [out=135,in=180,looseness=1.5] (white);
\draw[->,thick, densely dotted]  ($ (black)-(0,.5) $)--(black);
\draw[->,thick, densely dotted]   (black)--($ (black)-(.5,0) $);
\draw[->,thick, densely dotted]  (white)--($ (white)+(0,.5) $);
\draw[->,thick, densely dotted]   ($ (white)+(.5,0) $)--(white);
\node (label) at ($ (white)+(.5,-.5) $) {$1$};
\end{tikzpicture}
\hspace{2 mm}
&
\begin{tikzpicture}
 \node (black) at (0,1) [circle,draw,fill=black,inner sep=0pt,minimum size=1.5mm] {};
   \node (white) at (0,-1) [circle,draw,inner sep=0pt,minimum size=1.5mm] {};
\draw[->,thick] (black) to [out=0,in=315,looseness=1.5] ($ (black)+(.5,.5) $);
\draw[->,thick] ($ (black)+(.5,.5) $) to [out=135,in=90,looseness=1.5] (black);
\draw[->,thick] (white) to [out=270, in=315,looseness=1.5] ($ (white)-(.5,.5) $);
\draw[->,thick] ($ (white)-(.5,.5) $) to [out=135,in=180,looseness=1.5] (white);
\draw[->,thick] (white) -- (black);
\draw[<-,thick, densely dotted]   (black)--($ (black)-(.5,0) $);
\draw[->,thick, densely dotted]   ($ (white)+(.5,0) $)--(white);
\node (label) at ($ (white)+(.5,-.5) $) {$2$};
\end{tikzpicture}
\hspace{3 mm}
&
\begin{tikzpicture}
\coordinate (A) at (.5,1.5);
\coordinate (B) at (-.5,-1.5);
\coordinate (C) at (-.1,0);
 \node (black) at (0,1) [circle,draw,fill=black,inner sep=0pt,minimum size=1.5mm] {};
 \node (white) at (0,-1) [circle,draw,inner sep=0pt,minimum size=1.5mm] {};
\draw[->,thick] (white) -- (black);
\draw[->,thick] (C) to [out=0, in=0,looseness=1.5] (white);
\draw[->,thick] (black) to [out=180, in=180,looseness=1.5] (C);
\draw[->,thick, densely dotted]  ($ (black)+(0,.5) $)--(black);
\draw[->,thick, densely dotted]   (black)--($ (black)+(.5,0) $);
\draw[->,thick, densely dotted]  (white)--($ (white)-(0,.5) $);
\draw[->,thick, densely dotted]   ($ (white)-(.5,0) $)--(white);
\node (label) at ($ (white)+(.5,-.5) $) {$3$};
\end{tikzpicture}
\hspace{3 mm}
&
\begin{tikzpicture}
\coordinate (A) at (.5,1.5);
\coordinate (B) at (-.5,-1.5);
\coordinate (C) at (-.1,0);
 \node (black) at (0,1) [circle,draw,fill=black,inner sep=0pt,minimum size=1.5mm] {};
   \node (white) at (0,-1) [circle,draw,inner sep=0pt,minimum size=1.5mm] {};
\draw[->,thick] (white) -- (black);
\draw[->,thick, densely dotted]  ($ (black)+(0,.5) $)--(black);
\draw[->,thick, densely dotted]   (black)--($ (black)+(.5,0) $);
\draw[->,thick, densely dotted]  (white)--($ (white)-(0,.5) $);
\draw[->,thick, densely dotted]   ($ (white)-(.5,0) $)--(white);
\draw[->,thick, densely dotted]   (black)--($ (black)-(.5,0) $);
\draw[->,thick, densely dotted]   ($ (white)+(.5,0) $)--(white);
\node (label) at ($ (white)+(.5,-.5) $) {$4$};
\end{tikzpicture}
\hspace{3 mm}
&
\begin{tikzpicture}
\coordinate (A) at (-.5,-.4);
\coordinate (B) at (.5,-.1);
\coordinate (C) at (1,.9);
\coordinate (D) at (1,-.9);
 \node (black) at (0,1) [circle,draw,fill=black,inner sep=0pt,minimum size=1.5mm] {};
   \node (white) at (0,-1) [circle,draw,inner sep=0pt,minimum size=1.5mm] {};
\draw[->,thick] (black) to [out=270,in=90,looseness=1.5] (white);
\draw[->,thick] (black) to [out=90,in=100,looseness=1.5] (C);
\draw[->,thick] (C) to [out=280,in=80,looseness=1.0] (D);
\draw[->,thick] (D) to [out=260,in=270,looseness=1.5] (white);
\draw[->,thick, densely dotted]  ($ (black)+(.5,0) $)--(black);
\draw[->,thick, densely dotted]  ($ (black)-(.5,0) $)--(black);
\draw[->,thick, densely dotted]  (white)--($ (white)+(.5,0) $);
\draw[->,thick, densely dotted]  (white)--($ (white)-(.5,0) $);
\node (label) at ($ (white)+(-.5,-.5) $) {$5$};
\end{tikzpicture}
\\
\end{tabular}
\caption[Some non-maximal horizontal data diagrams.]{These figures represent some non-maximal horizontal data diagrams in
  $\HH(1,1)$.}
\label{fig: nonmaximal}
\end{figure}
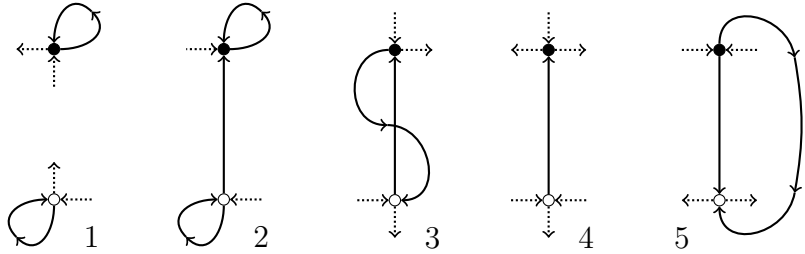

We will sometimes need an extension of $\Xi(M)$. Recall that at each 
$\xi_i \in \Sigma$ there are
$2(a_i+1)$ prongs. Some of these are part of saddle
connections in $\Xi(M)$, and we call the others the {\em unoccupied
  prongs} of $\Xi(M)$. Let $\Xi_-(M)$  denote
the graph which contains $\Xi(M)$ and has one additional {\em prong edge} for
every unoccupied prong, with one vertex at a singular point and one
vertex at a point of $M \sm \Sigma$, where the edge is realized by a
path in $M$ which lies on the horizontal separatrix issuing at the
corresponding prong. \index{X@$\Xi_-$}

We will sometimes need to be specific about the length of edges. In
that case, given $L>0$, we will let $\Xi_-(M, L)$
denote the graph $\Xi_-(M)$ described above, where the prong edges all
have length $L$. Note that since all of 
the prong edges are part of an infinite separatrix in the
horizontal direction, the graph $\Xi_-(M,L)$ can be
embedded in $M$ for any $L$. 

\section{An explicit surgery for real Rel}\label{subsec: explicit}
In this section, for fixed $M \in \HH$ and $z \in Z^{(M)}$, we will
present an explicit presentation of $\rel_zM$ in terms of glued
polygons. Our explicit surgery generalizes the special cases
treated  in \cite{McMullen-twists} and \cite[\S 3]{Matt}. 
See also the general result in \cite{MW}. 

As a
by-product, it will enable us to determine $Z^{(M)}$ explicitly from
the geometry of $M$. This analysis makes it possible to analyze limits $\lim_{j\to \infty}
\rel_{z_j} M$, for $z_j \in Z^{(M)}$ with $\lim_{j \to \infty} z_j \in \partial
Z^{(M)}$. Such limits do not exist as elements of $\HH$; loosely
speaking they belong to a bordification of $\HH$ obtained by adjoining boundary
strata. We will not construct this bordification in this paper but
hope to return to it in future work. A particularly simple case of this
bordification arises when one takes limits of surfaces in $\HH(1,1)$
for which a segment joining the two singularities collapses. Even
in this relatively simple case, which will arise in Theorem \ref{thm: criterion1},
the construction we use differs from earlier related work (see
\cite{KZ, EMZ}) and leads naturally to the use of framed
surfaces. 

We will establish some terminology and make a construction which will be used in stating Theorem 
\ref{thm: Bainbridge-Smillie surgery}. Let $L$ and $\vre$ be positive numbers. 
Given $M$, we let $\Xi(M), \, \Xi_-(M)$ and $ \Xi_-(M,L)$ 
be as above. We say that $\mathcal{N} = \mathcal{N}( L,
  \vre) \subset M$ is the {\em $(L, \vre)$-rectangle  thickening of
    $\Xi_-(M)$} if $\mathcal{N}$ is a union of
  rectangles $R_e^+$ and $R_e^-$ in $M$, with sides parallel to the
  coordinate axes, where $e$ ranges over the edges of
  $\Xi_-(M, L)$, $R_e$ has vertical sides of length $\vre$, and the
  edge $e$ runs along the bottom of $R_e^+$ and the top of
  $R_e^-$. \index{N@$\mathcal{N}$} Here the words `bottom' and `top' refer to the orientation
  provided by the translation surface structure and need not
  correspond to the directions shown in our figures.
  The edge identification  maps of the rectangles are
  inherited from $\Xi_-(M)$, namely they are as
  follows. Each $R_e^-$ is attached to $R_e^+$ along $e$, and the
  bottom of $R_e^-$ (resp. top of $R_e^+$) is unattached. 
If the right
  end of $e$ is not a singularity (i.e. $e$ is right-pointing prong
  edge) then the right hand boundaries of $R^\pm_e$ are unattached. If
  the right end of $e$ is the singularity $\xi$ then the right hand
  boundary of $R_e^-$ (resp. $R_e^+$) is attached to the left hand boundary of
  $R_f^-$, where $f$ is the edge
  of $\Xi_-(M)$ which is counterclockwise (resp. clockwise) from $e$
  at $\xi$. For the gluing rule for the left edges, replace left with
  right and clockwise with counterclockwise in the above description. 
See Figure \ref{fig:
  rectangle thickening}. 

Note that for any $L$ there is $\vre_0 = \vre_0(M,L)$ such
  that for all $\vre < \vre_0$,  the $(L, \vre)$-rectangle  thickening of
    $\Xi_-(M)$ exists and is embedded in $M$. 
Moreover the constant $\vre_0(M, L)$ can be
   chosen to be independent of $M$ and $L$, for $L$ in a
    bounded set of numbers   and 
    $M$ in a compact set of topologically horizontally equivalent surfaces.

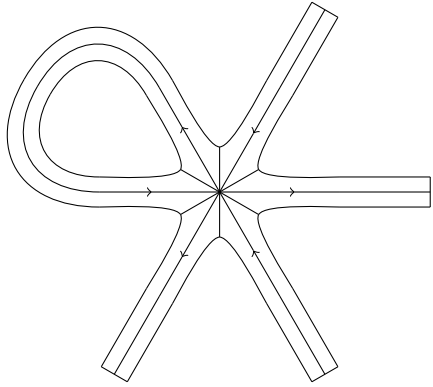
\begin{figure}[h]
\begin{tikzpicture}[scale=2]
\coordinate (center) at (0,0);
\coordinate (A0) at (0:0.5);
\coordinate (A1) at (60:0.45);
\coordinate (A2) at (120:0.5);
\coordinate (A3) at (180:0.45);
\coordinate (A4) at (240:0.5);
\coordinate (A5) at (300:0.45);

\coordinate (B0) at (0:0.8);
\coordinate (B1) at (60:0.8);
\coordinate (B2) at (120:0.8);
\coordinate (B3) at (180:0.8);
\coordinate (B4) at (240:0.8);
\coordinate (B5) at (300:0.8);

\coordinate (E0) at (0:1.4);
\coordinate (E1) at (60:1.4);
\coordinate (E2) at (120:0.8);
\coordinate (E3) at (180:0.8);
\coordinate (E4) at (240:1.4);
\coordinate (E5) at (300:1.4);

\coordinate (B0lft) at ($(B0)+(90:0.1)$);
\coordinate (B1lft) at ($(B1)+(150:0.1)$);
\coordinate (B2lft) at ($(B2)+(210:0.1)$);
\coordinate (B3lft) at ($(B3)+(270:0.1)$);
\coordinate (B4lft) at ($(B4)+(330:0.1)$);
\coordinate (B5lft) at ($(B5)+(390:0.1)$);

\coordinate (B0rt) at ($(B0)+(-90:0.1)$);
\coordinate (B1rt) at ($(B1)+(-30:0.1)$);
\coordinate (B2rt) at ($(B2)+(30:0.1)$);
\coordinate (B3rt) at ($(B3)+(90:0.1)$);
\coordinate (B4rt) at ($(B4)+(150:0.1)$);
\coordinate (B5rt) at ($(B5)+(210:0.1)$);

\coordinate (E0lft) at ($(E0)+(90:0.1)$);
\coordinate (E1lft) at ($(E1)+(150:0.1)$);
\coordinate (E2lft) at ($(E2)+(210:0.1)$);
\coordinate (E3lft) at ($(E3)+(270:0.1)$);
\coordinate (E4lft) at ($(E4)+(330:0.1)$);
\coordinate (E5lft) at ($(E5)+(390:0.1)$);

\coordinate (E0rt) at ($(E0)+(-90:0.1)$);
\coordinate (E1rt) at ($(E1)+(-30:0.1)$);
\coordinate (E2rt) at ($(E2)+(30:0.1)$);
\coordinate (E3rt) at ($(E3)+(90:0.1)$);
\coordinate (E4rt) at ($(E4)+(150:0.1)$);
\coordinate (E5rt) at ($(E5)+(210:0.1)$);

\draw (E0lft) to (E0rt);
\draw (E1lft) to (E1rt);
\draw (E4lft) to (E4rt);
\draw (E5lft) to (E5rt);

\draw (B0lft) to (E0lft);
\draw (B0rt) to (E0rt);
\draw (B0) to (E0);

\draw (B1lft) to (E1lft);
\draw (B1rt) to (E1rt);
\draw (B1) to (E1);

\draw (B4lft) to (E4lft);
\draw (B4rt) to (E4rt);
\draw (B4) to (E4);

\draw (B5lft) to (E5lft);
\draw (B5rt) to (E5rt);
\draw (B5) to (E5);

\draw[->] (center) to (A0);
\draw (A0) to (B0);

\draw (center) to (A1);
\draw[<-] (A1) to (B1);

\draw[->] (center) to (A2);
\draw (A2) to (B2);

\draw (center) to (A3);
\draw[<-] (A3) to (B3);

\draw[->] (center) to (A4);
\draw (A4) to (B4);

\draw (center) to (A5);
\draw[<-] (A5) to (B5);

\coordinate (C0) at (30:0.3);
\draw (center) to (C0);

\coordinate (C1) at (90:0.3);
\draw (center) to (C1);

\coordinate (C2) at (150:0.3);
\draw (center) to (C2);

\coordinate (C3) at (210:0.3);
\draw (center) to (C3);

\coordinate (C4) at (270:0.3);
\draw (center) to (C4);

\coordinate (C5) at (330:0.3);
\draw (center) to (C5);

\draw (B0lft) to [out=180,in=-60,looseness=0.4] (C0);
\draw (B1lft) to [out=240,in=0,looseness=0.4] (C1);
\draw (B2lft) to [out=300,in=60,looseness=0.4] (C2);
\draw (B3lft) to [out=0,in=120,looseness=0.4] (C3);
\draw (B4lft) to [out=60,in=180,looseness=0.4] (C4);
\draw (B5lft) to [out=120,in=240,looseness=0.4] (C5);

\draw (B1rt) to [out=240,in=120,looseness=0.4] (C0);
\draw (B2rt) to [out=300,in=180,looseness=0.4] (C1);
\draw (B3rt)to [out=0,in=240,looseness=0.4] (C2);
\draw (B4rt) to [out=60,in=300,looseness=0.4] (C3);
\draw (B5rt) to [out=120,in=0,looseness=0.4] (C4);
\draw (B0rt) to [out=180,in=60,looseness=0.4] (C5);

\coordinate (D3) at (150:1.4);
\coordinate (Dout) at ($(D3)+(150:0.1)$);
\coordinate (Din) at ($(D3)+(330:0.1)$);

\draw (B3rt) to [out=180,in=240,looseness=1.4] (Din);
\draw (B3lft) to [out=180,in=240,looseness=1.4] (Dout);
\draw (B2rt) to [out=120,in=60,looseness=1.4] (Dout);
\draw (B2lft) to [out=120,in=60,looseness=1.4] (Din);
\draw (B2) to [out=120,in=60,looseness=1.4] (D3);
\draw (B3) to [out=180,in=240,looseness=1.4](D3);

\end{tikzpicture}
\caption[A thickening of $\Xi_-(M)$]{A rectangle thickening of $\Xi_-(M)$ in $\HH(2)$ with 10 rectangles.}
\label{fig:
  rectangle thickening}
\end{figure}

Given translation surfaces $M, M'$, and a subset $M_0 \subset M$ we
say that {\em $M'$ can be obtained from $M$ by modifying 
  $M_0$} if there are polygon representations of $M$ and $M'$, and a
subset $M'_0 \subset M'$, such that $M_0, M'_0$ are 
unions of polygons, and there is a homeomorphism $f: (M \sm M_0)\to (M'
\sm M'_0)$, which is a translation in each coordinate chart of the
translation surface structures on $M$ and $M'$.

\begin{thm}\label{thm: Bainbridge-Smillie surgery}
Let $M \in \HH$. For each $\delta \in \Xi(M)$, let $Z^{(M, \delta)}$ denote
the connected component of $0$ in $\{z \in Z: \hol(M,
\delta)+z(\delta) \neq 0\}$.  
Then 
\begin{equation}\label{eq: ZM}{
Z^{(M)} = \bigcap_{\delta \in \Xi(M)} Z^{(M, \delta)}.
}\end{equation}
For any $z \in Z^{(M)}$ there is $L=L_z>0$, such that for any
$\vre < \vre_0(M, L)$, $\rel_zM$ can be obtained from $M$  by
modifying $\mathcal{N}(L, 
\vre)$. 
Moreover, the function $z \mapsto L_z$ can be taken to be bounded when 
$z$ varies in
a bounded subset of $Z^{(M)}$. 
\end{thm}

\begin{proof}
For $\delta \in \Xi(M)$, we have $\hol(M,\delta) = (x,0)$ for some
nonzero $x
\in \R$, and in this proof we change our notation slightly and denote
this number  $x$ by
$\hol(M, \delta)$. 

First note that $Z^{(M)} \subset \bigcap_{\delta \in \Xi(M)} Z^{(M,
  \delta)}$. Indeed, if $z \in Z^{(M)}$ then the straight line path
$\{\rel_{tz}M : t \in [0,1]\}$ is embedded in $\HH$. In particular,
for any saddle connection $\delta \in \Xi(M)$ and any $t \in [0,1]$,
$$\hol(M, \delta) + t z(\delta) = \hol(\rel_{tz}M,
\delta) \neq 0.$$  
So the path $\{tz: t \in [0,1]\}$ is contained in $Z^{(M, \delta)}$ and
  hence $z \in Z^{(M, \delta)}$. 

Now let $z \in \bigcap_{\delta \in \Xi(M)} Z^{(M, \delta)}$. 
Recall that we may think of an element $z$ 
   of $Z \cong H^0(\Sigma;\R)/H^0(S; \R)$ as
  being represented by a function $\bar{z}: \Sigma \to \R$. All such
  representatives $\bar{z}$ differ by constants, and 
\begin{equation}\label{eq: differ by constants}{
z_{ij} = \bar{z}(\xi_j) -
  \bar{z}(\xi_i) = z(\delta_{ij}),
}\end{equation}
for any path $\delta_{ij}: [0,1] \to
  S$ with $ \delta_{ij}(0) = \xi_i$, $\delta_{ij}(1) = \xi_j.$ We fix
  one representative $\bar{z}$ of $z$, and 
let 
\begin{equation}\label{eq: cond on L}
{L> \max_{i=1, \ldots, k} |\bar{z}(\xi_{i})|.
}\end{equation}
Our
  assumption about $z$ implies that 
\begin{equation}\label{eq: cond on z}{
\text{if } \delta 
 \text{ is from } \xi_i
\text{  to } \xi_j \text{ and } \hol(M,
\delta) >0 \text{, then } \hol(M, \delta) +z_{ij}>0.
}\end{equation}

Let $\vre < \vre_0(M, L)$ and let $\mathcal{N}_0 = \mathcal{N}(L,
\vre)$. Choose a marking $(f,M) \in \HH_{\mathrm{m}}$ of $M$. For each $t \in [0,1]$, we
will define a translation surface $M_t$ and a marking $f_t: S \to
M_t$, as follows. To each rectangle $R=R^\pm_e$ in $\mathcal{N}_0$ we
define a trapezoid $R_t=R^{\pm}_{t,e}$, by choosing a plane
development $\bar{R}$ of $R$ and moving the vertices of $\bar{R}$ horizontally, where
points of $\bar{R}$ which come from $\partial \mathcal{N}_0$
are not moved, and points which correspond to the singularity $\xi_i$
are moved by adding $t\bar{z}(\xi_{i})$ to their horizontal
component. See Figures \ref{fig: top correct} and \ref{fig: deformations}.

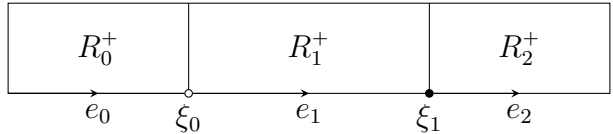
 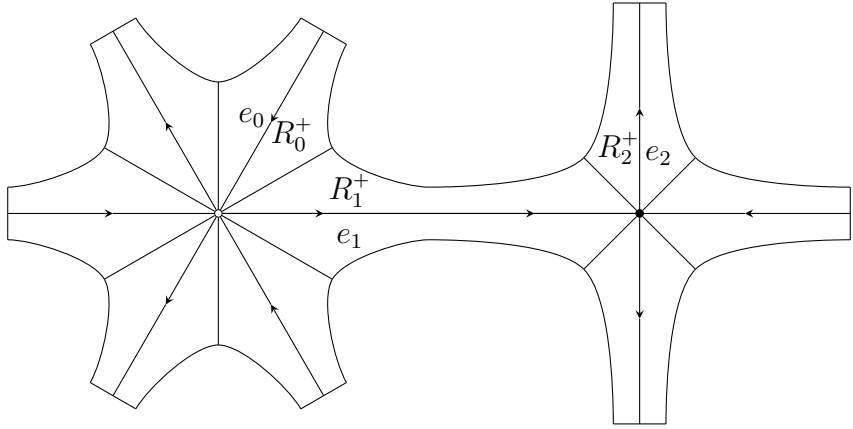
\begin{figure}[h]
\begin{tikzpicture}[scale=3.5,>=stealth]

\def\spread{0.8};

\node (lcenter) at (-\spread,0) [circle,draw,fill=white,inner sep=0pt,minimum size=1.0mm] {};

\node (rcenter) at (\spread,0) [circle,draw,fill=black,inner sep=0pt,minimum size=1.0mm] {};

\draw (lcenter) -- (rcenter);

\def\loose1{0.6}

\foreach \x in {0,60, ..., 300}
{
\draw (lcenter)-- +(\x+30:0.5);
\draw ($(lcenter)+(\x:0.8)+(\x+90:0.1)$) to [out=\x+180,in=\x-60,looseness=\loose1] ($(lcenter)+(\x+30:0.5)$);
\draw ($(lcenter)+(\x:0.8)+(\x-90:0.1)$) to [out=\x+180,in=\x+60,looseness=\loose1] ($(lcenter)+(\x-30:0.5)$);
}

\foreach \x in {0,120,240}
{
\draw[->] (lcenter)-- ++(\x:0.4);
\draw ($(lcenter)+(\x:0.4)$)--($(lcenter)+(\x:0.8)$);
}

\foreach \x in {60,180,300}
{
\draw (lcenter)-- ++(\x:0.4);
\draw[<-] ($(lcenter)+(\x:0.4)$)--($(lcenter)+(\x:0.8)$);
}

\foreach \x in {60,120, ..., 300}
\draw (lcenter)++(\x:0.8)+(\x+90:0.1) -- +(\x-90:0.1);

\node[anchor=north east] at ($(lcenter)+(60:0.6)+(-0.08,-0.08)$) {$e_0$};
\node[anchor=north] at (-0.3,-0.02) {$e_1$};
\node[anchor=west] at ($(rcenter)+(90:0.4)+(-0.02,-0.18)$) {$e_2$};

\node at (-0.52,0.3) {$R_0^+$};
\node at (-0.3,0.09) {$R_1^+$};
\node at (0.72,0.25) {$R_2^+$};

\def\loose2{0.6}
\def\short2{0.3}

\foreach \x in {0, 90,...,270}
{
\draw (rcenter)-- +(\x+45:\short2);
\draw ($(rcenter)+(\x:0.8)+(\x+90:0.1)$) to [out=\x+180,in=\x-45,looseness=\loose2] ($(rcenter)+(\x+45:\short2)$);
\draw ($(rcenter)+(\x:0.8)+(\x-90:0.1)$) to [out=\x+180,in=\x+45,looseness=\loose2] ($(rcenter)+(\x-45:\short2)$);
}

\foreach \x in {90,270}
{
\draw[->] (rcenter)-- ++(\x:0.4);
\draw ($(rcenter)+(\x:0.4)$)--($(rcenter)+(\x:0.8)$);
}

\foreach \x in {0,180}
{
\draw (rcenter)-- ++(\x:0.4);
\draw[<-] ($(rcenter)+(\x:0.4)$)--($(rcenter)+(\x:0.8)$);
}

\foreach \x in {0,90,270}
\draw (rcenter)++(\x:0.8)+(\x+90:0.1) -- +(\x-90:0.1);

\end{tikzpicture}
\caption[The complex $\mathcal{N}_0$]{The complex $\mathcal{N}_0$, presented topologically, showing
  three adjacent rectangles. The rectangles are presented with correct 
  geometries below.}
\label{fig: top correct}
\end{figure}

 \begin{figure}[h]
\begin{tikzpicture}[scale=2,>=stealth]

\def\spread{0.8};
\def\h{0.6}

\node (lcenter) at (-0.8,0) [circle,draw,fill=white,inner sep=0pt,minimum size=1.0mm] {};

\node (rcenter) at (0.8,0) [circle,draw,fill=black,inner sep=0pt,minimum size=1.0mm] {};

\coordinate (A) at (-2,0);
\coordinate (B) at (-2,\h);
\coordinate (C) at (-0.8,\h);
\coordinate (D) at (0.8,\h);
\coordinate (E) at (2,\h);
\coordinate (F) at (2,0);
+
\draw (A) -- (lcenter) -- (rcenter) -- (F) -- (E) -- (D) -- (C) -- (B)
-- (A) -- (lcenter);
\draw (lcenter) -- (C);
\draw (rcenter) -- (D);

\draw[->] (A)--+(0.6,0);
\draw[->] (lcenter)--+(0.8,0);
\draw[->] (rcenter)--+(0.6,0);

\node[anchor=north] (mid0) at ($(lcenter)!0.5!(A)$) {$e_0$};

\node[anchor=north] (mid1) at ($(lcenter)!0.5!(rcenter)$) {$e_1$};

\node[anchor=north] (mid2) at ($(rcenter)!0.5!(F)$) {$e_2$};

\node[anchor=north] at (lcenter) {$\xi_0$};

\node[anchor=north] at (rcenter) {$\xi_1$};

\node at (-1.4,0.3) {$R_0^+$};
\node at (0,0.3) {$R_1^+$};
\node at (1.4,0.3) {$R_2^+$};

\end{tikzpicture}

\end{figure}

 \begin{figure}[h]
\begin{tikzpicture}[scale=2]

\def\spread{0.8};
\def\h{0.6}

\node (lcenter) at (-1.2,0) [circle,draw,fill=white,inner sep=0pt,minimum size=1.0mm] {};

\node (rcenter) at (1.0,0) [circle,draw,fill=black,inner sep=0pt,minimum size=1.0mm] {};

\coordinate (A) at (-2,0);
\coordinate (B) at (-2,\h);
\coordinate (C) at (-0.8,\h);
\coordinate (D) at (0.8,\h);
\coordinate (E) at (2,\h);
\coordinate (F) at (2,0);

\draw (A) -- (lcenter) -- (rcenter) -- (F) -- (E) -- (D) -- (C) -- (B)
-- (A) -- (lcenter);
\draw (lcenter) -- (C);
\draw (rcenter) -- (D);

\node[anchor=north] (mid0) at ($(lcenter)!0.5!(A)$) {$e_0$};

\node[anchor=north] (mid1) at ($(lcenter)!0.5!(rcenter)$) {$e_1$};

\node[anchor=north] (mid2) at ($(rcenter)!0.5!(F)$) {$e_2$};

\node[anchor=north] at (lcenter) {$\xi_0$};

\node[anchor=north] at (rcenter) {$\xi_1$};

\node at (-1.5,0.3) {$R_0^+$};
\node at (-0.1,0.3) {$R_1^+$};
\node at (1.4,0.3) {$R_2^+$};

\end{tikzpicture}

\end{figure}

 \begin{figure}[h]
\begin{tikzpicture}[scale=2]

\def\spread{0.8};
\def\h{0.6}

\node (lcenter) at (-0.4,0) [circle,draw,fill=white,inner sep=0pt,minimum size=1.0mm] {};

\node (rcenter) at (0.8,0) [circle,draw,fill=black,inner sep=0pt,minimum size=1.0mm] {};

\coordinate (A) at (-2,0);
\coordinate (B) at (-2,\h);
\coordinate (C) at (-0.8,\h);
\coordinate (D) at (0.8,\h);
\coordinate (E) at (2,\h);
\coordinate (F) at (2,0);

\draw (A) -- (lcenter) -- (rcenter) -- (F) -- (E) -- (D) -- (C) -- (B)
-- (A) -- (lcenter);
\draw (lcenter) -- (C);
\draw (rcenter) -- (D);

\node[anchor=north] (mid0) at ($(lcenter)!0.5!(A)$) {$e_0$};

\node[anchor=north] (mid1) at ($(lcenter)!0.5!(rcenter)$) {$e_1$};

\node[anchor=north] (mid2) at ($(rcenter)!0.5!(F)$) {$e_2$};

\node[anchor=north] at (lcenter) {$\xi_0$};

\node[anchor=north] at (rcenter) {$\xi_1$};

\node at (-1.3,0.3) {$R_0^+$};
\node at (0.1,0.3) {$R_1^+$};
\node at (1.4,0.3) {$R_2^+$};

\end{tikzpicture}
\caption[Different deformations]{The effect of different deformations, transforming the
rectangles into trapezoids.}
\label{fig: deformations}

\end{figure}
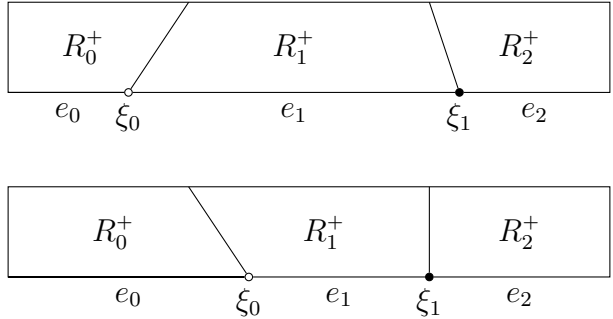

Note that for any $t \in [0,1]$, the lengths of edges of $R_t$ do not
vanish. Indeed,  $R^\pm_e$ has either one or two vertices which are in
$\Sigma$, depending on whether $e$ is a prong edge or an edge of
$\Xi(M)$. In case it is a prong edge, the length of horizontal sides
of $R$ is $L$
and \eqref{eq: cond on L} implies that the sidelength of $R_t$ is
positive, and in case it is an edge of $\Xi(M)$, \eqref{eq: cond on z}
implies that the sidelength of $R_t$ 
is positive. We glue the different trapezoids $\{R_t\}$ to each other
along their edges, using 
the same gluing that defines $\mathcal{N}_0$, to obtain a complex
$\mathcal{N}_t$. Since the sides of $R$ corresponding to 
$\partial \mathcal{N}_0$ have the same length in $R_t$, 
 their boundaries $\partial \mathcal{N}_t$ and $\partial \mathcal{N}_0$ can be
identified by a translation, and so we can glue 
$\mathcal{N}_t$ to $M \sm \mathcal{N}_0$ along
$\partial \mathcal{N}_0$. We denote the resulting translation surface
by $M_t$; clearly it is obtained from $M$ by modifying
$\mathcal{N}_0$. 

 On each rectangle $R$ we choose a homeomorphism
$\bar{f}_t : R \to R_t$ which sends vertices of $R$ to the
corresponding vertices of $R_t$, and acts affinely on each boundary
edge of $\partial R_t$. This choice ensures that $\bar{f}_t$ can be
extended consistently from rectangles to their union, defining a map $\bar{f}_t:
\mathcal{N}_0 \to \mathcal{N}_t$, and then extended to a
homeomorphism $\bar{f}_t: M \to M_t$. We set $f_t = \bar{f}_t \circ
f$. With this choice $(f_t, M_t)$ is a path in $ \HH_{\mathrm{m}}$, 
and the maps $f_t \circ f^{-1}$ are isotopic to the
identity via an isotopy fixing $\Sigma$. 

We claim that for each $t$, the pullback $\dev(f_t, M_t) = f^*_t \hol(M_t, \cdot)$ is
the cohomology class $\dev(f, M)+tz$. This will imply that $M_t =
\rel_{tz}(M)$ for all $t \in [0,1]$, which implies all the required
conclusions of the Theorem. We verify this formula on each
path $\gamma: [0,1] \to M$ between singularities $\xi_i$ and $\xi_j$. The path $\gamma$ is
homotopic to a concatenation of linear segments
$\delta_1, \ldots, \delta_\ell$ which begin and end at singular points
and are completely contained in
$\mathcal{N}_0$, 
and linear segments $\delta'_1, \ldots, \delta'_{m}$ which begin and end at
singular points and are not completely
contained in
$\mathcal{N}_0$.
Each of the segments $\delta = \delta_i$ is a
saddle connection in $\Xi(M)$, thus 
it is
one of the edges $e$ of the rectangles $R_e$, and by construction
$\hol(f^*_t M_t, \delta) = \hol(f^*M, \delta)+t z(\delta).$ If
$\delta' = \delta'_\ell$ is a path not contained in $\mathcal{N}_0$ we can subdivide
it into a concatention of paths $\sigma_1, \sigma_2, \sigma_3,
\sigma_4,$ where $\sigma_1$ goes from $\xi_1$ to $\partial
\mathcal{N}_0$, $\sigma_2$ goes from $\partial \mathcal{N}_0$ to
$\xi_j$, $\sigma_3$ (resp. $\sigma_4$) is a union of segments with
interior completely inside
(resp. outside) $\mathcal{N}$ from $\partial
\mathcal{N}_0$ to $\partial
\mathcal{N}_0$. Moreover by applying a homotopy we
can assume that on the initial surface $M$, each of the segments in  $\sigma_1$,
$\sigma_2$ and $\sigma_3$ proceeds along a vertical
line in the rectangles $R^{\pm}_e$. Now we compute the difference 
$\hol(f^*_t M_t, \sigma) -\hol(f^*M, \sigma)$ in each
case. We have 
$\hol(f^*_t M_t, \sigma_1) -\hol(f^*M, \sigma_1) = -\bar{z}(\xi_i)$
and $\hol(f^*_t M_t, \sigma_2) -\hol(f^*M, \sigma_2) =
\bar{z}(\xi_j)$ by definition of the deformed flat structure on
$R_{t,e}^{\pm}$. We also have 
$\hol(f^*_t M_t, \sigma_4) -\hol(f^*M, \sigma_4) =0$, since
$\sigma_4$ is in the complement of $\mathcal{N}_0$ where the two flat
structures are the same, and we have 
$\hol(f^*_t M_t, \sigma_3) -\hol(f^*M, \sigma_3) = 0$ since each of
the segments of 
$\sigma_3$ passes through both $R_e^+$ and
$R_e^-$ for some $e$, leaving and exiting at symmetric points, and the
change to the holonomy in these two 
rectangles cancel each other. All together we have $\hol (f^*M,
\delta') - \hol(f^*_tM_t, 
\delta')=\bar{z}(\xi_j) - \bar{z}(\xi_i)$, as required. 
\end{proof}
 
The following result says that the only obstruction to defining the $\rel$
flow is the one illustrated in Figure~\ref{fig: dec1}. 
Let $z_{ij}$ be defined as in equation \eqref{eq: differ by constants}.

\begin{cor}\label{thm: criterion} Suppose 
$Z$ is the real Rel subspace for a stratum $\HH$ as above and
let $z \in Z, \, M \in \HH$. Then $M\in\HH'_z$ exactly when
  there is no horizontal saddle connection $\delta$ on $M$ from
  singularity $\xi_i$ to singularity $\xi_j$, and $t \in [0,1]$ such
  that $\hol(M, \delta) + t z_{ij} =0$. In particular $\HH'_\infty$ is
  the set of surfaces which have no horizontal saddle 
connections joining distinct singularities. 

Furthermore, if $\HH$ has two singularities, and $v \in \mathfrak{R}
\cong \R^2$, then $M \notin \HH'_v$ if
and only if there is a saddle connection $\delta$ on $M$
from $\xi_2$ to singularity $\xi_1$, with $\hol(M, \delta) = tv$ for
some $t \in [0,1]$. 
\end{cor}

\begin{proof} The first assertion is a restatement of \eqref{eq:
    ZM}, and the second assertion follows immediately from the
  first. For the third assertion, let $r_\theta$ be the rotation matrix for
  which $r_\theta v$ is horizontal. We obtain the assertion by
  applying the first assertion to the surface $r_\theta M$ and using
  Proposition \ref{prop: distributive} with $g=r_\theta$.  
\end{proof}
Corollary \ref{thm: criterion} was proved in 
\cite[Thm. 11.2]{MW}. See also \cite{McMullen-twists}.

Figures~\ref{fig: dec1} 
and \ref{fig: tdec1} illustrate the meaning of Corollary~\ref{thm:
  criterion}. In Figure~\ref{fig: dec1} 
the saddle connection at the top and bottom of the decagon violates
the first condition, when $v = 
(T,0)$ for $T$ which is at least as large as the length of this
segment. In Figure~\ref{fig: tdec1} there are no horizontal saddle 
connections joining distinct singularities, and as a consequence of
Corollary~\ref{thm: criterion},  
$\rel_{(T,0)}M$ is defined for all $T$.

We now derive some consequences. As a first consequence we have:

\begin{proof}[Proof of Proposition \ref{prop: convex}]
Each $Z^{(M, \delta)}$ is a half-space, and in particular is
convex. Thus Proposition \ref{prop: convex} follows immediately from
\eqref{eq: ZM}. 
\end{proof}

An immediate consequence of the explicit surgery we have presented in
the proof of 
\ref{thm: Bainbridge-Smillie surgery} is the following useful statement:

\begin{cor}\label{prop: horizontal preserved}
For any $M$ and any $t$ for which $\rel_tM$ is defined, 
there is a natural bijection between horizontal saddle connections on
$M$ and on $\rel_t M$, and for each saddle connection $\delta$ directed from
$\xi_2$ to $\xi_1$, $\hol(\rel_tM, \delta) = \hol (M, \delta)- (t,0).$
In particular $M$ and $\rel_tM$ are topologically horizontally
equivalent. 
\end{cor}

The following will be crucial for analyzing $U$-invariant measures in
\S \ref{section: construction}. 

\begin{dfn}\label{def: defn}
Suppose $\HH = \HH(1,1)$, and let $T \in
\R \sm\{0\} $. 
  Let $\HH''_T$ \index{H@$\HH''_T$} denote the set of surfaces $M \in \HH$ for which 
\begin{itemize}
\item[(i)] $M$ contains exactly one directed saddle
  connection $\delta'$  from $\xi_2$ to $\xi_1$ 
 with $\hol(M, \delta') = (T,0)$; 
\item[(ii)] $M$ contains no directed saddle connection $\delta$ from 
 $\xi_2$ to $\xi_1$, such that $\hol(M, \delta) = (c,0)$ with $c$
 between $0$ and $T$. 
\end{itemize}
\end{dfn}

\begin{thm}\label{thm: criterion1} There is a  map 
$$
\Phi: \HH''_T \to \HH(2)
$$
which is affine in charts (hence continuous) and $U$-equivariant. For
each $M \in \HH''_T$, $\Phi(M)$ is obtained by modifying
$\mathcal{N}(\vre, L)$ for some $\vre>0, L>T$ depending on $M$. There
is a map 
$$
\Phi_{\mathrm{f}} = \Phi_{\mathrm{f}}^{(T)} : \HH''_T \to \HH_{\mathrm{f}}(2),
$$
which is a lift of $\Phi$ (i.e. $\Phi = P \circ \Phi_{\mathrm{f}}$ where $P: \HH_{\mathrm{f}}(2) \to \HH(2)$
is the projection), and $\Phi_{\mathrm{f}}$ is a homeomorphism onto its image.
\end{thm}

Suppose $T>0$. 
Note that assumption (ii) implies that $[0,T) \subset Z^{(M)}$, and
assumption (i) implies that $T \notin Z^{(M)}$; that is $T
\in \partial Z^{(M)}$. A topology on $\HH(1,1) \cup
\HH_{\mathrm{f}}(2)$ can be constructed in which the map $\Phi$ can be
recovered  as $\Phi(M) =
\lim_{s
  \to T-} \rel_s(M). $ We will not construct this topology here. 

\begin{proof} We will assume throughout the proof that $T>0$. The case in which
  $T<0$ can be dealt with by repeating the arguments below, switching
  the labels of the
  two singularities. 
Let $\delta'$ be as in (i) in the definition of $\HH''_T$ and let
$e'$ be the corresponding edge of $\Xi_-(M)$. Let $L>|T|$ and define  
$\mathcal{N}(L, \vre)$ as in the discussion prior to the statement of 
Theorem \ref{thm: Bainbridge-Smillie surgery},  where $\vre>0$ is small enough so
that the rectangles $R_e^\pm$ are all embedded in $M$. Define the polygons
$R_{T,e}^\pm$ as in the proof of Theorem \ref{thm: 
  Bainbridge-Smillie surgery}.  For all $e \neq e'$ condition (ii)
ensures that $R_{T,e}^\pm$ is a nondegenerate trapezoid. For $e'$,
the length of the edge corresponding to $e'$ is zero and so it is a
triangle. See  Figure \ref{fig: 2.1}. We glue the two triangles $R^\pm_{T, e'}$ to the
trapezoids $R^\pm_{T,e}, e \neq e'$ to obtain the complex
$\mathcal{N}_{T}$ which we glue to $M \sm \mathcal{N}_0$ as before, to
obtain $M_T$. We note that $\mathcal{N}_T$ is a thickening of a graph obtained from
$\Xi_-(M,L)$ by collapsing the edge $e'$. 
One can compute explicitly
that there is one singular point in $\mathcal{N}_T$
and that it has cone angle $6\pi$; that is $M_T \in \HH(2)$. We set
$\Phi(M)=M_T$.

 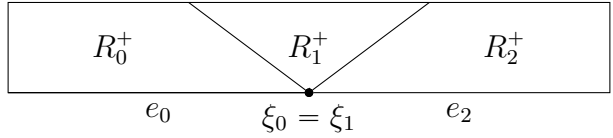
\begin{figure}[h]
\begin{tikzpicture}[scale=2]

\def\spread{0.8};
\def\h{0.6}

\node (lcenter) at (0,0) [circle,draw,fill=white,inner sep=0pt,minimum size=1.0mm] {};

\node (rcenter) at (0,0) [circle,draw,fill=black,inner sep=0pt,minimum size=1.0mm] {};

\coordinate (A) at (-2,0);
\coordinate (B) at (-2,\h);
\coordinate (C) at (-0.8,\h);
\coordinate (D) at (0.8,\h);
\coordinate (E) at (2,\h);
\coordinate (F) at (2,0);

\draw (A) -- (lcenter) -- (rcenter) -- (F) -- (E) -- (D) -- (C) -- (B)
-- (A) -- (lcenter);
\draw (lcenter) -- (C);
\draw (rcenter) -- (D);

\node[anchor=north] (mid0) at ($(lcenter)!0.5!(A)$) {$e_0$};

\node[anchor=north] (mid2) at ($(rcenter)!0.5!(F)$) {$e_2$};

\node[anchor=north] at (lcenter) {$\xi_0=\xi_1$};

\node at (-1.3,0.3) {$R_0^+$};
\node at (0.0,0.3) {$R_1^+$};
\node at (1.3,0.3) {$R_2^+$};

\end{tikzpicture}
\caption[The developing image of degenerate rectangles]{The developing image of degenerate rectangles. The trapezoid 
  $R_1^+$ has degenerated to a triangle, the edge $e'=e_1$ has
  disappeared, and the two singular points
  have coalesced.}
\label{fig: 2.1}
\end{figure}

 
  \begin{figure}[h]
\begin{tikzpicture}[scale=3.0,>=stealth]

\def\spread{0.8};

\node (lcenter) at (-\spread,0) [circle,draw,fill=white,inner sep=0pt,minimum size=1.0mm] {};

\node (rcenter) at (\spread,0) [circle,draw,fill=black,inner sep=0pt,minimum size=1.0mm] {};

\draw (lcenter) -- (rcenter);

\def\loose1{0.6}
\def\short1{0.4}
\def\endlength{0.2}

\foreach \x in {0, 90,...,270}
{

\draw (lcenter)-- +(\x+45:\short1);
\draw ($(lcenter)+(\x:0.8)+(\x+90:\endlength)$) to [out=\x+180,in=\x-45,looseness=\loose1] ($(lcenter)+(\x+45:\short1)$);
\draw ($(lcenter)+(\x:0.8)+(\x-90:\endlength)$) to [out=\x+180,in=\x+45,looseness=\loose1] ($(lcenter)+(\x-45:\short1)$);
}

\foreach \x in {0, 180}
{
\draw[->] (lcenter)-- ++(\x:0.4);
\draw ($(lcenter)+(\x:0.4)$)--($(lcenter)+(\x:0.8)$);
}

\foreach \x in {90, 270}
{
\draw (lcenter)-- ++(\x:0.4);
\draw[<-] ($(lcenter)+(\x:0.4)$)--($(lcenter)+(\x:0.8)$);
}

\foreach \x in {90,180,270}
\draw (lcenter)++(\x:0.8)+(\x+90:\endlength) -- +(\x-90:\endlength);


\def\loose2{0.6}
\def\short2{0.4}

\foreach \x in {0, 90,...,270}
{
\draw (rcenter)-- +(\x+45:\short2);
\draw ($(rcenter)+(\x:0.8)+(\x+90:\endlength)$) to [out=\x+180,in=\x-45,looseness=\loose2] ($(rcenter)+(\x+45:\short2)$);
\draw ($(rcenter)+(\x:0.8)+(\x-90:\endlength)$) to [out=\x+180,in=\x+45,looseness=\loose2] ($(rcenter)+(\x-45:\short2)$);
}

\foreach \x in {90, 270}
{
\draw[->] (rcenter)-- ++(\x:0.4);
\draw ($(rcenter)+(\x:0.4)$)--($(rcenter)+(\x:0.8)$);
}

\foreach \x in {0, 180}
{
\draw (rcenter)-- ++(\x:0.4);
\draw[<-] ($(rcenter)+(\x:0.4)$)--($(rcenter)+(\x:0.8)$);
}

\foreach \x in {0,90,270}
\draw (rcenter)++(\x:0.8)+(\x+90:\endlength) -- +(\x-90:\endlength);

\node at (-0.9,0.4) {$q_3$};
\node at (-0.67,0.4) {$R_3^+$};
\node at (-0.9,-0.4) {$q_2$};
\node at (-0.67,-0.4) {$R_2^-$};
\node at (0.9,0.4) {$q_4$};
\node at (0.67,0.4) {$R_4^+$};
\node at (0.9,-0.4) {$q_1$};
\node at (0.67,-0.4) {$R_1^-$};
\node at (0,-0.1) {$R_{\delta'}^-$};
\node at (0,0.1) {$R_{\delta'}^+$};
\node at (0.4,-0.1) {$\delta'$};

\end{tikzpicture}
\caption{The complex $\mathcal{N}_0$, shown topologically, with
  $\delta'$ connecting the two 
singularities.}
\label{fig: connecting}
\end{figure}
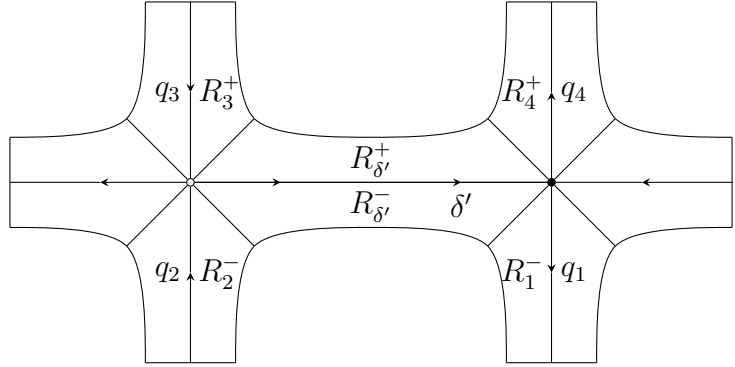

 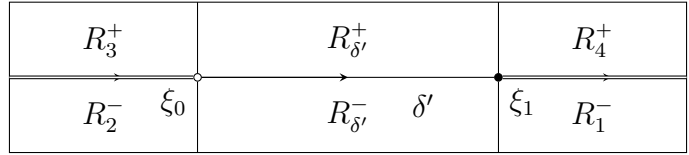
\begin{figure}[h]
\begin{tikzpicture}[scale=2.5,>=stealth]

\def\spread{0.8};
\def\h{0.4}

\node (lcenter) at (-0.8,0) [circle,draw,fill=white,inner sep=0pt,minimum size=1.0mm] {};

\node (rcenter) at (0.8,0) [circle,draw,fill=black,inner sep=0pt,minimum size=1.0mm] {};

\coordinate (A) at (-1.8,0);
\coordinate (B) at (-1.8,\h);
\coordinate (C) at (-0.8,\h);
\coordinate (D) at (0.8,\h);
\coordinate (E) at (1.8,\h);
\coordinate (F) at (1.8,0);

\draw (lcenter) -- (rcenter);
\draw (F) -- (E) -- (D) -- (C) -- (B)-- (A) ;
\draw (lcenter) -- (C);
\draw (rcenter) -- (D);

\coordinate (Bm) at (-1.8,-\h);
\coordinate (Cm) at (-0.8,-\h);
\coordinate (Dm) at (0.8,-\h);
\coordinate (Em) at (1.8,-\h);

\draw  (F) -- (Em) -- (Dm) -- (Cm) -- (Bm)
-- (A);
\draw (lcenter) -- (Cm);
\draw (rcenter) -- (Dm);

\node at (0.4,-0.15) {$\delta'$};

\draw[->] (A)--+(0.6,0);

\draw[double] (A) -- (lcenter);
\draw[->] (lcenter)--+(0.8,0);
\draw[->] (rcenter)--+(0.6,0);
\draw[double] (rcenter) -- (F);

\node[anchor=north east] at (lcenter) {$\xi_0$};
\node[anchor=north west] at (rcenter) {$\xi_1$};
\node at (-1.3,0.2) {$R_3^+$};
\node at (0,0.2) {$R_{\delta'}^+$};
\node at (1.3,0.2) {$R_4^+$};
\node at (-1.3,-0.2) {$R_2^-$};
\node at (0,-0.2) {$R_{\delta'}^-$};
\node at (1.3,-0.2) {$R_1^-$};

\end{tikzpicture}
\caption[The developing image of rectangles.]{The developing image of rectangles. The double lines represent
  two different edges on the surface.}
\label{fig: developing image}
\end{figure}


 \begin{figure}[h]
\begin{tikzpicture}[scale=2.5,>=stealth]

\def\spread{0.8};
\def\h{0.4}

\node (lcenter) at (0,0) [circle,draw,fill=white,inner sep=0pt,minimum size=1.0mm] {};

\node (rcenter) at (0,0) [circle,draw,fill=black,inner sep=0pt,minimum size=1.0mm] {};

\coordinate (A) at (-1.8,0);
\coordinate (B) at (-1.8,\h);
\coordinate (C) at (-0.8,\h);
\coordinate (D) at (0.8,\h);
\coordinate (E) at (1.8,\h);
\coordinate (F) at (1.8,0);

\draw (A) -- (lcenter) -- (rcenter) -- (F) -- (E) -- (D) -- (C) -- (B)
-- (A) -- (lcenter);
\draw (lcenter) -- (C);
\draw (rcenter) -- (D);

\coordinate (Bm) at (-1.8,-\h);
\coordinate (Cm) at (-0.8,-\h);
\coordinate (Dm) at (0.8,-\h);
\coordinate (Em) at (1.8,-\h);

\draw (A) -- (lcenter) -- (rcenter) -- (F) -- (Em) -- (Dm) -- (Cm) -- (Bm)
-- (A) -- (lcenter);
\draw (lcenter) -- (Cm);
\draw (rcenter) -- (Dm);


\node[anchor=north] at (lcenter) {$\xi$};

\draw[->] (A)--+(0.8,0);
\draw[->] (0,0)--+(0.8,0);

\node at (-1.3,0.2) {$R_3^+$};
\node at (0,0.2) {$R_{\delta'}^+$};
\node at (1.3,0.2) {$R_4^+$};
\node at (-1.3,-0.2) {$R_2^-$};
\node at (0.2,-0.25) {$R_{\delta'}^-$};
\node at (1.3,-0.2) {$R_1^-$};

\draw[double] (A) -- (lcenter);
\draw[double] (lcenter) -- (F);

\end{tikzpicture}
\caption[Coalescing singularities]{The developing image of rectangles when $\delta'$ shrinks to a
  point and the two singularities coalesce.} 

\end{figure}
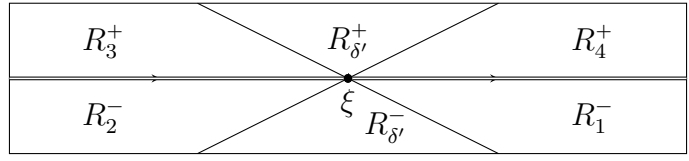
 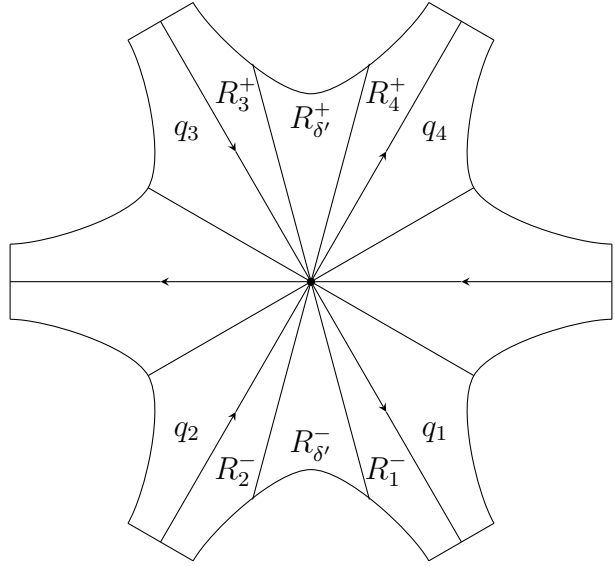
\begin{figure}[h]
\begin{tikzpicture}[scale=5.0,>=stealth]

\node (lcenter) at (0,0) [circle,draw,fill=black,inner sep=0pt,minimum size=1.0mm] {};

\def\loose1{0.6}

\foreach \x in {0,60, ..., 300}
{
\draw ($(lcenter)+(\x:0.8)+(\x+90:0.1)$) to [out=\x+180,in=\x-60,looseness=\loose1] ($(lcenter)+(\x+30:0.5)$);
\draw ($(lcenter)+(\x:0.8)+(\x-90:0.1)$) to [out=\x+180,in=\x+60,looseness=\loose1] ($(lcenter)+(\x-30:0.5)$);
\draw (lcenter)++(\x:0.8)+(\x+90:0.1) -- +(\x-90:0.1);
}
\foreach \x in {0,120,240}
{
\draw (lcenter)-- ++(\x:0.4);
\draw[<-] ($(lcenter)+(\x:0.4)$)--($(lcenter)+(\x:0.8)$);
}
\foreach \x in {60, 180, 300}
{
\draw[->] (lcenter)-- ++(\x:0.4);
\draw ($(lcenter)+(\x:0.4)$)--($(lcenter)+(\x:0.8)$);}

\foreach \x in {0,120,180, 300}
\draw (lcenter)-- +(\x+30:0.5);

\node at (-0.33,0.4) {$q_3$};
\node at (-0.33,-0.4) {$q_2$};
\node at (0.33,0.4) {$q_4$};
\node at (0.33,-0.4) {$q_1$};

\draw (lcenter)--(75:0.60);
\draw (lcenter)--(105:0.60);
\draw (lcenter)--(255:0.60);
\draw (lcenter)--(285:0.60);

\node at (0,0.43) {$R_{\delta'}^+$};
\node at (0.0,-0.43) {$R_{\delta'}^-$};
\node at (-0.2,0.5) {$R_3^+$};
\node at (0.2,0.5) {$R_4^+$};

\node at (-0.2,-0.5) {$R_2^-$};
\node at (0.2,-0.5) {$R_1^-$};

\end{tikzpicture}
\caption[Topologically correct picture.]{The corresponding topologically correct picture. The chosen prong edge
  is labeled $q_1$. } 
\label{fig: corresponding}
\end{figure}

We now show that $\Phi$ is affine in charts. We first explain what
this means. Let
$\HH_{\mathrm{m}, T}''$ be the pre-image of $\HH_T''$ in
$\HH_{\mathrm{m}}(1,1)$, and let $\dev_{(1,1)}: \HH_{\mathrm{m}}(1,1)
\to H^1(S, \{\xi_1, \xi_2\}; \R^2)$ be the developing map as in 
\eqref{eq: dev}. Condition (i) in Definition \ref{def: defn} can be expressed as a
linear condition on the image of $\dev_{(1,1)}$ and condition (ii) is
an open condition on the image of $\dev_{(1,1)}$, and thus
$\HH''_{\mathrm{m},T}$ is an affine submanifold of
$\HH_{\mathrm{m}}(1,1)$. Also let $\dev_{(2)} : \HH_{\mathrm{m}}(2)
\to H^1(S', \{\xi'\}; \R^2)$, for some model surface $S'$ of genus 2
with one distinguished point $\xi'$. To say that $\Phi$ is affine in
charts is to say that 
that there is a map $\Phi_{\mathrm{m}} : \HH''_{\mathrm{m},T} \to
\HH_{\mathrm{m}}(2)$, which is a lift of $\Phi$, and a linear map $L:
H^1(S, \{\xi_1, \xi_2\}; 
\R^2) \to H^1(S', \{\xi'\}; \R^2)$ such that 
\begin{equation}\label{eq: what we want}{
\dev_{(2)} \circ
\Phi_{\mathrm{m}} = L \circ \dev_{(1,1)}.} 
\end{equation}

Let $(f,M)
\in \HH_{\mathrm{m},T}''$ be a marked surface projecting to $M \in \HH''_T$, and
define the map $\bar{f}_T: M \to M_T$  as in 
the proof of  Theorem \ref{thm: 
  Bainbridge-Smillie surgery}. Let $f_T = f \circ \bar{f}_T: S \to M_T$. 
The
map $\bar{f}_T$ is injective on the complement of 
$e'$ and maps all points in $e'$, including its endpoints $\xi_1$ and
$\xi_2$, to the unique 
singularity of $M_T$, which we denote by $\xi$. 
Note that $f^{-1}(e')$ is a simple path connecting the 
two points $\xi_1 $ and $ \xi_2$. Let $S'$ be the
surface obtained from $S$ by collapsing  
$f^{-1}(e')$ to a point $\xi'$, and let $p:S\to S'$ be the quotient
map. Since $f^{-1}(e')$ is contractible, $S'$ is also a genus 2
surface, and since $f_T$ is constant on $f^{-1}(e')$, it descends to a
homeomorphism $S' \to M_T$, which we continue to denote by
$f_T$. We see that 
$$\Phi_{\mathrm{m}} : \HH_{\mathrm{m}, T}'' \to \HH_{\mathrm{m}}(2), \
\ \Phi_{\mathrm{m}}(f, M) = (f_T, M_T)$$ is a lift of $\Phi$. Since $f^{-1}(e')$ is contractible, 
the pullback  $p^*: H^1(S'; \R^2) \to H^1(S ; \R^2)$ is an
isomorphism. Let $\Res_{(1,1)}$ and $\Res_{(2)}$ denote the restriction maps
in \eqref{eq: exact}, in the two cases corresponding respectively to
$\HH(1,1)$ and $\HH(2)$. Since $\xi_1, \xi_2$ are contained in
$f^{-1}(e')$, and the holonomies of absolute periods are on the same
on $(f,M)$ and $(f_T, M_T)$, we find that $\Res_{(2)} \circ \dev_{(2)} \circ
\Phi_{\mathrm{m}} = p^* \circ \Res_{(1,1)} 
\circ \dev_{(1,1)}$. Since $\Res_{(2)}$ is an isomorphism, this yields
\eqref{eq: what we want} with $L = \Res_{(2)}^{-1} \circ p^* \circ \Res_{(1,1)}$. This
computation, and the fact that the action of $U$ does not move the
points of $e'$, also shows that $\Phi$ is $U$-equivariant.

We now show that $\Phi$ lifts to a continuous map $\Phi_{\mathrm{f}}: \HH''_T \to
\HH_{\mathrm{f}}(2)$. In view of the discussion in \S \ref{subsec:
  framings}, in order to lift $\Phi$ to a map to $\HH_{\mathrm{f}}(2)$ 
we need to equip $M_T$ with a right-pointing horizontal prong at the
singular point $\xi$.  
Let $\delta' = \delta'(M)$ be as in the definition of
$\HH''_T$ and let $q(M)$ be the prong which is obtained by moving an
angle $\pi$ in the 
counterclockwise direction from the terminal prong of $\delta'$, at the
singularity $\xi_2$. Then $q(M)$ is in the complement of $\delta'$
and so is mapped by $\bar{f}_T$ to a horizontal prong on
$M_T$. See Figures \ref{fig: connecting} and
\ref{fig: corresponding}, where $q(M)$ is marked as $q_1$. We need to
show that with this choice of selected prong, $\Phi_{\mathrm{f}}$ is
continuous. In light of Proposition \ref{prop: recovering framed}, it
is enough to show two things: (i) that the choice of prong $M \mapsto
q(M)$ is continuous with respect to the coordinates given by the
developing map, for any fixed triangulation; and (ii), that $M \mapsto
q(M)$ is $\Mod(\check S)$-invariant. It is clear from our description
of $q(M)$ that (i) and (ii) are satisfied.

Finally, since $\Phi_{\mathrm{f}}$ is a lift of a locally affine 
map in charts, 
 in order to show that
it is a homeomorphism onto its image we need only
verify that it is injective, and for this we explicitly construct its
inverse. We first pick one basepoint $M_0' = \Phi_{\mathrm{f}}(M_0)$ in
each connected component of the image of $\Phi$, choose a horizontal
prong at $M_0'$ using the direction of $\delta$ in $M_0$ as in the
preceding paragraph, and extend this choice continuously to all
surfaces in the image of $\Phi$. Now for any $M'$ in the image of
$\Phi$, let $\mathcal{N}'$ be the $(L, \vre)$-rectangle thickening of $\Xi_-(M')$ for
$L>|T|$ and small enough $\vre$. We consider $\mathcal{N}'$ with its
decomposition into rectangles as in the preceding discussion, and we
now modify this decomposition. Let $q_1$ be the chosen right-pointing
prong on $M'$ and let $q_2$ be the left-pointing prong which is clockwise
from $q_1$. Similarly let $q_3, q_4$ be the left- and right-pointing
prongs which are at angular
distance $3\pi$ from $q_1, q_2$ respectively. Let $\sigma_1$ (resp. $\sigma_2$)
be the two vertical segment  of lengths $\vre$ between $q_1, q_2$
(resp., between $q_3$ and $q_4$) connecting $\xi$ to the boundary of
$\partial{N}'$. For $i=1,2$, the segments $\sigma_1, \sigma_2$ are boundary
segments of two rectangles $R_i^{\pm}$ of the complex $\mathcal{N}'$,
one on each side. Let $\Delta_i$ be a triangle which is embedded in
the union of $R_i^+ \cup R_i^-$, has an apex at $\xi$, and has
$\sigma_i$ as an altitude contained in its interior. We now replace each
$R_i^{\pm}$ with $R_i^{\pm} \sm \Delta_i$, and add the $\Delta_i$ to
the polygonal decomposition of $\mathcal{N}'$. Thus we have a
decomposition of $\mathcal{N}'$ into rectangles, trapezoids, and two
triangles. To each of them we apply the map described  in the proof of
Theorem \ref{thm: Bainbridge-Smillie surgery}, with $-T$ instead of
$T$. That is, we do not move
points on $\partial \mathcal{N}'$ and the non-boundary edges. The two
triangles are thought of as 
degenerate trapezoids. The choice of
the prongs at $\xi$, and the fact that $M'$ is in the image of $\Phi$,
ensure that these operations are well-defined, that is for all $t$
strictly 
between $0$ and $T$, the deformed shapes are nondegenerate
trapezoids. Gluing them to each other using the gluing map of
$\mathcal{N}'$ and gluing the resulting complex to $M \sm
\mathcal{N}'$ completes the definition of the inverse of
$\Phi_{\mathrm{f}}$. 
\end{proof}

The image of $\Phi$ in Theorem \ref{thm: criterion1} can be described
explicitly in 
terms of the choice of horizontal prong at the singularity. Namely
suppose again that $T>0$ and that $q_1$ is the chosen prong.  Let $\bar{q}_2,
\ldots, \bar{q}_6$ be the additional
prongs at $\xi$ in counterclockwise order (note that this differs from
the labeling in Figure \ref{fig: corresponding}). Then the image of $\Phi$ is the
set of $M \in \HH_{\mathrm{f}}(2)$ which have no horizontal saddle
connections of length at most $T$ from any one of $q_1$ or $\bar{q}_3$, to
one of $\bar{q}_4$ or 
$\bar{q}_6$. 

The inverse of $\Phi$ appearing in Theorem \ref{thm: criterion1} is
the operation of `splitting open a singularity' which was discussed in
\cite{EMZ}.

\section{The eigenform locus}
\label{section: eigenform locus}
In this section we will define the eigenform locus $\EE_D$. We describe its
intersection with $\HH(1,1)$ and $\HH(2)$ and describe how it meets some boundary strata.
We summarize some properties of surfaces in the eigenform locus.

\subsection{Definition of the eigenform loci}
The eigenform loci were defined by Calta \cite{Calta} and McMullen
\cite{McMullen-JAMS}.   Calta made use of the $J$ invariant and McMullen used properties
of real multiplication on Jacobians. Here we follow the approach of McMullen.  

For every positive integer $D \equiv 0,1 \pmod 4$ with $D\geq 4$ there is a closed, connected,
$G$-invariant locus $\EE_D \subset \HH(2)\cup\HH(1,1)$, called the
\emph{eigenform locus}, which we now describe. 

An \emph{order} in a number field $F$ is a subring $\mathcal{O}$ of the ring of integers
$\mathcal{O}_F$ which is finite index as an abelian group.  Orders in quadratic fields are
particularly simple as they can be classified by a single integer $D$, the
\emph{discriminant}.  More precisely,  
for every positive integer $D \equiv 0,1 \pmod 4$, we consider the real quadratic order \begin{equation*}
  \mathcal{O}_D\index{O@$\mathcal{O}_D$}  = \zed[T]/(T^2 + bT + c),
\end{equation*}
where $b,c\in\zed$ are such that $b^2-4c=D$.  If $D$ is not square, it is a subring of the real
quadratic field, $F_D = \ratls[T]/(T^2+bT+c)$.  We also allow $D$ to be square, in which case $F_D$
is isomorphic to $\ratls\oplus\ratls$ as a $\ratls$-algebra.  In either case the isomorphism classes
of $\mathcal{O}_D$ and $F_D$ depend only on $D$.  We fix a choice of a ring homomorphism
$\iota\colon F_D\to \reals$.  When $D$ is not square, $\iota$ is a field embedding.  If $D$ is
square, $\iota$ is obtained from the projection of $\ratls\oplus\ratls$ onto its first factor.  A
more detailed discussion of orders in number fields appears in \cite{BoSh}.

Consider a genus two Riemann surface $X$ with Jacobian variety $\Jac(X) = \Omega(X)^* / H_1(X;\zed)$,
where $\Omega(X)$ is the space of holomorphic one-forms on $X$.  \emph{Real multiplication} by
$\mathcal{O}_D$ on $\Jac(X)$ is a ring monomorphism $\rho\colon \mathcal{O}_D\to \End^0(\Jac(X))$,
where $\End^0(\Jac(X))$ is the ring of endomorphisms of $\Jac(X)$ which are self-adjoint with respect
to the intersection form on $H_1(X; \zed)$.  We also require $\rho$ to be \emph{proper}, in the
sense that it does not extend to
$\mathcal{O}_E\supsetneq\mathcal{O}_D$ for some $E \mid D$. 

Real multiplication by $\mathcal{O}_D$ induces a representation of $\mathcal{O}_D$ on $\Omega(X)$,
and by self-adjointness, a decomposition of $\Omega(X)$ into complementary eigenspaces. A nonzero
holomorphic one-form on $X$ is an eigenform if it lies in one of these eigenspaces.  We say that a
pair $(X, \omega)$ is an eigenform for real multiplication if $\Jac(X)$ has real multiplication with
$\omega$ an eigenform. 

Real multiplication for curves in genus two is very special, as it can
be detected from knowledge of the 
absolute periods of a single one-form on the curve.  That is to say,
real multiplication in genus two has a `purely flat' description.
More precisely: 

\begin{prop}[\cite{Matt}]\label{prop: detecting RM}
  A genus two translation surface $M$ is an eigenform for real multiplication by
  $\mathcal{O}_D$ if and only if there is a proper monomorphism
  $\rho_0\colon \mathcal{O}_D 
  \to \End^0(H_1(M; \zed))$ such that
  \begin{equation}\label{eq:1}
   \hol(M, \rho_0(\lambda)\cdot \gamma) = \iota(\lambda)\hol(M, \gamma)
     \end{equation}
  for each $\lambda\in\mathcal{O}_D$ and $\gamma\in H_1(M; \zed)$.   
\end{prop}
The $\rho_0$ in this Proposition is simply the action on homology
induced by the real 
multiplication $\rho\colon\mathcal{O}_D\to \End^0(\Jac(M))$. 
See also \cite[Lemma~7.4]{McMullen-JAMS}, and see \cite{CS} for an
alternative approach.

We define the eigenform locus $\EE_D\subset\HH(2)\cup\HH(1,1)$ \index{E@$\EE_D$} to be
the locus of eigenforms for 
real multiplication by $\mathcal{O}_D$, and we define $\EE_D(2)$ \index{ED2@$\EE_D(2)$}and
$\EE_D(1,1)$ \index{ED1@$\EE_D(1,1)$}to be the 
intersections of $\EE_D$ with the respective strata. Similarly we
denote the corresponding subsets of area-one surfaces by
$\EE_D^{(1)}(2)$ and $\EE_D^{(1)}(1,1)$.\index{Ed@$\EE_D^{(1)}$}

The locus $\EE_D(1,1)$ is
$\GL_2(\reals)$-invariant, as it can be easily seen that the condition
of Proposition~\ref{prop: 
  detecting RM} is $G$-invariant (this was first proved in \cite{McMullen-JAMS} and
\cite{Calta}). It is also Rel invariant since this condition only involves absolute
periods. Moreover $\EE_D(1,1)$ is a six dimensional linear submanifold
of $\HH(1,1)$ with respect to 
the period coordinates from \S\ref{subsec: strata}.  To see this
explicitly, choose two generators $\gamma_1, 
\gamma_2$ of $H_1(M;\Z)$ (as an $\mathcal{O}_D$-module) and complete
to a set of four generators as 
a $\zed$-module, e.g.\ by adjoining a multiple of each $\gamma_i$ by a
generator of $\mathcal{O}_D$ 
over $\zed$. Equation \eqref{eq:1} now gives linear equations which
the vectors $\hol(M, \gamma)$ 
must satisfy, and these equations define $\EE_D(1,1)$ locally. As a
consequence $\EE_D^{(1)}(1,1)$ is a five-dimensional manifold locally
defined in period coordinates by linear equations and one quadratic
equation. 

This dimension count easily implies that $\GL_2(\reals )$-orbits and
Rel \index{Rel} leaves locally fill  $\EE_D(1,1)$. 

Recall the semi-direct product $L = G \ltimes \R^2$ introduced in \S
\ref{subsec: real rel}. The group $L$ is embedded in $\GL_2(\R) 
\ltimes \R^2$ which is an open set in a 6-dimensional vector space. 

\begin{prop}\label{prop: local
  L action}
For any $M \in \EE_D^{(1)}(1,1)$ there is a neighborhood $\mathcal{U}$ of the identity in
$L $ and a neighborhood $\mathcal{U}'$ of $M$ in  
$\EE_D^{(1)}(1,1)$ such that the map $p\colon \mathcal{U}\to \mathcal{U}'$ defined by
$$ p(g, v) = gM \pluscirc v $$ 
is the restriction of an affine homeomorphism to $\mathcal{U}$.  
\end{prop}

\begin{proof}
  Consider a precompact neighborhood $\mathcal{W}$ of the identity in
  $\GL_2(\reals)$.  For some $\varepsilon>0$, 
  no surface in $\mathcal{W} \cdot M$ has saddle connections of length less than $\varepsilon$.  By
  Corollary~\ref{thm: criterion}, $p$ is well-defined on $\mathcal{V}=
  \mathcal{W} \times B_\varepsilon(0)\subset
  \GL_2(\reals)\ltimes \reals^2$.  Possibly decreasing $\varepsilon$ so that the image of $p$ is
  contained in an affine coordinate chart as defined above, $p$ is a homeomorphism onto its image.
  Since the action of $\rel$ preserves the area of surfaces, $p$ sends $L$ into the locus of
  area-one surfaces.  Intersecting $\mathcal{V}$ with $L$ and the image of $p$ with the locus of area-one
  surfaces, we obtain $\mathcal{U}$ and $\mathcal{U}'$ with the
  required properties.
\end{proof}

Recall that the Rel operations were defined in \S \ref{subsec: real rel} in terms of a
vector field $\mathfrak{R}$ and in $\HH(1,1)$ we have $\mathfrak{R}
\cong \R^2$. Also note that the $G$-action on a stratum makes it
possible to associate to each element of the Lie algebra
$\mathfrak{sl}_2(\R)$ of $G$, a vector field on $\HH$. 
From Proposition \ref{prop: local
  L action} and the Frobenious integrability criterion (see
e.g. \cite{Warner}) we obtain: 

\begin{cor}\label{cor: frobenius}
For any $D$, the vector fields on $\EE_D^{(1)}(1,1)$ obtained as the infinitesimal generators of the
$G$-action and the Rel operations, taken with the natural Poisson
bracket, have the structure of the Lie algebra $\mathfrak{l} = \mathfrak{sl}_2(\R)
\ltimes \R^2$. For any Lie subalgebra $\mathfrak{l}_0 \subset
\mathfrak{l}$, the plane field corresponding to the vector fields in
$\mathfrak{l}_0$ are integrable and define a foliation of
$\EE_D^{(1)}(1,1)$. 
\end{cor}

The eigenform locus has a more elementary description in the case where $D$ is a square.  A
translation surface $X$ is a \emph{torus cover} if there is a branched cover $p\colon X \to T$ which
is a local translation for some flat torus $T$ (note that the branch points of $p$ are not required
to lie over a single point of $T$).  We say $p$ is \emph{primitive} if it does not
factor through a torus cover of smaller degree, equivalently if the map on homology $p_*\colon H_1(X;
\mathbb{Z}) \to H_1(T; \mathbb{Z})$ is onto.  McMullen established in \cite{McMullen-SL(2)} that
$\EE_{d^2}$ is the locus of primitive degree $d$ torus covers.

\subsection{$G$-invariant measures in genus two}
We now discuss the $G$-invariant measures in genus two. 
Of course each of these is also a horocycle invariant measure. These were classified by McMullen in
\cite{McMullen-SL(2)}.  Measures supported on the full strata were constructed by Masur
\cite{Masur1} and Veech \cite{Veechstrata} using period coordinates on
these strata. 
In \cite{McMullen-SL(2)} McMullen constructed measures on
the eigenform loci in an analogous way using period coordinates.

We may use Proposition~\ref{prop: local L action} to define a measure on $\EE_D(1,1)$ by locally
pushing forward Haar measure on $L$.  
More precisely, recall that $L$ is unimodular,
that is its Haar measure is invariant under both right and left
multiplications. Given $E$ in the image $\mathcal{U}'$ of $p$,
we assign to $E$ the Haar measure of $p^{-1}(E)\subset L$. Right-invariance of Haar measure
implies that the measure of $E$ doesn't depend on the choice of
basepoint, and left-invariance implies that the measure is
$G$-invariant. McMullen \cite{McMullen-SL(2)} proved that this measure
is finite. We call this the \emph{flat measure} on 
$\EE_D(1,1)$. 

Here is an alternative description of the flat measure which can be generalized to define a measure
on any linear submanifold of a stratum.  Suppose $f\colon S\to M$ is a
marked translation surface, 
write $M=(X,\omega)$, and suppose that $\Jac(X)$ admits real
multiplication by $\mathcal{O}_D$.  The 
real multiplication on $\Jac(X)$ gives $H_1(S; \zed)$ the structure of
an $\mathcal{O}_D$-module. A 
choice of embedding $\mathcal{O}_D\to \reals$ makes $\reals^2$ an
$\mathcal{O}$-module as well.  We 
define $H^1_{\mathcal{O}_D}(S, \Sigma; \reals^2)\subset H^1(S, \Sigma;
\reals^2)$ to be the subspace 
of cocycles for which the induced period map $H_1(S; \zed) \to
\reals^2$ is $\mathcal{O}_D$-linear. 
This is in other words the space of cocycles satisfying \eqref{eq:1}.  By Proposition~\ref{prop: detecting RM}, the linear subspace $H^1_{\mathcal{O}_D}(S, \Sigma;
\reals^2)$ parameterizes the eigenform locus in $\HH_{\mathrm{m}}$ near $M$.

We have the
commutative diagram of cohomology groups (all coefficients in $\reals^2$):
$$
\xymatrix{  \ar[r]0  &  \ar[r] \ar[d] H^0(\Sigma)/H^0(S) &
  \ar[r]\ar[d] 
  H^1_{\mathcal{O}_D}(S, \Sigma) & \ar[r]\ar[d]  H^1_{\mathcal{O}_D}(S)& 0\\
  \ar[r] 0 & \ar[r] H^0(\Sigma)/H^0(S) & \ar[r] H^1(S,
  \Sigma) & \ar[r] H^1(S) & 0
}
$$
We give $\EE_D(1,1)$ a measure by defining a measure on the linear space
$H^1_{\mathcal{O}_D}(S, \Sigma)$ on which it is modeled which is invariant under the monodromy
action.  To define such a measure, we give a measure on the other two terms of the short exact
sequence.  The left term is canonically $\reals^2$, and the monodromy action is trivial.  We give it
the usual Euclidean area form.  The $H^1(S)$ term has a symplectic form arising from the
intersection form, which is preserved by the action of the mapping class group.  This descends to a
symplectic form on $H^1_{\mathcal{O}_D}(S)$ which is preserved by monodromy. This form is
non-degenerate; this could be checked by direct computation in this case, and was proved in complete generality
by Avila, Eskin and M\"oller \cite{AEM}. Therefore the symplectic form defines a volume form on
$H^1_{\mathcal{O}_D}(S)$ which is also monodromy invariant.  The product of these volume forms
induces one on $H^1_{\mathcal{O}_D}(S, \Sigma)$ which defines a measure on $\EE_D(1,1)$.  Finally,
we apply the standard cone construction (meaning we push forward the
restriction of the measure to surfaces in $\EE_D(1,1)$ of area at most
1, by the canonical
projection onto the locus of surfaces of area one; see \cite{zorich survey} for details)
to obtain a 
$G$-invariant measure on $\EE^{(1)}_D(1,1)$.

The eigenform locus $\EE_D(1,1)$ is nonempty for each $D\geq 4$ and $D \equiv 0$ or $1\pmod 4$.  In
each case it is connected 
by \cite{McMullen-JAMS}.  The $D=1$ case can be regarded as the
locus of degenerate eigenforms where one separating curve has been pinched.  This locus lies in the
boundary of $\EE_D(1,1)$ and is called the ``product locus'' in \cite{Matt_thesis}.   

In the stratum $\HH(2)$, the locus of eigenforms $\EE_D(2)$, is called the Weierstrass curve in
McMullen's papers. By \cite{mcmullenspin}, $\EE_D(2)$ consists of a single $G$-orbit if $D\not\equiv
1 \pmod 8$ and $D\geq 5$ (note $\EE_5(2)$ is empty), or if $D = 9$. Otherwise $\EE_D(2)$
consists of two orbits.  It is equipped with a finite measure coming from
Haar measure on $G$. 

The square-discriminant eigenform locus $\EE_{d^2}$ also contains a countable, dense collection of
closed $G$-orbits.  A translation surface $X$ is called a \emph{square-tiled surface} if it is a
branched cover of the standard square torus with all of the branching lying over a single point.
Every square-tiled surface has a closed $G$-orbit, and the square-tiled surfaces are dense in each
$\EE_{d^2}$.

Closed $G$-orbits inherit a measure from the Haar measure on $G$, and
this measure is finite by a result of Smillie (see
\cite{toronto}). We 
will refer to this measure as the Haar measure on the closed $G$-orbit.
The {\em decagon surface} is the surface obtained by identifying
opposite sides of the regular 10-gon. It was shown by Veech
\cite{Veech - alternative} that it has a closed $G$-orbit. It belongs to 
$\HH(1,1)$ and in fact to the eigenform locus $\EE_5(1,1)$.  We write $\mathcal{L}_{\rm dec}\subset
\EE_5(1,1)$\index{L@$\mathcal{L}_{\rm dec}$} for its $G$-orbit.
Closed $G$-orbits in genus two were constructed by Calta 
\cite{Calta} and McMullen \cite{McMullen-JAMS} and then classified by McMullen \cite{McMullen - decagon}.

\begin{thm}
  \label{thm:closed-orbit-classification}
Each connected component of $\EE_D(2)$ is a closed $G$-orbit, and
every closed $G$-orbit in $\HH(2)$ is of this form. 
  
  Every closed $G$-orbit in $\HH(1,1)$ is either the $G$-orbit of a
  square-tiled surface or is $\mathcal{L}_{\rm dec}$.
\end{thm}

In \cite{McMullen-SL(2)}, McMullen showed that the measures defined above are the full list of ergodic
$G$-invariant measures in genus two:

\begin{thm}
  \label{thm:ergodic-SLtwoR-classification}
  Every ergodic $G$-invariant measure on $\HH(2)$ is either the flat measure on the full stratum or
  Haar measure on a closed $G$-orbit.
  
  Every ergodic $G$-invariant measure on $\HH(1,1)$ is either the flat measure on the full stratum,
  the flat measure on some $\EE_D(1,1)$, or the Haar measure on a closed $G$-orbit.
\end{thm}

\subsection{Degenerate eigenform surfaces}
We will also be interested in eigenforms which are not genus two
surfaces but can be thought of as surfaces lying in a bordification of
$\HH(1,1)$. We will consider two cases, where the role of ``boundary 
strata'' will be played respectively by $\HH(0) \times \HH(0)$ and
$\HH(0,0)$. 

 Given a pair of genus
one translation surfaces $E_1$ and $E_2$, we may consider the
one-point connected sum $X = E_1 \# 
E_2$.  These degenerations arise from families of genus two surfaces
where a separating curve has 
been pinched.  As we have the direct sum decompositions $H_1(X) =
H_1(E_1) \oplus H_1(E_2)$ and 
$\Omega(X) = \Omega(E_1) \oplus \Omega(E_2)$, the Jacobian of $X$ is
simply the product of $E_1$ and 
$E_2$:
\begin{align*}
  \Jac(X) &= \Omega^*(X)/ H_1(X; \zed)\\
  &\cong \Omega^*(E_1)/H_1(E_1; \zed) \oplus \Omega^*(E_2) / H_1(E_2; \zed)\\
  &\cong E_1 \times E_2
\end{align*}
Just as for the smooth case, we say that $X$ is an eigenform
for real multiplication if 
$\Jac(X)$ has real multiplication, with the one-form defining the
translation structure belonging to 
one of the eigenspaces in $\Omega(X)$.   McMullen gave a more explicit description of real multiplication for
these degenerate surfaces in terms of isogenies of the $E_i$. 

Recall that $E_1$ and $E_2$ are \emph{isogenous} if there is a holomorphic covering map
$p\colon E_1\to E_2$.  The isogeny $p$ is \emph{primitive} if it cannot be written as a composition
of an isogeny of lower degree with a self-covering of $E_2$.  Existence of $p$ yields a dual isogeny
$\bar{p}\colon E_2\to E_1$, so isogeny is an equivalence relation.  In translation coordinates, an
isogeny $p$ is of the form $p(z) = \lambda z + c$ for some complex number $\lambda$ which we call
the \emph{scaling factor} of $p$.  In our setting, $\lambda$ will always be real, in which case $p$
preserves the horizontal direction but scales the metric by a factor of $\lambda$.

\begin{prop}[\cite{mcmullenspin}]\label{prop:PD explicit}
  The surface $E_1 \# E_2$ is an eigenform for real multiplication by $\mathcal{O}_D$ if and only if
  there exists a primitive degree $m$ isogeny $p\colon E_1\to E_2$, together with an integral
  solution $(e, \ell)$ to the equation $e^2 + 4m\ell^2 = D$ with $\ell > 0$ and $\gcd(e, \ell) = 1$,
  such that the scaling factor $\lambda$ is the unique real positive root to the equation $\lambda^2
  - e\lambda - \ell^2 m = 0$.
\end{prop}

We define $\mathcal{P}_D \subset \HH(0) \times \HH(0)$ to be the locus
of pairs $(E_1, E_2)$ such 
that $E_1 \# E_2$ is an eigenform for real multiplication by
$\mathcal{O}_D$.  With the diagonal 
$G$-action on $\HH(0) \times \HH(0)$, the locus $\mathcal{P}_D$ is $G$-invariant by
Proposition~\ref{prop:PD explicit}.

By \cite{mcmullenspin}, the locus $\mathcal{P}_D$ consists of finitely many closed $G$-orbits.  We
recall McMullen's classification of these $G$-orbits.  A \emph{prototype} for real
multiplication by $\mathcal{O}_D$ is a triple of integers $(e, \ell,
m)$ such that $D= e^2 + 4 
\ell^2 m$, with $\ell,m > 0$ and $\gcd(\ell,m)=1$.  A prototype $(e,
\ell, m)$ determines a prototypical 
form in $\mathcal{P}_D$ as follows.  Let $\lambda$ be the unique positive solution of
$\lambda^2=e\lambda + \ell^2 m$.  We define a pair of lattices in $\mathbb{C}$:
$$\Lambda_1 = \zed(\lambda, 0) \oplus \zed(0, \lambda) \quad \Lambda_2 = \zed(\ell m, 0)\oplus\zed(0,
\ell),$$
and associated genus one translation surfaces  $E_i = (\mathbb{C} /
\Lambda_i, dz)$.  Multiplication by 
$\lambda$ defines an isogeny $p \colon E_1 \to E_2$ of degree
$l^2m$. 

\begin{prop}[\cite{mcmullenspin}] 
  Each $G$-orbit of $\mathcal{P}_D$ contains a unique prototypical form.
\end{prop}

Given a prototype $(e,\ell,m)$, we define $\mathcal{P}_D(e,\ell,m)$ to
be the closed $G$-orbit 
containing the prototypical form associated to $(e, \ell, m)$.  By the
above proposition, this gives 
a bijection between prototypes and components of $\mathcal{P}_D$.  We
say that two pairs of 
genus-one forms \emph{have the same combinatorial type}\index{combinatorial type} if they lie on
the same component of 
$\mathcal{P}_D$.  

By \cite[Theorem~2.1]{mcmullenspin}, the $G$-orbit
corresponding to the prototype 
$(e,\ell,m)$ is isomorphic to the modular curve $G/\Gamma_0(m)$, where
$\Gamma_0(m)\subset \SL(2, 
\mathbb{Z})$ is the group of matrices which are upper-triangular mod $m$.

Finally, in the case when $D=d^2$ with $d>1$, we consider the moduli
space $\HH(0,0)$ of genus one 
translation surfaces $E$ with two marked points $p$ and $q$.  Again,
there is a natural $G$-action 
on this space.  We define $\mathcal{S}_{d^2}\subset\HH(0,0)$ to be the
locus of $(E,p,q)$ such that 
$p-q$ is exactly $d$-torsion in the group law on $E$ (that is
$d(p-q)=0$ and $d'(p-q)\neq 0$ for any 
$d'<d$, in particular this implies $p\neq q$).  By \cite{Matt_thesis},
$\mathcal{S}_{d^2}$ is 
a closed $G$-orbit isomorphic to $G/\Gamma_1(d)$, where
\begin{equation*}
  \Gamma_1(d) = \left\{
    \begin{pmatrix}
      a & b \\
      c & e 
    \end{pmatrix}
    \in \SL_2(\zed)
    : a \equiv e \equiv 1 \text{ (mod } d) \text{ and } c\equiv 0 \text{ (mod } d) \right\}.
\end{equation*}\index{G@$\Gamma_1(d)$}

We can regard $(E, p,q)$ as a degenerate degree $d$ torus cover obtained by pinching a nonseparating
curve.  Since $p-q$ is a point of order
$d$ in the group law on $E$, we have a degree $d$ map $\pi\colon E\to F$, where $F$ is the quotient
of $E$ by the order $d$ subgroup generated by $p-q$.  This $\pi$ is the torus cover of minimal
degree sending
$p$ and $q$ to the same point.

The precise sense in which surfaces in the loci $\mathcal{P}_D$, $\mathcal{S}_{D}$, and $\EE_D(2)$ can all be regarded as lying in the
boundary of $\EE_D(1,1)$, is explained in \cite{Matt_thesis}. We will
not explicitly use this point of view.

\subsection{Properties of eigenform surfaces}
We will require the following properties of surfaces in $\EE_D(1,1)$.  Recall that the
\emph{modulus} of a flat cylinder is defined to be its height divided by its circumference.

\begin{prop}\label{prop: configurations HD}
  In any cylinder decomposition of a surface in $\EE_D(1,1)$ with more than one cylinder, there is
  at least one rational relation among the moduli of horizontal cylinders.
\end{prop}

\begin{proof}
  Given a horizontally periodic surface $M$ with $n$ cylinders, the
  orbit-closure $\overline{UM}$ is 
  a torus of dimension $n-r$, where $r$ is the number of independent
  rational relations among the 
  moduli of the cylinders (see \cite [Proposition~4]{calanque}).  We
  must then show that in any case the dimension of $\overline{UM}$ is 
  smaller than $n$.

  First, suppose $n = \dim \overline{UM} = 3$.  We may parameterize a
  neighborhood of $M$ in 
  $\HH(1,1)$ by the three holonomy vectors $(x_i, y_i)$ of a saddle
  connection joining cone points on the 
  two boundary components of the $i$-th cylinder $C_i$.  Since the
  $y_i$ are constant on 
  $\overline{UM}$ and $\dim\overline{UM}=3$, the $x_i$ are arbitrary
  on $\overline{UM}$.  However, 
  since $\overline{UM}$ is contained in the eigenform locus, the $x_i$
  satisfy the nontrivial real-linear equation \eqref{eq:1} which defines the eigenform locus (see
  also equation~(6.5) of \cite{Matt} where this equation is given explicitly), a contradiction.

  Now suppose $n = \dim \overline{UM} = 2$.  There must then be a
  horizontal saddle connection 
  joining distinct zeros.  Applying the map $\Phi$ of Theorem
  \ref{thm: criterion1},
  i.e., the real-rel flow so that the
  length of this saddle connection 
  tends to $0$, yields a surface $M'$ in $\HH(2)$ with two
  horizontal cylinders with the same 
  moduli.  A neighborhood of $M'$ in $\HH(2)$ is parameterized by two
  holonomy vectors $(x_i, y_i)$. 
  Again, if $\dim \overline{UM}=2$, the $x_i$ would be arbitrary on
  $\overline{UM'}$, which 
  contradicts the real-linear equation \eqref{eq:1} which defines the eigenform
  locus.  Alternatively, since $GM'$ is closed, $UM'$ must also be closed, and so is $UM$, a
  translate of $UM$ by real-rel.
  (Note in the two-cylinder case, the 
  claim also follows directly from \cite [Theorem~9.1]{McMullen-SL(2)}.)
\end{proof}

\begin{prop}\label{prop: ergodicity}
The $U$-action on each $\EE_D(1,1)$, with respect to the flat
measure, is ergodic. 
\end{prop}

\begin{proof}
In a $G$-action,
ergodicity of the geodesic flow implies ergodicity of the $U$-action
by the ``Mautner phenomenon'' (see e.g. \cite{EW}). So it suffices to prove ergodicity of the
$G$-action. This can be proved by applying the Hopf argument to the
geodesic flow as in \cite{Masur1}.
\end{proof}

Recall that a \emph{periodic direction} of a translation surface is a direction in which the surface
is a union of parallel cylinders and saddle connections.  A translation surface is \emph{completely
  periodic} if for any direction which contains a cylinder, the surface is a union of parallel
cylinders and saddle connections.

\begin{thm}[\cite{McMullen - decagon, Calta}]\label{thm:complete periodicity}  
  Every genus two surface $M$ which is an eigenform is completely
  periodic. Moreover for $M\in
  \EE_D(1,1)$, if 
  $\Xi(M)$ contains a saddle connection joining a singularity 
to itself, then the horizontal direction of
  $M$ is periodic.
\end{thm}

Given two genus one translation surfaces $E_1$ and $E_2$, consider a horizontal segment
$I\subset\reals^2$ which embeds into each $E_i$.  We may form the \emph{connected sum} $X = E_1 \#_I
E_2$ by removing the image of $I$ from each $E_I$, and then gluing the resulting boundary
components.  The result is a surface $X\in \HH(1,1)$ with two horizontal saddle connections whose
union is a loop separating $X$ into two tori (see \cite{McMullen-SL(2)} for details).

\begin{thm}\label{thm: loops in the spine}
  Suppose that $D$ is not square.
  Let $M\in \EE_D$, and suppose $M$ contains a loop which is a union of horizontal saddle
  connections.  Then either
  \begin{itemize}
  \item The horizontal direction of $M$ is periodic, or
  \item $M$ is obtained by gluing two genus one translation surfaces with uniquely ergodic
    horizontal directions along a horizontal slit.
  \end{itemize}
  In particular, in the second case $\Xi(M)$ consists of two horizontal saddle connections joining distinct
singularities which are interchanged by the hyperelliptic involution (for
  the corresponding horizontal data diagram, see 
  diagram 5 in Figure \ref{fig: nonmaximal}).\end{thm}

\begin{proof}
  We regard $M$ as a Riemann surface $X$ equipped with a holomorphic one-form $\omega$.  By
  \cite{McMullen-Acta}, if $\RE\omega$ has zero ``Galois flux'' and $\Xi(M)$ contains a loop,
  then we are in one of these two cases.  Zero flux follows from equation (4.2) of
  \cite{McMullen-Acta} and Theorem~5.1 of \cite{McMullen-SL(2)}.

  The final statement is proved in the following lemma.
\end{proof}

\begin{lem}
  \label{lem: interchange equals disconnect}
  Let $M\in \HH(1,1)$ be a surface with two horizontal saddle connections $I$ and $I'$ joining
  distinct 
singularities. The loop $\gamma = I\cup I'$ disconnects $M$ if and only if the hyperelliptic
  involution interchanges $I$ and $I'$.  
\end{lem}

\begin{proof}
  Let $\eta\colon M\to M$ denote the hyperelliptic involution and $\omega$ the holomorphic one-form
  on $M$ induced by the translation structure.

  Suppose that we are in the case where $\eta$ interchanges $I$ and $I'$.  Since $\eta^*\omega =
  -\omega$, the map $\eta$ sends $\gamma$ to itself preserving the
  orientation.  Thus 
  $[\eta_*\gamma] = [\gamma] \in H_1(M; \zed)$.  But the hyperelliptic
  involution acts on $H_1(M; 
  \zed)$ as $-\mathrm{Id}$. 
 It follows that $\gamma$ is
  homologous to zero, so it separates 
  $M$.

  Suppose now that $\eta$ fixes $I$ and $I'$.  Since $\eta$ reverses
  the orientation of each of 
  these saddle connections, each must contain a single fixed point of
  $\eta$ (i.e.\ one of the six 
  Weierstrass points of $M$).  If $f\colon M\to S^2$ is the quotient map induced by the
  hyperelliptic involution, then in this case the union $f(I) \cup
  f(I')$ is a smooth embedded path 
  joining two branch points of $f$.  Since an embedded path does not
  disconnect the sphere, and $f$ 
  has branch points disjoint from $f(I)\cup f(I')$, the complement $M
  \setminus (I \cup I')$ is also 
  connected, a contradiction.
\end{proof}

When $D=d^2$, there is one more possible configuration of horizontal saddle connections.  Suppose
$(E, p,q)$ is a genus one translation surface with two marked points $p,q$ whose difference $p-q$ is
exactly $d$-torsion in the group law of $E$.  There is then a genus one surface $F$ and degree $d$
cover $\pi\colon E\to F$ with
 $\pi(p) = \pi(q)$.  Let $I\subset \reals^2$ be a segment which
may be embedded in $E$ by a translation to yield disjoint parallel segments $I'$ and $I''$ beginning
at $p$ and $q$ respectively.  We may then form the self-connected-sum $M=(E, p,q)\#_I$ by cutting
along $I'$ and $I''$ and then regluing to obtain a genus two surface.  Since the gluings are
compatible with the covering $\pi\colon E\to F$, the surface $M$ is a primitive degree $d$ branched
cover of $F$, so $M\in \EE_{d^2}(1,1)$.  If the slope of $I$ is not a periodic direction on $E$, the
surface $M$ has exactly two saddle connections of this slope of the same length, and the complement
of this pair of saddle connections is a genus one surface. As in the
second case in Theorem \ref{thm: loops in the spine}, the horizontal diagram is once again
diagram 5 of Figure \ref{fig: nonmaximal}, but the complement of the
slit has a different topology than in the case that $D$ is not a
square.

The following theorem says when $D$ is square, this is the only additional configuration of saddle
connections. 

\begin{thm}\label{thm: loops in the spine square case}
  Suppose that $D = d^2$.
  Let $M\in \EE_D$, and suppose $M$ contains a loop which is a union of horizontal saddle
  connections.  
Then one of the following holds. 
  \begin{itemize}
  \item The horizontal direction of $M$ is periodic. 
  \item $M$ is obtained by gluing two genus one translation surfaces with uniquely ergodic
    horizontal directions along a horizontal slit. 
  \item $M$ is a self-connected sum of $(E, p,q)$ along two horizontal slits of the same length
    based at $p$ and $q$ respectively, where $E$ is a genus one translation surface with uniquely
    ergodic horizontal direction, and $p-q$ has order exactly $d$ in the group law of $E$.
  \end{itemize}
  In particular, in the second case $\Xi(M)$ consists of two horizontal saddle connections joining distinct
singularities which are interchanged by the hyperelliptic involution.  In the third case, $\Xi(M)$
  consists of two horizontal saddle connections of the same length fixed by the hyperelliptic involution.
\end{thm}

\begin{proof}
  Suppose that $M$ is not periodic.  If $M$ has a saddle connection joining a singularity to itself,
  this saddle connection is the boundary of a cylinder, which implies that the horizontal direction
  of $M$ is periodic by Theorem~\ref{thm:complete periodicity}, a contradiction.  Therefore $\Xi(M)$
  consists of two horizontal saddle connections $I$ and $I'$ joining singularity $\xi_1$ to
  singularity $\xi_2$.  

  Let $\pi\colon M\to F$ be a primitive degree $d$ cover of a
  genus one translation surface.   The horizontal direction of $F$ is not periodic, otherwise $M$
  would be periodic as well.  The singularities $\xi_i$ are also the branch points of $\pi$, and the
  images $\pi(I)$ and $\pi(I')$ are horizontal segments on $F$ joining $\pi(\xi_1)$ to
  $\pi(\xi_2)$.  Since the horizontal direction of $F$ is not periodic, the images $\pi(\xi_1)$ and
  $\pi(\xi_2)$ must be distinct.  There is then a unique horizontal segment on $F$ joining these
  points which $I$ and $I'$ are both mapped to injectively by $\pi$.  It follows that $I$ and $I'$
  have the same length. 

  By Lemma~\ref{lem: interchange equals disconnect},  the loop $I\cup I'$
  disconnects $M$ if and only if  $I$ and $I'$ are interchanged by the hyperelliptic involution.  
  
  Suppose we are in the case where $I$ and $I'$ are fixed by the hyperelliptic involution.  The
  complement $M\setminus (I\cup I')$ is then a genus one translation surface with two boundary
  components, each a union of two horizontal segments of the same length.  The holonomy around each
  boundary component is trivial, so we have exhibited $M$ as a self-connected-sum $(E, p, q)\#_I$.
  The branch covering $\pi$ must arise from a primitive covering $E\to F$ which identifies $p$ and
  $q$.  This implies that $p-q$ is exactly $d$-torsion in the group law of $E$.
\end{proof}

\section{Construction of $U$-invariant ergodic measures in $\EE_D$}
\label{section: construction}
In the previous section we discussed $G$-invariant measures on
$\EE_D$. In this section we will discuss $U$-invariant measures which
are not $G$-invariant. 

\subsection{Minimal sets}\label{subsec: minimal sets}
A {\em minimal set} for a flow is a minimal (with respect to
inclusion) nonempty closed invariant subset. An example is a closed
orbit, but more complicated examples may arise. The minimal sets for the
horocycle flow in any stratum were classified in \cite{calanque},
where the following was shown (see also \cite[\S6.1]{Barak-Pat-AY}). 

\begin{prop}\label{prop: minimal}
For $M \in \HH$, the following are equivalent:
\begin{itemize}
\item
$\overline{UM}$ is minimal. 

\item
$\overline{UM}$ is compact. 
\item
$M$ has a horizontal cylinder decomposition, with $\Xi(M)$ consisting
of the boundaries of these cylinders.
\end{itemize}
In case these hold, $\overline{UM}$ is topologically conjugate to the
quotient of a torus by a finite group, via a map which is
affine in charts. The dimension of this torus
is the dimension over $\Q$ of the span of the moduli of the
cylinders in the horizontal direction of $M$, and under this conjugacy,
the $U$-action on the torus is a straight-line linear flow. The
closure of any $U$-orbit in $\HH$ contains some $M$ satisfying 
the above conditions. 
\end{prop}

\begin{cor}\label{cor: minimal}
Each minimal set supports a unique $U$-invariant measure (which in particular is
ergodic). The measure is linear with respect to the linear structure
on the torus which is the minimal set. 

\end{cor}
\begin{proof}
These follow from well-known properties of straight line flows on tori, and
from the isomorphism above. 
\end{proof}

\begin{cor}\label{cor: missing} 
Let $M \in \EE_D$ be a surface with a horizontal cylinder
decomposition as in Proposition \ref{prop: minimal}. 
  For cylinder decompositions of type (A) of Figure~\ref{fig: maximal}, the corresponding
  orbit-closure is a minimal set which is one-dimensional or two-dimensional. In the former case it
  is a closed horocycle, and in the latter, the corresponding measure is invariant under $UZ$, and
  this group acts transitively on the support of the measure.  For cylinder decompositions of types
  (B), (C) and (D), $U$-orbits are periodic.
\end{cor}

\begin{proof}
  By Proposition \ref{prop: minimal} the dimension of the $U$-orbit closure is the dimension of the
  $\Q$-span of the moduli, which by Proposition~\ref{prop: configurations HD} is either 1 or 2.
  Suppose that it is 2.  Again by Proposition~\ref{prop: configurations HD}, there must be three
  cylinders, so we are in case (A).  The description in \cite{calanque} shows that the additional
  dimension of the orbit closure is given by changing only the twists
  around cylinders, and it is straightforward to check using the
  explicit coordinates used in \cite{calanque}, that the $Z$-action
  also affects only the twists, changing them linearly in a way which
  is distinct from the $U$ action. Thus the $UZ$-orbits are 2
  dimensional, and hence $UZ$ is transitive on the orbit-closure and
  the corresponding measure is $UZ$-invariant.

  In cases (B) and (C) there are only two cylinders which, by
  Proposition~\ref{prop: configurations 
    HD}, have rationally related moduli. In case (D) there is only one
  cylinder. So in all of these cases 
  the minimal set is 1-dimensional.
\end{proof}

\subsection{Parameter spaces of minimal sets} 
For every discriminant $D$, there are examples in $\EE_D$ of surfaces
belonging to minimal sets of type (A), for which the $U$-orbit closure
is two-dimensional, i.e.\ is an orbit of $UZ$.  These minimal spaces
are naturally grouped in continuous families, which we refer to
as beds. Their structure
was studied in  
detail in \cite[\S6]{Matt} when $D$ is not square. We summarize this
structure here. This analysis will not be required in the
sequel.  Contrary to the rest of 
the paper, this discussion will be in the context of surfaces with
unlabeled singularities as the  
combinatorial structure is simpler.

Each $\EE_D(1,1)$ contains finitely many four (real) dimensional beds 
which parameterize unit area 
surfaces with periodic horizontal directions having three cylinders,
where the ratios of circumferences of 
these cylinders are fixed, while twist parameters and heights are
allowed to vary, subject to the 
linear condition which defines the eigenform locus. These beds are
invariant under the horocycle flow, geodesic flow and real and
imaginary rel. 
 Each bed is foliated by a family 
of 2-tori defined by fixing the heights as well as the circumferences
of the cylinders, and varying the 
twist parameters.  Each torus is a closed $UZ$-orbit and is either a
$U$-minimal set or a union of 
closed $U$-orbits, depending on whether or not moduli of the cylinders
have rational ratios.  We 
will see in Theorem~\ref{thm: limits of UZ} that applying the geodesic
flow to one of these tori 
gives a family of measures which equidistribute in $\EE_D(1,1)$.

As we vary the heights of the cylinders in one of these  families of
three-cylinder surfaces (see Figure \ref{fig: 3 cylinder
  deformation}), the height of a 
cylinder may tend to $0$, yielding a family of two-cylinder surfaces
in the boundary of the family of three-cylinder surfaces (see Figure
\ref{fig: first 2 cylinder}). 
For a given three cylinder surface there are two cylinders which are
candidates for degeneration and for each candidate cylinder there is a
path in the family of three cylinder surfaces leading to a surface in
which this cylinder is degenerate.  Thus there are two distinct
families of two-cylinder surfaces lying in the 
boundary of each three-cylinder family.  These two-cylinder families
are three-dimensional and 
composed of surfaces having horizontal data diagram of type (B) or
(C). 
By Corollary~\ref{cor: missing}, each two-cylinder family is a union
of closed $U$-orbits. 

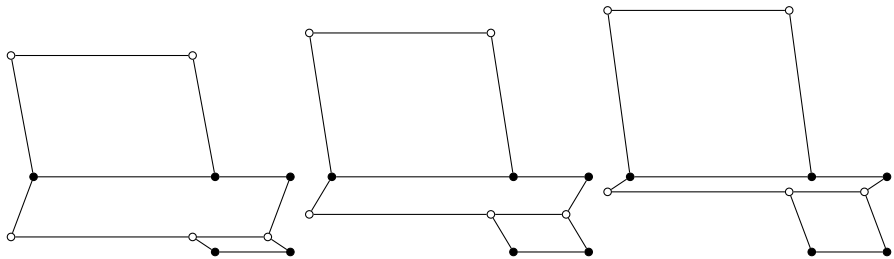
\begin{figure}[h]
\begin{tabular}{lcr}


\begin{tikzpicture}[scale=1.0]
\def\sq{1.414}
\def\t{0.8}
\def\a{1.0+\sq};
\def\b{1};
\def\c{1-\t};
\def\d{\t};
\def\e{1.0+\sq-\t};
\def\rel{0.3}

\node (A) at (\a+\rel,0) [circle,draw,fill=black,inner sep=0pt,minimum size=1.0mm] {};
\node (B) at (\a+\b+\rel,0) [circle,draw,fill=black,inner sep=0pt,minimum size=1.0mm] {};
\node (C) at (0,\c) [circle,draw,inner sep=0pt,minimum size=1.0mm] {};
\node (D) at (\a,\c) [circle,draw,inner sep=0pt,minimum size=1.0mm] {};
\node (E) at (\a+\b,\c) [circle,draw,inner sep=0pt,minimum size=1.0mm] {};
\node (F) at (\rel,\c+\d) [circle,draw,fill=black,inner sep=0pt,minimum size=1.0mm] {};
\node (G) at (\a+\rel,\c+\d) [circle,draw,fill=black,inner sep=0pt,minimum size=1.0mm] {};
\node (H) at (\a+\b+\rel,\c+\d) [circle,draw,fill=black,inner sep=0pt,minimum size=1.0mm] {};
\node (I) at (0,\c+\d+\e) [circle,draw,inner sep=0pt,minimum size=1.0mm] {};
\node (J) at (\a,\c+\d+\e) [circle,draw,inner sep=0pt,minimum size=1.0mm] {};

\draw (A) -- (B) -- (E) -- (H)--(G) --(J) -- (I) -- (F) -- (C) --(D) -- (A);
\draw (E) -- (D); 
\draw (G) -- (F);

\end{tikzpicture}

\begin{tikzpicture}[scale=1.0]
\def\sq{1.414}
\def\t{0.5}
\def\a{1.0+\sq};
\def\b{1};
\def\c{1-\t};
\def\d{\t};
\def\e{1.0+\sq-\t};
\def\rel{0.3}

\node (A) at (\a+\rel,0) [circle,draw,fill=black,inner sep=0pt,minimum size=1.0mm] {};
\node (B) at (\a+\b+\rel,0) [circle,draw,fill=black,inner sep=0pt,minimum size=1.0mm] {};
\node (C) at (0,\c) [circle,draw,inner sep=0pt,minimum size=1.0mm] {};
\node (D) at (\a,\c) [circle,draw,inner sep=0pt,minimum size=1.0mm] {};
\node (E) at (\a+\b,\c) [circle,draw,inner sep=0pt,minimum size=1.0mm] {};
\node (F) at (\rel,\c+\d) [circle,draw,fill=black,inner sep=0pt,minimum size=1.0mm] {};
\node (G) at (\a+\rel,\c+\d) [circle,draw,fill=black,inner sep=0pt,minimum size=1.0mm] {};
\node (H) at (\a+\b+\rel,\c+\d) [circle,draw,fill=black,inner sep=0pt,minimum size=1.0mm] {};
\node (I) at (0,\c+\d+\e) [circle,draw,inner sep=0pt,minimum size=1.0mm] {};
\node (J) at (\a,\c+\d+\e) [circle,draw,inner sep=0pt,minimum size=1.0mm] {};

\draw (A) -- (B) -- (E) -- (H)--(G) --(J) -- (I) -- (F) -- (C) --(D) -- (A);
\draw (E) -- (D); 
\draw (G) -- (F);

\end{tikzpicture}

\begin{tikzpicture}[scale=1.0]
\def\sq{1.414}
\def\t{0.2}
\def\a{1.0+\sq};
\def\b{1};
\def\c{1-\t};
\def\d{\t};
\def\e{1.0+\sq-\t};
\def\rel{0.3}

\node (A) at (\a+\rel,0) [circle,draw,fill=black,inner sep=0pt,minimum size=1.0mm] {};
\node (B) at (\a+\b+\rel,0) [circle,draw,fill=black,inner sep=0pt,minimum size=1.0mm] {};
\node (C) at (0,\c) [circle,draw,inner sep=0pt,minimum size=1.0mm] {};
\node (D) at (\a,\c) [circle,draw,inner sep=0pt,minimum size=1.0mm] {};
\node (E) at (\a+\b,\c) [circle,draw,inner sep=0pt,minimum size=1.0mm] {};
\node (F) at (\rel,\c+\d) [circle,draw,fill=black,inner sep=0pt,minimum size=1.0mm] {};
\node (G) at (\a+\rel,\c+\d) [circle,draw,fill=black,inner sep=0pt,minimum size=1.0mm] {};
\node (H) at (\a+\b+\rel,\c+\d) [circle,draw,fill=black,inner sep=0pt,minimum size=1.0mm] {};
\node (I) at (0,\c+\d+\e) [circle,draw,inner sep=0pt,minimum size=1.0mm] {};
\node (J) at (\a,\c+\d+\e) [circle,draw,inner sep=0pt,minimum size=1.0mm] {};

\draw (A) -- (B) -- (E) -- (H)--(G) --(J) -- (I) -- (F) -- (C) --(D) -- (A);
\draw (E) -- (D); 
\draw (G) -- (F);

\end{tikzpicture}

\end{tabular}
\caption[Moving in the bed of minimal sets]{Moving in the bed of minimal sets: 3 cylinder surfaces
  in $\EE_8(1,1)$, 
  deformed by imaginary rel.  }

\label{fig: 3 cylinder deformation}
\end{figure}

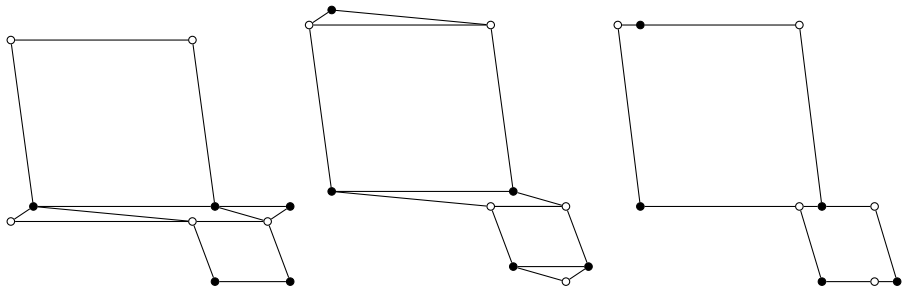
\begin{figure}[h]

\begin{tabular}{lcr}

\begin{tikzpicture}[scale=1.0]
\def\sq{1.414}
\def\t{0.2}
\def\a{1.0+\sq};
\def\b{1};
\def\c{1-\t};
\def\d{\t};
\def\e{-\t+1.0+\sq};
\def\rel{0.3}

\node (A) at (\a+\rel,0) [circle,draw,fill=black,inner sep=0pt,minimum size=1.0mm] {};
\node (B) at (\a+\b+\rel,0) [circle,draw,fill=black,inner sep=0pt,minimum size=1.0mm] {};
\node (C) at (0,\c) [circle,draw,inner sep=0pt,minimum size=1.0mm] {};
\node (D) at (\a,\c) [circle,draw,inner sep=0pt,minimum size=1.0mm] {};
\node (E) at (\a+\b,\c) [circle,draw,inner sep=0pt,minimum size=1.0mm] {};
\node (F) at (\rel,\c+\d) [circle,draw,fill=black,inner sep=0pt,minimum size=1.0mm] {};
\node (G) at (\a+\rel,\c+\d) [circle,draw,fill=black,inner sep=0pt,minimum size=1.0mm] {};
\node (H) at (\a+\b+\rel,\c+\d) [circle,draw,fill=black,inner sep=0pt,minimum size=1.0mm] {};
\node (I) at (0,\c+\d+\e) [circle,draw,inner sep=0pt,minimum size=1.0mm] {};
\node (J) at (\a,\c+\d+\e) [circle,draw,inner sep=0pt,minimum size=1.0mm] {};

\draw (A) -- (B) -- (E) -- (H)--(G) --(J) -- (I) -- (F) -- (C) --(D) -- (A);
\draw (G) -- (E) -- (D) -- (F); 
\draw (G) -- (F);

\end{tikzpicture}

\begin{tikzpicture}[scale=1.0]
\def\sq{1.414}
\def\t{0.2}
\def\a{1.0+\sq};
\def\b{1};
\def\c{1-\t};
\def\d{\t};
\def\e{1.0+\sq-\t};
\def\rel{0.3}

\node (A) at (\a+\rel,0) [circle,draw,fill=black,inner sep=0pt,minimum size=1.0mm] {};
\node (B) at (\a+\b+\rel,0) [circle,draw,fill=black,inner sep=0pt,minimum size=1.0mm] {};

\node (D) at (\a,\c) [circle,draw,inner sep=0pt,minimum size=1.0mm] {};
\node (E) at (\a+\b,\c) [circle,draw,inner sep=0pt,minimum size=1.0mm] {};
\node (F) at (\rel,\c+\d) [circle,draw,fill=black,inner sep=0pt,minimum size=1.0mm] {};
\node (G) at (\a+\rel,\c+\d) [circle,draw,fill=black,inner sep=0pt,minimum size=1.0mm] {};
\node (I) at (0,\c+\d+\e) [circle,draw,inner sep=0pt,minimum size=1.0mm] {};
\node (J) at (\a,\c+\d+\e) [circle,draw,inner sep=0pt,minimum size=1.0mm] {};

\node (K) at (\rel,\c+\d+\d+\e)  [circle,draw,fill=black,inner sep=0pt,minimum size=1.0mm] {};
\node (L) at (\a+\b,-\t) [circle,draw,inner sep=0pt,minimum size=1.0mm] {};

\draw (A) -- (B) -- (E) -- (G) --(J) -- (I) -- (F);
\draw (E) -- (D) -- (A); 
\draw (G) -- (F);
\draw(E) -- (D) -- (F); 
\draw (I) -- (K) -- (J);
\draw (A) -- (L) -- (B);

\end{tikzpicture}

\begin{tikzpicture}[scale=1.0]
\def\sq{1.414}
\def\t{0.0}
\def\a{1.0+\sq};
\def\b{1};
\def\c{1-\t};
\def\d{\t};
\def\e{1.0+\sq-\t};
\def\rel{0.3}

\node (A) at (\a+\rel,0) [circle,draw,fill=black,inner sep=0pt,minimum size=1.0mm] {};
\node (B) at (\a+\b+\rel,0) [circle,draw,fill=black,inner sep=0pt,minimum size=1.0mm] {};

\node (D) at (\a,\c) [circle,draw,inner sep=0pt,minimum size=1.0mm] {};
\node (E) at (\a+\b,\c) [circle,draw,inner sep=0pt,minimum size=1.0mm] {};
\node (F) at (\rel,\c+\d) [circle,draw,fill=black,inner sep=0pt,minimum size=1.0mm] {};
\node (G) at (\a+\rel,\c+\d) [circle,draw,fill=black,inner sep=0pt,minimum size=1.0mm] {};
\node (I) at (0,\c+\d+\e) [circle,draw,inner sep=0pt,minimum size=1.0mm] {};
\node (J) at (\a,\c+\d+\e) [circle,draw,inner sep=0pt,minimum size=1.0mm] {};

\node (K) at (\rel,\c+\d+\d+\e)  [circle,draw,fill=black,inner sep=0pt,minimum size=1.0mm] {};
\node (L) at (\a+\b,-\t) [circle,draw,inner sep=0pt,minimum size=1.0mm] {};

\draw (B) -- (E) -- (G) --(J); 
\draw (I) -- (F);
\draw (D) -- (A); 
\draw (D) -- (F); 
\draw (I) -- (K) -- (J);
\draw (A) -- (L) -- (B);
\draw (D) -- (G);

\end{tikzpicture}

\end{tabular}
\caption[Applying a cut and paste surgery]{Applying a cut and paste surgery,   one can 
continue the imaginary rel motion and arrive at a
  2-cylinder surface. }
 \label{fig: first 2 cylinder}
\end{figure}

Each  family of two-cylinder surfaces is in turn in the boundary of
exactly two  families of three-cylinder surfaces. 
Starting from a given  family of three-cylinder surfaces, we may
continue this process of moving in the family so that the cylinders
degenerate to one of the boundary two-cylinder families and then by
undoing a different degeneration we get to a different three-cylinder
family. 
After passing through a finite sequence of three-cylinder families, we
eventually return to the 
original one.  These families of minimal sets then comprise finitely
many disjoint cycles.

If $\mathcal{O}_D$ is a maximal order, the set of cycles is naturally
in bijection with the ideal 
class group of $\mathcal{O}_D$.  
The bijection is given by associating to a periodic surface in one
of these families the fractional ideal in $F_D$ obtained by taking the
$\zed$-span of the circumferences of 
its cylinders.  These circumferences are constant (up to scale) in
each two- or three-cylinder family of 
minimal sets, and when a three-cylinder family degenerates to a
two-cylinder family, one cylinder is 
lost but the fractional ideal generated by their circumferences is unchanged.
If  $\mathcal{O}_D$ is not maximal the analysis becomes more involved.
In the case where $D$ is square, the structure of the collection of
minimal sets is more complicated 
due to the presence of minimal sets consisting of one-cylinder
surfaces, in addition to two and three  
cylinder decompositions.

\subsection{Framing and splitting}\label{subsection: framing}
Let $\HH_{\mathrm{f}}(2) \to \HH(2)$ be the threefold cover
corresponding to choosing a right-pointing
horizontal prong at the singularity, as in \S\ref{subsec: framings}.
 This cover is connected, since as we have seen in Proposition \ref{prop: two pi},
 the $\til \SO_2(\R)$-orbits of a point contain all of its pre-images
 under the map $\HH_{\mathrm{f}}(2) \to \HH(2)$. We denote by
 $\widehat{\EE}_{D}(2)$ the pre-image of $\EE_D(2)$ in 
 $\HH_{\mathrm{f}}(2)$.
 By the same
reasoning as above, each connected component of $\EE_D(2)$ has connected inverse image.
Recall that since $U$ is simply connected it lifts to $\widehat{ G}$
as a subgroup, so
it acts on $\widehat{\EE}_{D}(2)$. 

\begin{prop}\label{prop: Hopf strikes back}
  The action of $U$ is ergodic on any connected component of $\widehat{\EE}_{D}(2)$.
\end{prop}

\begin{proof}

Since any connected component of $\EE_D(2)$ is a $G$-orbit (see
Theorem \ref{thm:closed-orbit-classification}), $\widehat{G}$ acts
transitively on each connected component of $\widehat{\EE}_{D}(2)$, and in
particular, ergodically. The inclusion $U \subset \widehat{G}$ has the
Mautner property (see
  e.g.\ \cite{EW}) and hence any ergodic $\widehat{G}$-action is also
  ergodic for $U$. 
\end{proof}

\subsection{Other $U$-invariant measures}\label{subsection:
  nonempty}
We now discuss the remaining ergodic $U$-invariant measures.  Each of these arise from applying the rel
flow to a closed $G$-orbit, possibly lying in a boundary stratum of $\EE_D$.  To avoid working with
compactifications, we will not always take this point of view explicitly, in favor of more elementary
constructions.

We classify the remaining measures in terms of their horizontal data
diagram $\Xi(\mu)$, defined in \S \ref{sec: separatrix diagram}. 
We have the following possibilities.

\begin{prop}\label{prop: configurations HD2}
  Let $\mu$ be an ergodic $U$-invariant measure, which is not
  supported on a minimal set. There are four possibilities for $\Xi(\mu)$: 
  \begin{itemize}
  \item[(i)] $\Xi(\mu) = \emptyset$.  
  \item[(ii)] $\Xi(\mu)$ consists of one saddle connection joining
    distinct singularities. For $\mu$-a.e. surface, this
    saddle connection is fixed by the hyperelliptic involution.
  \item[(iii)] $\Xi(\mu)$ consists of two saddle connections joining distinct singularities which
    are interchanged by the hyperelliptic involution.  For
    $\mu$-a.e. surface $M$, the union of these saddle
    connections disconnects $M$ into a pair of tori with slits
    removed.
  \item[(iv)] $\Xi(\mu)$ consists of two saddle connections joining distinct singularities which are
    fixed by the hyperelliptic involution.  For $\mu$-a.e. surface
    $M$, the complement of these saddle connections
    is a torus with two slits removed, and $D$ is a
    square. 
  \end{itemize}

\end{prop}

\begin{proof}
  This follows directly from Theorems~\ref{thm: loops in the spine} and \ref{thm: loops in the spine
    square case}.
\end{proof}

In the case where $M$ belongs to a minimal set, $\Xi(M)$ is one of
configurations pictured in Figure~\ref{fig: maximal}. 

Our goal in this section will be to provide examples of $U$-ergodic
$U$-invariant measures in each of these cases.

For fixed discriminant $D$, let $\CC$\index{C} denote one of the following:

\begin{itemize}
\item[(i)] The $G$-orbit of the regular decagon $\mathcal{L}_{\rm dec}\subset\EE_5(1,1)$, or if $D$ is square, the
  $G$-orbit of a square-tiled surface in $\EE_D(1,1)$.
\item[(ii)] A component of $\widehat{\EE}_{D}(2)$ in $\HH_{\mathrm{f}}(2)$.
\item[(iii)] A component of $\mathcal{P}_D\subset \HH(0) \times \HH(0)$.
\item[(iv)] If $D = d^2$, the $d$-torsion locus $\mathcal{S}_{D}\subset\HH(0,0)$.
\end{itemize}
Note that in this case and subsquently we use the
same symbol to denote subsets of different spaces; that is, more
formally we should replace $\CC$ with one of $\CC({\rm i})$--$\CC({\rm iv})$
depending on the cases above. We will continue with this slight
inaccuracy with the set $\BB$ and map $\Psi$ below, and this should
cause no confusion.

In each case, $\CC$ has a natural $G$ or $\widehat{G}$ action, and $\CC$ is isomorphic to $G/\Gamma$ or
$\widehat{G}/\Gamma$ for some lattice $\Gamma$.
Let $\BB\subset\CC$\index{B@$\BB$} denote the subset consisting of surfaces without
horizontal saddle connections.
In case (iii) or (iv), we interpret a saddle connection to be either a closed geodesic or a segment
joining two marked points.

\begin{prop}
  \label{prop:complement of horocycles} The set $\BB\subset\CC$ is the complement of the set of closed
  horocycles.
\end{prop}

\begin{proof}
  In cases (i) and (ii) this follows from the Veech alternative \cite{Veech - alternative}.

  In case (iii), recall that every point of $\CC$ represents a pair of isogenous genus one
  translation surfaces, so the horizontal direction is periodic for one if and only if it is
  periodic for the other.  A torus with periodic horizontal direction is stabilized by an infinite
  cyclic subgroup of $U$, and for two isogenous tori, these subgroups are commensurable.  Thus a
  point of $\CC$ representing a pair of tori with periodic horizontal direction lies on a closed
  horocycle.

  In case (iv), every point of $\CC$ represents a torus with two marked points $(E, p,q)$ where
  $p-q$ is $d$-torsion.  It has a horizontal saddle connection exactly when the horizontal direction
  of $E$ is periodic, in which case $E$ is stabilized by an infinite cyclic subgroup of $U$.  Given
  $E$, there are finitely many choices of $p$ and $q$ which differ by $d$-torsion, up to isomorphism of
  $E$.  Thus $(E,p,q)$ is also stabilized by an infinite cyclic subgroup of $U$, so lies on a closed
  horocycle.
\end{proof}

Note that $\BB$ is a dense $G_\delta$ subset of a locally compact
metrizable space, and hence the Borel $\sigma$-algebra structure on 
$\BB$ is a standard Borel 
space (see \cite{Kechris}). 
\begin{cor}
  \label{cor:uniquely ergodic}
  In each case the action of $U$ on $\BB$ is uniquely ergodic,
  i.e. there is a unique $U$-invariant regular Borel measure on $\BB$.
\end{cor}

\begin{proof}
  By Proposition~\ref{prop:complement of horocycles}, the action of $U$ on $\BB$ is measurably
  conjugate to the horocycle flow on the complement of the set of horocycles in $G/\Gamma$ or
  $\widehat{G}/\Gamma$ for some lattice $\Gamma$.  Dani \cite{Dani 78} classified the $U$-invariant
  measures on  $G/\Gamma$ and $\widehat{G}/\Gamma$, showing that they are either supported
  on closed horocycles or are the global measure induced by Haar measure. Thus Dani's theorem
  implies that the $U$-action on each $\BB$ is uniquely ergodic.
\end{proof}

In each of the four cases introduced above, 
we define a map $\Psi \colon \BB \times (0, \infty) \to \HH(1,1)$\index{psi@$\Psi$} as follows:

\begin{itemize}
\item[(i)]  $\Psi(M, T) = \rel_TM. $ This is well-defined by
  Corollary~\ref{thm: criterion}.
\item[(ii)] $\Psi(M, T)$ is the inverse of the map $\Phi_{\mathrm{f}}^{(-T)}$ of Theorem
  \ref{thm: criterion1}. Informally, $\Psi(M,T)$
  is the surface obtained by splitting the singularity of angle $6\pi$ to two singularities
  of angle $4\pi$ each, and moving them apart using $\rel_T$. The way
  in which the two singularities are to be split apart is made
  explicit by the framing.  
\item[(iii)] Given a pair of genus one translation surfaces $(E_1,
  E_2)\in \BB \subset\mathcal{P}_D$, we
  define $\Psi((E_1, E_2), T)$ to be the surface obtained by removing a horizontal slit of length
  $T$ from each $E_i$ and then gluing the resulting surfaces along their boundaries to obtain a
  genus two surface in $\HH(1,1)$.  The resulting surface is an eigenform, since real multiplication
  is preserved by deformations which leave absolute periods constant by Proposition~\ref{prop:
    detecting RM}. 
\item[(iv)]  Given a genus one surface $E$ with aperiodic horizontal
  foliation 
and with marked points $p,q$ which differ by $d$-torsion, we
  define $\Psi((E, p,q), T)$ to be the surface obtained by cutting $E$ along two horizontal
  slits of length $T$ with left endpoints $p$ and $q$ and then gluing the resulting two boundary
  components to obtain a genus two surface in $\HH(1,1)$.  We saw above that there is a primitive
  degree $d$ torus covering $\pi\colon E\to F$ sending $p$ and $q$ to the same point.  Our gluing
  identifies only pairs of points which have the same image, so the resulting surface is also a
  primitive degree $d$ torus cover belonging to $\EE_{d^2}(1,1)$.   
\end{itemize}

In each case, we label the singularities of the resulting surface so that singularity $\xi_1$ is on the
left hand side of the horizontal saddle connection of length $T$, and singularity $\xi_2$ is on the right
hand side.

We extend the definition of $\Psi(M,T)$ to $T<0$ as follows. 
In case (i), we simply apply $\rel_T$.  In the remaining cases, we repeat
the same construction using $|T|$ in place of $T$, but choose the opposite
labeling for the singularities.

For fixed $T \neq 0$, let $\Psi_T : \BB
\to \EE_D(1,1)$ be defined by $\Psi_T(M) = \Psi(M, T)$.  In cases (ii), (iii), and (iv), we denote by $\BB_T$ the
image of $\Psi_T$.

We summarize this discussion in the following:
\begin{prop}\label{prop: tear}
For each $T$, 
$\Psi_T$ is continuous on $\BB$. In case (i), the image is a set of surfaces with no
horizontal saddle connections. In the remaining cases, the image
$\BB_T$ is the set of surfaces in 
$\EE_D(1,1)$ with
\begin{itemize}
\item in case (ii) there is exactly one horizontal saddle connection,
 with length $|T|$, which  joins distinct singularities. 
\item in case (iii) there are exactly two horizontal saddle
  connections of length $|T|$ joining distinct singularities which are 
  interchanged by the hyperelliptic involution. 
\item in case (iv) there are exactly two horizontal saddle connections
  of length $|T|$ joining distinct singularities which have the 
  same length and are both fixed by the hyperelliptic involution.
\end{itemize}
These saddle connections are all positively oriented in the case $T>0$
and negatively oriented
in the case $T<0$.

In each case, 
the pushforward under $\Psi_T$ of the $G$-invariant (or
$\widehat{G}$-invariant) 
measure on $\BB$ 
is the unique $U$-invariant and $U$-ergodic measure $\mu$ on
$\EE_D(1,1)$, such that $\Xi(\mu)$ is 
as in cases (i)-(iv) above. 
\end{prop}

\begin{proof}
 Most of the statements in the Proposition were established in the
 preceding discussion.  To obtain 
  continuity of $\Psi_T$, in case (i), continuity follows from Proposition \ref{prop:
    continuous}, and in case (ii), from Theorem \ref{thm:
    criterion1}. In the other cases, it follows immediately from the
  explicit definition of the 
  surgeries described above that continuity holds in a neighborhood of
  $M$, for all $T$ sufficiently 
  small (depending on this neighborhood). Continuity for arbitrary $T$
  follows by combining this 
  with Proposition \ref{prop: continuous}.

We  now show that for each of the cases (i)--(iv),
each surface $M$ with $\Xi(M) = \Xi(\mu)$ as described in the list, is
in the image of the map  
  $\Psi_T$. To show that $M$ is in the image of $\Psi_T$, we define the
  inverse of $\Psi_T$ by inverting the surgery construction used to
  define it. Thus for case (ii), we use Theorem
  \ref{thm: criterion1}. In case (iii), $\Psi_T$ associates to a pair
  of isogenous tori $(E, F)$, the connected sum 
  $E\#_I F$.  Given a surface $M$ with two horizontal saddle
  connections interchanged by the 
  hyperelliptic involution, we cut $M$ along these saddle connections
  to obtain a pair of isogenous 
  tori $(E, F)$, and so $M = \Psi_T(E,F)$.  The other cases are
  similar. 

The uniqueness of the measures $\mu$ for which
$\Xi(\mu)$ is as in each of the cases above, now follows from the unique ergodicity of the
  $U$-action on the corresponding set
  $\BB$. 
\end{proof}

\section{Classification of ergodic measures}\label{sec: classification}
In this section we show that the only $U$-invariant
$U$-ergodic measures in $\EE_D(1,1)$ are those described in Section 
\ref{section: construction}.  The following is our
full measure classification result, which is an expanded version of
Theorem \ref{thm: Calta Wortman 
  revisited}. 

\begin{thm}\label{thm: Calta Wortman revisited2} 
Let $\mu$ be a $U$-invariant $U$-ergodic Borel probability measure on
$\EE_D$. Then exactly one of the following cases holds: 
\begin{enumerate}

\item 
$\mu$ is length measure on a periodic $U$-orbit, 
$\Xi(\mu)$ is a complete separatrix diagram of type (A), (B), (C) or
(D) (see Figure~\ref{fig: maximal}), and all the cylinders in the corresponding cylinder
decomposition have commensurable moduli. 

\item 
$\mu$ is the flat measure on a 2-dimensional torus which is a  $U$-minimal set,
$\Xi(\mu)$ is a complete separatrix 
diagram of type (A), and the cylinders in the corresponding cylinder
decomposition do not have commensurable moduli. The stabilizer of
$\mu$ is $U \times Z$, and $U \times Z$ acts transitively on $\supp \,
\mu$. 

\item
$\Xi(\mu)$ consists of two saddle 
connections joining distinct singularities, which disconnect $M$. The
complement consists of 
two isogenous tori glued along a slit as in case (iii) of \S
\ref{subsection: nonempty}, and there is a
$T \in \R \sm \{0\}$ such that $\mu$ is 
the pushforward of Haar measure on a connected component of
$\mathcal{P}_D$, via the map $\Psi_T$. 

\item
$\Xi(\mu)$ consists of one saddle connection
joining distinct singularities, as in case (ii) of \S
\ref{subsection: nonempty}, and there is a $T \in \R \sm \{0\}$
such that $\mu$ is the image under $\Psi_T$ of Haar measure on a
finite volume $\widehat{G}$-orbit in $\HH_{\mathrm{f}}(2)$, where
$\widehat{G}$ and $\HH_{\mathrm{f}}(2)$ are the threefold covers of $G$ and $\HH(2)$
described in \S \ref{subsec: action}.

\item
$\Xi(\mu)=\varnothing$ and there is $T \in \R$ such that $\mu$ is the
image under $\rel_T$, of Haar measure on a closed $G$-orbit in
$\EE_D(1,1)$. In this case $D$ is either a square or is equal 
to 5. 
\item $\Xi(\mu)$ contains two saddle connections joining distinct
  singularities, whose complement in 
  $M$ is a torus with two parallel slits of equal length, as in case
  (iv) of \S \ref{subsection: 
    nonempty}, and there is a $T \in \R \sm \{0\}$ such that $\mu$ is
  the image under $\Psi_T$ of a 
  $G$-invariant measure on the space of tori with two marked
  points. In this case $D$ is a square. 
\item
$\mu$ is the  flat
measure on $\EE_D(1,1)$. 

\end{enumerate}
\end{thm}

\begin{dfn}\label{dfn: types of measures} The numbering in Theorem
  \ref{thm: Calta Wortman revisited2} corresponds to  
that in Theorem \ref{thm: Calta Wortman revisited}. If $\mu$ is a
measure satisfying the conditions of item $(j)$ in Theorem \ref{thm:
  Calta Wortman revisited2} (where
$j \in \{1, \ldots, 7\}$) then we
will refer to $j$ as the {\em type} \index{type of a measure} of 
the measure $\mu$. 
\end{dfn}

Recall from Proposition \ref{prop: local L action} and Corollary
\ref{cor: frobenius} that $\EE_D(1,1)$ is an affine manifold in
period coordinates, that the tangent space
to $\EE_D(1,1)$ at a point $M$ is naturally identified with the Lie
algebra $\mathfrak{l} \cong \mathfrak{sl}_2(\R) \ltimes \R^2$, \index{l@$\mathfrak{l}$} and
every Lie subalgebra of $\mathfrak{l}$ defines a foliation of
$\EE_D(1,1)$. \index{l@$\mathfrak{l}$}
Theorem \ref{thm: Calta Wortman revisited2} can be reformulated in Lie
algebra terms, as follows. For any ergodic $U$-invariant
measure $\mu$ on $\EE_D(1,1)$, there is a Lie subalgebra
$\mathfrak{l}_{\mu}$ \index{lm@$\mathfrak{l}_{\mu}$} such that one
leaf of the foliation corresponding to $\mathfrak{l}_\mu$ is of full
$\mu$-measure, and $\mu$ is invariant under the vector fields
corresponding to the elements of $\mathfrak{l}_\mu$. 
The following list identifies $\mathfrak{l}_{\mu}$ in each of the
cases of Theorem \ref{thm: Calta Wortman revisited2}:
\begin{enumerate}
\item
 $\mathfrak{l}_{\mu} = \mathfrak{u} = \Lie (U).$ 
\item
$\mathfrak{l}_{\mu} = \mathfrak{u} \oplus \mathfrak{z}$\index{z@$\mathfrak{z}$} where $\mathfrak{z}$ is the Lie algebra of $Z$.
\item
$\mathfrak{l}_{\mu}$ is the Lie algebra of the subgroup of $L$
fixing
the holonomy of the two saddle connections in $\Xi(\mu)$ 
(note that these vectors have the same holonomy). 
\item
$\mathfrak{l}_{\mu}$ is the Lie algebra of the subgroup of $L$ fixing
the holonomy vector of either of the two saddle connections in $\Xi(\mu)$
(note that these vectors have the same holonomy).
\item
$\mathfrak{l}_{\mu}$ is the Lie algebra of the subgroup of $L$ fixing
the vector $(T,0) \in \R^2$.  
\item
$\mathfrak{l}_{\mu}$ is the Lie algebra of the subgroup of $L$ fixing
the holonomy of either of the two vectors in $\Xi(\mu)$ . 
\item
$\mathfrak{l}_{\mu} = \mathfrak{l}$. 
\end{enumerate}

The proof of Theorem \ref{thm: Calta Wortman revisited2} occupies the
rest of this section. It is modeled on arguments of Ratner 
dealing with horocycle orbits in homogeneous spaces
(see \cite{Ratner SL2} for a survey), 
which were first applied to spaces of translation surfaces in \cite{EMM} and applied to the eigenform locus by
Calta and Wortman in \cite{CW}. We follow the outline of \cite{CW}
but make several improvements. These enable us to bypass some entropy
arguments used in \cite{CW} and give a clearer justification of Proposition
\ref{prop: CW main} below. 

\subsection{The structure of Ratner's proof} \label{subsec: Ratner}
Ratner's argument hinges on the analysis of transverse divergence of nearby trajectories for
the $U$-action which she calls the {\em R-property} (see \cite[ p. 22]{Ratner SL2}). 
Suppose we want to compare two orbits $u_tM$ and $u_tM'$ where $M$ and $M'$ are close.
We can write $M'=gM \pluscirc v$. The case when $(g,v)\in N$ is somewhat special and we consider
it first. In this case $g$ normalizes $U$ and satisfies
 $u_tg =
gu_{\lambda t}$ for some $\lambda >0$ independent of $t$, and therefore 
$$u_tM'=u_t(g,v)M=(u_tg,v)M=(gu_{\lambda t},v)M=(g,v)u_{ \lambda t}M.$$
The divergence is caused by two factors. Since $(g,v)$
is independent of $t$ and small we think of $u_tM'$ and $u_{\lambda t}M$ as
being close. In particular 
if $\lambda $ is not 1 then the primary divergence of these two points is in
the horocycle direction. Rescaling time along the second orbit by
replacing 
$t$ with $\lambda t$ has the effect of removing this
divergence. 
If we refer to the remaining
orbit divergence as transverse divergence, then in case $(g,v) \in N$
we see that the transverse divergence is constant. 

Now let us consider the general case where $M'=gM \pluscirc v$ with
$\displaystyle{ 
g = \left( \begin{matrix} a & b \\ c & d\end{matrix} \right) \in G,}$ and $v$ arbitrary.
 In order to pick out the divergence in the directions 
transverse to the horocycle direction we will rescale the time factor along the second orbit. 
Specifically let $u_t M$ be the first orbit and write the second orbit as $u_s M'$
with $s=s(t)$. Unlike the previous case, the time parameter $s(t)$ for
the second orbit will depend on the time parameter $t$ in the first orbit
in a nonlinear way and the rescaling will only work in a finite
interval of time. We will show that the time change is not too far
from being linear. 
  In order to do this we choose a family of transversals to the horocycle orbits
and choose the function $s(t)$ so that  $u_t M$ and $u_s M'$ stay in the
same transversal. Recall from Proposition \ref{prop: local L action}
that locally we can model $\EE_D(1,1)$ on 
$L$. We choose as our family of transversals the elements of $L$
corresponding to elements 
$(g,v)$ where $g$ is lower triangular and $v$ is in $\R^2$, and both
$g, v$ are small. This defines a submanifold transverse to $U$ of
complementary dimension.

Since $M'=gM \pluscirc v$
 is defined, Proposition \ref{prop:
  distributive} tells us that $u_s g M \pluscirc u_s v$ is also defined and is
equal to $u_s(gM \pluscirc v) = u_sM'$. Since $u_s gM = (u_s g u_{-t})
u_t M$, the pair $(u_s g u_{-t}, u_s v)$, considered as an element of
$L$, gives a well-defined map which moves $u_t M$ to $u_s M'.$ 
We calculate: 
$$u_sgu_{-t}=\begin{pmatrix} a+sc&b-at+s(d-ct)\\c& d-ct\end{pmatrix}.$$

In order for $u_tM$ and $u_s M'$ to lie on the same transversal, we
require that the matrix $u_sgu_{-t}$ be close
to the identity and be lower triangular.
\begin{dfn}
For $t \in \R$ we set 
\begin{equation}\label{eq: formula for s}{ 
s=s(t,g) =\frac{at-b}{d-ct}.
}\end{equation}
\end{dfn}
With this choice we have:
\begin{equation}\label{eq: see that}{ 
u_{s(t)}gu_{-t} = \left ( \begin{matrix} \frac{1}{d-ct} & 0 \\ c &
    d-ct\end{matrix} \right)}\end{equation}
(here $s(t) = s(t,g)$), 
so that $u_tM$ and $u_{s(t)}M'$ lie in the same transverse leaf as long as $t$ is chosen 
such that $d-tc$ is close to 1.

Note that while the initial displacement 
may consist of displacement in the Rel direction as well as the
$G$-direction, the time change only depends on the initial displacement in the $G$ direction. 

Let $\mu$ be a $U$-invariant ergodic measure on $\EE_D(1,1)$ with
$\Xi(\mu) = \varnothing$. Since $N$ acts on the set of surfaces with $\Xi(\mu) = \varnothing$ 
we can consider the subgroup of $N$ that preserves $\mu$. Let $N_\mu^\circ$ to be the
connected component of the identity in $N_\mu$, where $N_\mu$ is the
stabilizer of $\mu$ as in 
Definition \ref{dfn: stabilizer}.

\begin{prop}
\label{prop: CW main} Let $\mu$ be a $U$-invariant ergodic measure on $\EE_D(1,1)$ with
$\Xi(\mu) = \varnothing$.
For any $\vre>0$ there is a compact subset $K$ of $\EE_D(1,1)$,
consisting of generic points for $\mu$, 
with $\mu(K)\geq 1-\vre$ such that if
$M \in K$ is a limit of $M_k = h_k M \pluscirc v_k \in K$, with $h_k
\in G$ and $ v_k = (x_k, y_k) \in \R^2$, and $(h_k, v_k) \to (1_G,0)$,
$(h_k, v_k) \notin U$ for infinitely many $k$, then
$\dim N_\mu^\circ \geq 2.$ 
\end{prop}

\begin{proof}
Let $\Omega_\mu$ be the set of $\mu$-generic points for the
$U$-action. The set $\Omega_\mu$ has full $\mu$ measure.
Given $\vre$, let $\Omega_1$ be a compact subset of $\Omega_\mu$ such that
$\mu(\Omega_1)\geq 1-\vre$ and every surface in $\Omega_\mu$ satisfies $\Xi(\mu) = \varnothing$. By the Birkhoff ergodic theorem there is
a compact subset $K$ of $\Omega_\mu \cap \EE_D(1,1)$ of measure
$\mu(K) \geq 1-\vre$, and a $T_0>0$, such 
that for all $M \in K$ and all $T>T_0$, 
\begin{equation}\label{eq: numbered}{\frac{1}{T} \left|\left \{t \in [0,T]: u_tM \in \Omega_1 \right \} \right| \geq 1-2\vre.
}\end{equation}
We will show that this choice of $K$ satisfies the conclusion of the
Proposition. So let $M_k$, $M$, $h_k$ and $v_k = (x_k, y_k)$ be as in the statement of the Proposition.   Write $c_k =
c(h_k)$ and  $d_k = d(h_k)$.

By Corollary \ref{cor: closed}, $N_{\mu}$ is a closed subgroup of $N$  and hence a Lie group.  
Thus $N_\mu^\circ$ is a connected Lie group that contains $U$. In order to show that $\dim N_\mu^\circ \geq 2$ it suffices
to rule out the possibility $U = N_\mu^\circ$, that is, in any neighborhood
of the identity in $N$ we need to find an element in $N \sm U$ preserving
$\mu$. If the sequence $(h_k, v_k)$ contains infinitely many elements of $N\sm U$ then
it follows from Corollary \ref{cor: Katok Spatzier} that we have such a sequence of elements.
We may thus assume that $(h_k, v_k) \notin N$ for infinitely many $k$.

Write 
elements of $Z$ as $\rel_s$ (where defined). Assume first that
there is a subsequence  of indices so that
\begin{equation}\label{eq: assume first that}{
y_k= o(c_k).
}\end{equation}
In this case we will show that for each sufficiently small 
$\delta>0$, there is an element 
\begin{equation}\label{eq: in this
  case1}{
\ell = \ell(\delta)=g_{\tau} \rel_\sigma }
\end{equation}
preserving $\mu$, such that 
\begin{equation}\label{eq: in this case2}{\frac{\vre \delta}{2} \leq  \tau \leq 
2\delta \ \text{ and } 
\sigma = O(\delta). 
}\end{equation}  
The fact that $\ell$ can be written as $g_{\tau} \rel_\sigma$ means that it lies in $N$.
The fact that $\tau\ne 0$ means that $\ell$ does not lie in $U$.

In light of Corollary \ref{cor: Katok Spatzier},  in order to show that $\ell \in N$ belongs to
$N_{\mu}$, it suffices to prove that
there are surfaces $M^{(1)}$ and $M^{(2)}$ in $\Omega_\mu$ such that $\ell M^{(1)}=M^{(2)}.$ We will
find $M^{(1)}$ and $M^{(2)}$ as limits of convergent subsequences  $M_j\to M^{(1)}$ with
$M_j\in\Omega_1$  and $g_jM_j\pluscirc v_j \to M^{(2)}$ with $g_jM_j\pluscirc v_j \in\Omega_1$ using the compactness of $\Omega_1$. We will show that 
$g_j$ converges to an element of $g_\infty\in B$ and $v_j$ converges to an element of $v\in Z$. Let $\ell=(g_\infty,v)$.
Since elements of $K$ have no horizontal saddle connections and $g_{\infty}\in B$ it follows that $g_{\infty}M^{(1)}$ has no horizontal saddle connections. Since $v\in Z$, Theorem \ref{thm: Bainbridge-Smillie surgery} then tells us that $\rel_v(g_{\infty}M^{(1)})$ exists so we can apply Proposition \ref{prop: limit} to conclude that $\ell(M^{(1)})=M^{(2)}$.

Recall that in order for $u_tM$ and $u_{s(t)}M'$ to be close it is necessary that $t$ be chosen so that $d-tc$
is close to 1. 
\begin{dfn} For $g \in G$ set 
$$I_g = \{t \in
\R: |\ln (d-tc)| \leq \delta\}, \text{ where } d =d(g), c=c(g).
$$
\end{dfn}
An elementary computation shows that there is a
neighborhood $\mathcal{W}$ of $1_G$ such that if $g \in
\mathcal{W}$ then $0 \in I_g$, and for any $t_1, t_2$ in the connected component
of $I_g$ containing $0$, we have 
\begin{equation}\label{eq: bilipschitz}{
\frac{|t_1-t_2|}{2} \leq |s(t_1,g) - s(t_2,g)| \leq 2 |t_1-t_2|.
}\end{equation}
Here the $s(t_i,g)$ are defined by equation \eqref{eq: formula for s} and $\delta$ is assumed to
be sufficiently small.

Let $P(t)$ be any linear function, that is $P$ is of the form $P(t) =
At + B$ for some $A$ and $B$, and for $a<b$ let 
$\|P\| = \max_{x \in [a,b]} |P(x)|.$ Then it is easy to check that 
\begin{equation}\label{eq: linear maps}{
\frac{\left|\left\{t \in [a,b]: |P(t) | < \vre \|P\|
    \right\}\right|}{b-a} <2 \vre.  
}\end{equation}

The definition of $I_g$ and equation \eqref{eq: see that} ensure that when $t$ is at the
an endpoint of $I_g$, the $g_t$-component of the displacing element $u_sgu_{-t}$ is $2\delta$. 
Namely, 
given $c \neq 0$ and $d$, the formula $T = T(c,d)=\frac{d-e^{\pm
    \delta}}{c}$ gives the two solutions to the equation $|\ln(d-cT)| = \delta$. Since
$d_k \to 1$ and $c_k\to 0$, for all sufficiently large $k$ there is a
unique such $T_k=T(c_k, d_k)$ with $T_k>0$, and moreover for all
sufficiently large $k$, we will have $T_k\geq T_0$. 
Here  we take $\delta$ sufficiently small and $k$
sufficiently large so that 
\begin{equation}\label{eq: one more}{\frac12 \leq \frac{|d_k-c_kt-1|}{|\ln(d_k-c_kt)|}
  \leq 2.
}\end{equation}
We apply the above estimates 
for $T=T_k$ and $P(t) = d_k-c_kt-1$ and obtain the existence of a subset
of $t \in [0,T]$, of measure at least $1-8\vre$, for which all of the
following hold:
\begin{itemize}
\item
$u_tM \in \Omega_1$ (by \eqref{eq: numbered}). 
\item
$u_sM_k \in \Omega_1,$ where $s=s(h_k, t)$ (by \eqref{eq: bilipschitz},
\eqref{eq: numbered} and since $M_k \in K$). 
\item
$|d_k-c_kt-1| \geq \frac{\vre \delta}{2},$ and hence $|\ln(d_k-c_kt)| \geq
\frac{\vre\delta}{4}$ (by \eqref{eq: linear maps}
and \eqref{eq: one more}).   
\end{itemize}
In particular, if we choose $\vre<1/8$ (which we may with no loss of
generality), then 
the set of $t$ satisfying 
these three properties is nonempty. Fixing such a $t$, 
it follows from \eqref{eq: see that} and Proposition \ref{prop: distributive} that 
\begin{equation}\label{eq: need to show}{
\begin{split}
u_s M_k &= u_s( h_kM \pluscirc v_k ) = u_s h_k u_{-t} u_t M \pluscirc
u_s v_k \\ 
& = \left ( \begin{matrix} \frac{1}{d_k-c_kt} & 0 \\ c_k &
    d_k-c_kt\end{matrix} \right) u_tM \pluscirc \left(x_k+sy_k ,y_k \right) .
\end{split}
}\end{equation}
For each $k$ we let $t_k$ and $s_k = s(h_k, t_k)$ satisfy \eqref{eq:
  need to show}. Using the compactness of
$\Omega_1$ we can take a subsequence along which we have convergence
$$u_{t_k} M \to M^{(1)}, \ u_{s_k}M_k \to M^{(2)}.$$ 
By \eqref{eq: need to show}, we have $u_{s_k} M_k  = \ell_k u_{t_k} M$ where 
$$ 
\ell_k=\left[ 
\left ( \begin{matrix} \frac{1}{d_k-c_kt_k} & 0 \\ c_k &
    d_k-c_kt_k\end{matrix} \right),
\left(x_k+
s(t_k, h_k) y_k
,y_k \right)
\right].
$$
We claim that $M^{(2)} = \ell M^{(1)}$, for $\ell \in N$ satisfying
the bounds \eqref{eq: in this case1}, \eqref{eq: in this case2}. 
It is enough to show that after passing to a further subsequence,
the sequence $\ell_k$ converges to an element $\ell$ satisfying these
bounds. Since 
$|\ln(d_k-c_kt)| \geq \frac{\vre\delta}{4}$, since $c_k \to 0$, and by
definition of the interval $I_{h_k}$, 
the matrices $\left ( \begin{matrix} \frac{1}{d_k-c_kt_k} & 0 \\ c_k &
    d_k-c_kt_k\end{matrix} \right)$ converge to $g_{\tau}$ with
$\tau$ as in \eqref{eq: in this case2}. Since $x_k, y_k \to 0$, in
order to obtain the required bound on $\sigma$ we need
to show that $s_k y_k = O(\delta)$. But 
\[
\begin{split}
s_k y_k  = & y_k \, \frac{a_kt_k - b_k}{d_k-c_kt_k}  = \frac{y_k}{c_k}
 \, \frac{a_kt_k - b_k}{d_k/c_k-t_k} \\ 
\stackrel{\eqref{eq: assume first that}}{=} & o(1) \frac{a_kt_k
    - b_k}{c_k/d_k - t_k} \stackrel{g_k \to 1_G}{=} o(1) = O(\delta), 
\end{split}
\]
as required. 

In case the condition of equation \eqref{eq: assume first that} does not hold, we have 
$$c_k = O(y_k).$$
In this case we can employ a similar argument, and find an element
$\ell(\delta)=g_{\tau} \rel_\sigma$ as in \eqref{eq: in this 
  case1}  
   which satisfies (instead of \eqref{eq: in this case2}) the estimates
\begin{equation}\label{eq: further estimate}{
\sigma \geq c\delta \ \text{ for some } c>0, \text{ and } \tau=
O(\delta), \ \sigma = O(\delta). 
}\end{equation}
We leave the details of this computation to the reader. In particular
\eqref{eq: further estimate} implies that $\ell(\delta) \notin U$, as required. 
\end{proof}

\begin{remark}
 The proof of \cite[Lemma 2]{CW}, which is the statement analogous to
Proposition \ref{prop: CW main}, is incomplete. The formula preceding the final paragraph 
of the proof in \cite{CW} would hold for an honest group action, but
in this case requires justification.
\end{remark}

\begin{proof}[Proof of Theorem \ref{thm: Calta Wortman revisited2}]
To each ergodic horocycle invariant measure $\mu$ we associate a
horizontal data diagram via Corollary \ref{cor: horizontal relations}. The
possibilities for $\Xi(\mu)$ given in the statement of Theorem \ref{thm: Calta Wortman revisited2} 
cover all cases by Proposition~\ref{prop: configurations HD2}. In \S\ref{section: construction}, for
each possibility for $\Xi(\mu)$, a measure was constructed. If it is a
complete separatrix diagram, then by Proposition \ref{prop: minimal}
and Corollary \ref{cor: minimal}, we must be in cases (1) or (2). In
all other cases for which $\Xi(\mu) \neq \varnothing$, it was shown in
Proposition \ref{prop: tear} that there
are no additional measures. This shows that if $\Xi(\mu)\neq
\varnothing$, then $\mu$ must have type (1), (2), (3),
(4) or (6).

We now treat the case $\Xi(\mu)=\varnothing$. 
Recall from Definition~\ref{dfn: aitch prime sub infinity} and
Corollary \ref{thm: criterion} that $\HH'_\infty$ is the set of surfaces in $\HH(1,1)$
with no horizontal saddle connections joining distinct singular points. 
Then $\mu(\HH'_\infty)=1$ and $N= B \ltimes Z$
acts on $\HH'_\infty$. Recall that  $N_\mu^\circ$ is the connected component of the stabiliser of
$\mu$ so that $U \subset N_\mu^\circ$. We first claim that there is no 
$U$-orbit with positive measure. Indeed, if this were the case then by ergodicity, $\mu$ is supported
on one orbit which, by Poincar\'e recurrence, must be a periodic
orbit. A periodic orbit is a special case of a minimal set so by
Proposition \ref{prop: minimal}, $\Xi(\mu)$ consists of a full separatrix
diagram. This contradicts our assumption that $\Xi(\mu)=\varnothing.$

 We now claim that we cannot have
$N_\mu^\circ=U$, i.e.\ that $\dim N_\mu^\circ \geq 2$. To see this, let $\Omega_\mu \subset \HH'_\infty$ be the set
of generic points as in Proposition \ref{prop: Katok Spatzier}, and
let $K$ be as in Proposition \ref{prop: CW main}. We have $K
\subset \Omega_\mu$.
Since
$\mu$ assigns zero measure to individual $U$-orbits and $K$ is a
compact set of positive measure, $K$ contains 
a sequence $(M_k)$ such that $M_k \to M \in
K$ and none of the $M_k$ are on the $U$-orbit of $M$. According
to Proposition \ref{prop: CW main}, $\dim N_\mu^\circ \geq 2$. 

Now suppose $Z \subset N_\mu^\circ$. We will show that in this case $\mu$ is the flat measure
on $\EE_D(1,1)$. Let $\mathcal{L}$ be an affine
$G$-invariant subspace of $\HH(1,1)$, and let $\mathcal{L}^{(1)}$\index{L1@
$\mathcal{L}^{(1)}$} be
its subset of area-one surfaces. We define the {\em
  horospherical foliation for $\{g_t\}$}\index{horospherical foliation} to be the foliation into
leaves locally modelled on the intersection of the tangent space of
$\mathcal{L}^{(1)}$ with the horizontal space $H^1(S, \Sigma; \R_x)$ (the
first summand in the splitting \eqref{eq: splitting}).  This
terminology is motivated by the fact that the $\{g_t\}$-flow is
non-uniformly hyperbolic and the above leaves are generically its
strong unstable leaves. 
For the
case $\mathcal{L} = \EE_D(1,1)$, the leaves of the horospherical
foliation are just the orbits of $UZ$, so in the case we are now considering, the measure
$\mu$ is invariant under the horospherical foliation, and must be flat
measure in light of the following:

\begin{claim}\label{claim: linchpin}
Suppose that:
\begin{itemize}
\item
$\mathcal{L}$ is $\GL_2(\R)$-invariant and affine in period coordinates,
 and is defined by real-linear equations in $H^1(S, \Sigma; \R^2)
 \cong H^1(S, \Sigma; \C).$ 
\item
The flat measure on $\mathcal{L}^{(1)}$ obtained from applying the cone
construction to $\mathcal{L}$ is $G$-invariant and ergodic.

\end{itemize}
Then any strong-stable invariant measure $\mu$ on $\mathcal{L}^{(1)}$, with $\Xi(\mu) = \varnothing$
coincides with the flat measure.
\end{claim}

This result was proved by Lindenstrauss and Mirzakhani in \cite{LM}
under the additional hypothesis that $\mathcal{L}$ the principal stratum, i.e. $\mathcal{L} = \HH(1, \ldots,
1)$. In \cite{linchpin} we adapt the argument of
Lindenstrauss and Mirzakhani to prove the more general 
statement given above. The special case we require here, namely $\mathcal{L} = 
\EE_D(1,1)$, was explained by
Calta and Wortman, see \cite[\S6]{CW}. 

\medskip 

Now suppose $\dim N_\mu^\circ \geq 2$ but $N_\mu^\circ$ does not contain $Z$. Since $Z$ is
a normal subgroup of $N$, $N_\mu^\circ$ does not contain a conjugate of $Z$. Let $A$ be the
diagonal group in $N$. Up to
conjugacy, the only three one-parameter subgroups of $N$ are $U, Z$
and $A$, so we find that $N_\mu^\circ$ contains a conjugate of $A$. We can write
this conjugate as 
$A' = n A n^{-1}$ where $n \in N$, and since $N = B \ltimes Z$, we can write $n = \rel_Tua$ where
$a \in A, \, u \in U$ 
and $T \in \R$. Since $N_\mu^\circ$ contains $U$ and $A =
aAa^{-1}$ we find that $N_\mu^\circ$ contains $(\rel_T)A(\rel_{-T})$. Therefore the measure
$\mu'= (\rel_{-T})_*\mu$ is $B$-invariant. By a result of
Eskin, Mirzakhani and Mohammadi \cite{EMM2}, $\mu'$ is $G$-invariant,
and thus by McMullen's classification of $G$ invariant measures \cite{McMullen-SL(2), McMullen -
decagon}, the
following are the only  possibilities:
\begin{itemize}
\item
$\mu'$ is the $G$-invariant measure on a closed $G$-orbit, and $D$ is
either a square or $D=5$. This implies that $\mu = \rel_{T*} \mu'$ is as
described in case (5). 
\item
$\mu'$ is the flat measure on $\EE_D(1,1)$. Since the flat measure is
$Z$-invariant, $\mu=\mu'$ and we are in case (7). 
\end{itemize} 
\end{proof}

\section{Injectivity and nondivergence}\label{sec: injectivity}
In this section we will prove some results which will be used in the
proof of Theorem \ref{thm: generic in orbit closure}. The strategy
of proof involves showing that typical horocycle orbits do not spend
too much time close to a closed G-orbit $\LL$ or translations of closed G-orbits
by the $\rel$ flow. The argument depends on the fact that a point in
an eigenform locus which is close to the $\rel_t$ orbit of $\LL$ but not in the $\rel_t$ orbit of $\LL$
drifts slowly in the direction of the $\rel_t$ vector field. In order to exploit
this phenomenon it is useful to have coordinates close to the $\rel_t$ orbit of $\LL$.
This argument is similar in spirit to the linearization method of
Dani-Margulis, see \cite[\S 3.4]{KSS handbook}.

The following is an analog of \cite[Cor. 3.4.8]{KSS handbook}. 
For a stratum $\HH$ and
$r>0$, let $\HH_r$
be the set of surfaces $M$ in $\HH$ for which there
are no horizontal saddle connections of length less than $r$. This is closely related
to, but not identical with the set $\HH'_r$ defined in Definition \ref{dfn: domain of rel}. 
\begin{dfn} $$\HH_\infty = \{M \in \HH: \Xi(M) = \varnothing\} = \bigcap_{r>0}
\HH_r. $$
\end{dfn}
\index{Hinfinity@$\HH_\infty$}

\begin{lem}\label{lem: decagon injective}
Let $\LL$ be a closed $G$-orbit in $\EE_D(1,1)$. Then for any $T>0$
there is $r>0$ such that the map 
\begin{equation}\label{eq: defn map 1}{
(M, x, y) \mapsto M \pluscirc (x,y)
}\end{equation}
is well-defined and injective on $\LL_r \times [-T,T] \times \{0\},$
where $\LL_r = \LL \cap \HH_r$.\index{Lr@$\LL_r$}
In particular, if we set 
\begin{equation}\label{eq: definition L0}{
\LL_\infty = \LL \cap \HH_\infty
\index{L@$\LL_\infty$}
}\end{equation}
then the map \eqref{eq: defn map 1} is well-defined and injective on $\LL_\infty \times \R \times
\{0\}$. 
\end{lem}

\begin{proof}
We begin by proving that the map \eqref{eq: defn map 1} is well-defined and
injective on $\LL_\infty \times \R \times
\{0\}$.  The assertion that the map is well-defined on $\LL_\infty$ follows from Corollary
\ref{thm: criterion}.  We prove injectivity. Suppose that $M_1, M_2 \in \LL_\infty$ and $x_1,
x_2 \in \R$ are such that $\rel_{x_1}M_1 = M_1 \pluscirc (x_1, 0) = M_2 \pluscirc (x_2,
0) = \rel_{x_2}M_2$. If $x_1 = x_2$ then by applying $\rel_{-x_1}$ we find that
$M_1=M_2$, so we can assume $x_1 \neq x_2$. Let $x_3 = x_1-x_2$.  
Then by Proposition \ref{prop: commuting}
\begin{equation}\label{eq: for a contradiction}{
\rel_{x_3} (\rel_{x_1}  M_1) = \rel_{x_3} (\rel_{x_2} M_2)
= \rel_{x_1}M_2,
}\end{equation}
and $\rel_{x_3}$ is not the identity element of $Z$. Our next objective is to 
exploit this relationship.

Let $\nu$ denote the $G$-invariant measure on $\LL$ coming from the
Haar measure on $G$, and let $\nu_1 = \rel_{x_1*} \nu$. 
By \cite{DS}, the set of $U$-generic points for
$\nu$ in $\LL$ is the set of surfaces which are not on periodic $U$-orbits, and
by \cite{Veech - alternative}, this is the set of 
surfaces without horizontal saddle connections. By Proposition
\ref{prop: Katok Spatzier}, the same conclusion
holds for $\nu_1,$ that is the set of $U$-generic points for $\nu_1$
in $\rel_{x_1} \LL$ 
are the surfaces without horizontal saddle connections. 
By \eqref{eq: for a contradiction}, 
$\rel_{x_3}$ maps $\rel_{x_1}M_1$, which is a generic point for 
$\nu_1$, to $\rel_{x_1}M_2$, which is another generic point. In light
of Corollary \ref{cor: Katok Spatzier},
$\rel_{x_3}$ preserves $\nu_1$. The stabilizer of $\nu$ in $N$
contains the diagonal subgroup $A$ and hence the stabilizer of $\nu_1$
in $N$ contains 
a non-unipotent element (take for example a nontrivial element of the
conjugate  $\rel_{x_1} A
\rel_{-x_1}$ in $B \ltimes Z$). Proposition \ref{prop: stabilizer
  connected} implies that the stabilizer of $\nu_1$ properly contains the group $ZU$. According
to Theorem \ref{thm: Calta Wortman revisited2}, the only $U$-invariant
measure on $\EE_D(1,1)$ whose stabilizer properly contains $ZU$ is
the flat measure, so $\nu_1$ is
the flat measure. This implies that  $\nu = \rel_{-x_1*}\nu_1$, is also the flat
measure contradicting the assumption that $\nu$ is supported on a closed
$G$ orbit. This proves the second assertion of the Lemma.

We now prove the first assertion of the Lemma. 
According to Corollary \ref{thm: criterion}, the map \eqref{eq: defn map 1} is
well-defined on $\LL_r \times [-T,T] \times \{0\}$ if $r>T$. 
It remains to show that this map is injective when $r$ is sufficiently
large. Set $M'= \rel_{x_1} M_1$. Suppose that $M'= \rel_{x_2} M_2$ for
$M_1, M_2 \in \LL$ and $x_1, x_2 \in \R$. According to the first part
of the proof, at least one of the surfaces
$M_i$ has horizontal saddle connections. Since the $Z$-action
preserves the property of having horizontal saddle connections (by
Corollary \ref{prop: horizontal preserved}), both surfaces 
have horizontal saddle connections, and since $\LL$ is a closed
$G$-orbit, they lie on closed $U$-orbits and have a horizontal
cylinder decomposition. 
Since circumferences and heights of cylinders are
the same for $M_i$ and $M'$, they are the same for $M_1$ and $M_2$ and
this ensures that $M_1$ and $M_2$ lie on the union of one of finitely many closed
horocycles $Ux_1, \ldots, Ux_\ell$ of equal length. Suppose first that
this length is 1, and denote the union of the horocycles of length 1
by $A_1$. Let $r_1>0$ be the length of the shortest horizonal saddle connection on any
surface in $A_1$, and choose $T_1\in (0,r_1)$ for which the map
\eqref{eq: defn map 1} is injective on $A_1 
\times [-T_1,T_1] \times \{0\}$. Such a $T_1$ exists because the
$\rel$ vector field is transverse to $\LL$, and by a compactness
argument. 
Now suppose the length of the closed
horocycles is $L$. Then for $t_0 \df \log L$, the union of the closed horocycles
is $A_L \df g_{t_0}A_1$, the length of the shortest
saddle connections on these horocycles is $r_L \df e^{t_0/2}r_1$ and the map
\eqref{eq: defn map 1} is injective on $[-T_L, T_L]$ where $T_L \df e^{t_0/2}T_1$. That is,
the ratio $r_L/T_L$ is independent of $L$. In particular, for any
$T>0$, we can choose $L$ so that $T=T_L$ and choose $r>r_L$. Since 
$$\LL_r \subset
\LL \sm \bigcup_{r_L \leq r}  A_L,$$
for these choices, the map \eqref{eq: defn map 1} will be injective on
$\LL_r \times [-T,T] \times \{0\}$. 
\end{proof}

\begin{remark}
The condition that $M \in \LL_r$ is necessary for the validity of
Lemma \ref{lem: decagon injective}. 
When $D$ is a square and $\LL$ is the closed $G$-orbit of
a square-tiled surface, there are surfaces $M \in \LL$, with a horizontal
cylinder decomposition of type (A), whose real Rel orbit is
compact. When $D=5$ the $ZU$-orbit of a surface in $\LL$ with a 
cylinder decomposition of type $(A)$ is also not injectively embedded. 
That is, in both cases, injectivity will fail if  $M$ has short
horizontal saddle connections.

\end{remark}

We will also need the following quantitative nondivergence results for the
horocycle flow. 
The following results are proved in \cite{mw0}:

\begin{thm}\label{thm: from MW}
For any stratum $\HH$ there are constants $C$ and $\alpha$ such
that for any $M \in \HH$, any $\rho \in (0,1]$ and any $T>0$, if
for any saddle connection $\delta$ for $M$,  $\max_{s \in [0,T]}
\|\hol(u_s M, \delta)\| \geq \rho$ then 
$$
|\{s \in [0,T]: u_sM \text{ has saddle connections of length} <
\vre\}| < C T \left( \frac{\vre}{\rho} \right)^{\alpha}.
$$

In particular:

\begin{itemize}
\item[(I)]
For any $\vre>0$ and any compact $K' \subset \HH$ there is a compact
$K \subset \HH$ such that for any $T>0$ and any $M \in K'$, 
\begin{equation}\label{eq: nondivergence estimate}
{\frac{1}{T}\left| \left\{ s \in [0,T]: u_sM \notin K\right\} \right| < \vre.
}\end{equation}
\item[(II)]
For any $\vre>0$ there is a compact $K \subset \HH$ such that for any
$M \in \HH_\infty$ there is $T_0>0$ such that for all $T>T_0$, \eqref{eq:
  nondivergence estimate} holds. 

\end{itemize}
\end{thm}

We will need a refinement of statement (I).

\begin{prop}\label{cor: new nondivergence}
For any positive constants $\vre$ and $r$, and any compact set $L \subset \HH_\infty$,
there is an open neighborhood $\mathcal{W}$ of
$L$ and a compact $\Omega \subset \HH$ containing $\mathcal{W}$, such
that surfaces in $\Omega$ have no horizontal saddle connections
shorter than $r$, 
and such that for any surface $M$ and any interval $I
\subset \R$ for which $u_s M \in \mathcal{W}$ for some $s \in I$, we have 
\begin{equation}\label{eq: what we need}{
\frac{\left| \left\{ s \in I: u_sM \notin 
\Omega 
\right\} \right|}{|I|}  < \vre.
}\end{equation}
\end{prop}
\begin{proof}
If $s_0 \in I$ is such that $u_{s_0}M \in \mathcal{W}$ then by making the
change of variables $s \mapsto s-s_0$ we can assume $M \in
\mathcal{W}$ and $I=[a,b]$ with $a \leq 0 \leq b$. By considering separately the
subintervals $[a,0]$ and $[0, b]$, we can assume that $I=[0,T]$ for some $T>0$. 
Given a surface $M$ we let $\hol(M)$\index{hol@$\hol(M)$} be the subset of
$\R^2$ consisting of the holonomies of the saddle connections of $M$, and set 
\begin{equation}\label{eq: notation 2}{
K_{r_1, r_2} = \left\{M \in \HH: \hol(M) \cap \left((-r_1,r_1)
    \times (-r_2, r_2) \right)
    = \varnothing \right\}.
}\end{equation}
Then $K_{r_1, r_2}$ is compact for any $r_1, r_2$.

For any surface $M$ and saddle connection $\delta$ we define $x(M,
\delta), y(M,\delta)$ via the formula \index{x$(M,
  \delta)$@$x(M, \delta)$} \index{y$(M, \delta)$@$y(M, \delta)$}
$$
\hol(M, \delta) = \left(x(M,
\delta), y(M,\delta) \right). 
$$

Given $\vre, r$ and $L$ as in the statement of the proposition,
let $\eta$ satisfy $C \eta^{\alpha} < \vre$, where $C$ and $\alpha$ are as in
Theorem \ref{thm: from MW}, and let $\displaystyle{t_0 = 2
  \log
\left(\frac{\eta}{\sqrt{2}r} \right)
 }$ and $\displaystyle{\sigma' = \frac{\eta^2}{2r}}$. These choices
guarantee
\begin{equation}\label{eq: guarantee}{
e^{t_0/2}r =
\frac{\eta}{\sqrt{2}} \ \ \text{  and  } \ \ e^{-t_0/2}\sigma' =
\frac{\eta}{\sqrt{2}}.
}\end{equation}
 Let $L'' = g_{t_0}(L)$. Then $L''$ is
a compact subset of $\HH_\infty$. Compactness implies that there is a constant $\theta>0$ with
the property that for
any $M \in L''$ with a saddle connection $\delta$ of length
less than 1, the vertical component $y = y(M, \delta)$ satisfies
$|y| \geq \theta$. Moreover by making $\theta$ smaller if necessary,
we can ensure that the same property holds for all surfaces in a
neighborhood $\mathcal{W}''$ of $L''$. Let $T_1 =
\frac{2}{\theta}.$ Then for any $M \in \mathcal{W}''$ and any saddle connection
$\delta$ for $M$, either $\|\hol(M, \delta)\| \geq 1$ or
\[
\begin{split}
\left\|\hol(u_{T_1}M, \delta) \right\| & \geq |x(u_{T_1}M, \delta) |
\\
& = |x(M,
\delta)+T_1 y(M, \delta)| \\ & \geq \frac{2}{\theta} |y(M, \delta)| - 1
\geq 1.
\end{split}
\]
That is, $\max_{s \in [0,T_1]} \|\hol(u_sM, \delta)\| \geq
1$, so we can apply Theorem \ref{thm: from MW} with $\rho =1$. We 
obtain that for all $T \geq T_1$, and all $M \in \mathcal{W}''$, 
\begin{equation}\label{eq: what we have}{
\frac{1}{T} \left| \left\{ s \in [0,T]: u_sM \text{ has saddle
      connections of length} < \eta
  \right\} \right| < \vre.
}\end{equation}

We now claim that if we set $\Omega = K_{r, \sigma}$ for any $\sigma
\leq \sigma'$  then \eqref{eq: what we need} holds when $M \in \mathcal{W}'
\df g_{-t_0} \left(\mathcal{W}''\right)$ and 
$I=[0,T], \, T>
e^{-t_0}T_1$. Indeed suppose $M \in \mathcal{W}'$, $\delta$ is a saddle
connection for $M$, and $s \in [0, T]$ such that $\hol(u_sM, \delta)
\in (-r, r) \times (-\sigma, \sigma)$. 
Let $M'' \df  g_{t_0}M \in \mathcal{W}''$, $s' \df e^{t_0}s$ so that
$g_{t_0}u_sM = u_{s'}M''$. By \eqref{eq: guarantee}, 
$$\hol(u_{s'}M'',
\delta) = g_{t_0} \hol(u_sM, \delta) \in 
\left(-\frac{\eta}{\sqrt{2}},
    \frac{\eta}{\sqrt{2}} \right)\times 
    \left(-\frac{\eta}{\sqrt{2}},
    \frac{\eta}{\sqrt{2}} \right)$$ 
   and in particular
$\|\hol(u_{s'}M'', \delta )\|
< \eta$. Since the change of variables $s \mapsto s'$ is linear, \eqref{eq: what
  we need} follows from \eqref{eq: what we have}. 

This completes the proof for the interval $I=[0, T]$ when $T>
e^{-t_0}T_1$. 
Finally, since $L \subset \HH_\infty$ is compact, by making
$\sigma$ small enough, we see that $\Omega$ contains a neighborhood of
$L$, so that taking a small enough neighborhood
$\mathcal{W}$ of $L$
contained in $\mathcal{W}'$, we can ensure that $\mathcal{W} \subset
\Omega$ and also that for any $s \in [0,
e^{-t_0}T_1]$ and any $M \in \mathcal{W}$, $\hol(u_sM) \cap (-r, r)
\times (-\sigma, \sigma) = \varnothing.$ 
\end{proof}

In light of Proposition \ref{prop: continuous}, the following is an
immediate consequence of Lemma \ref{lem: decagon injective} and
Proposition \ref{cor: new nondivergence}: 

\begin{cor}\label{cor: needed for linearization}
Given positive  numbers $T$ and $\vre$, and a compact subset $L \subset \LL_\infty$, there
is $\delta>0$, a neighborhood $\mathcal{W}$ of $L$ and a compact set $\Omega
\subset \LL$, containing $\mathcal{W}$, such that the map \eqref{eq: defn map 1}
is well-defined, continuous and injective on $\Omega \times [-T, T]
\times [-\delta, \delta]$, and for any interval $I$ and
any $M \in \HH$, if there is $s_0 \in I$ such that $u_{s_0}M \in
\mathcal{W}$ then equation \eqref{eq: what we need} holds. 
\end{cor} \qed

We take this opportunity to record a consequence of Proposition
\ref{cor: new nondivergence}, which will not be used in this paper but
is of independent interest. 

\begin{cor}\label{thm: new nondivergence2}
For any $\vre>0$ and any compact $L \subset \HH_\infty$, there is a compact 
$\Omega \subset \HH_\infty$ containing $L$ such that 
for any $M \in L$ and any interval $I \subset \R$ containing 0, 
\eqref{eq: what we need}
holds. 
\end{cor}

\begin{proof}
For any $\vre>0$ and $j=1,2, \ldots,$ let
$\displaystyle{\vre_j = \frac{\vre}{2^j}}$. By Proposition \ref{cor:
  new nondivergence}, we can find $\Omega_j$ containing $L$ 
such that all surfaces in $\Omega_j$ have no horizontal saddle
connections of length less than $j$, and \eqref{eq: what we need} holds with $\vre_j$ in place of
$\vre$ for any interval $I$ containing 0 and
any $M \in L$. Then $\Omega = \bigcap_j \Omega_j$ has the required
properties. 
\end{proof}

\section{All horocycle orbits are generic}\label{sec: generic}  
The goal of this section is to prove the following more detailed
version of Theorem \ref{thm: orbit closures}. 

\begin{thm}
\label{thm: generic in orbit closure}
Let $M \in \EE_D(1,1)$. Then $M$ is generic for a $U$-invariant
ergodic measure $\mu_M$, and the type of $\mu_M$ (as described in 
Definition \ref{dfn: types of measures}) is as follows: 
\begin{itemize}
\item
$\mu_M$ has type (1) when $\Xi(M)$ is a complete separatrix diagram of types (B), (C) or (D), or
a complete separatrix diagram of type (A) with commensurable moduli. 
\item
$\mu_M$ has type (2) when $\Xi(M)$ is a complete separatrix diagram of type (A) with two cylinders of
incommensurable moduli. 
\item
$\mu_M$ has type (3) when 
$\Xi(M)$ consists of two saddle 
connections joining distinct singularities, which disconnect $M$ into
two isogenous tori glued along a slit, as in case (iii) of \S
\ref{subsection: nonempty}. 
\item
$\mu_M$ has type (4) when $\Xi(M)$ consists of one saddle connection joining distinct
singularities, as in case (ii) of \S \ref{subsection: nonempty}. 

\item
$\mu_M$ has type (5) when $\Xi(M) = \varnothing$ and there is $s \in \R$ so that $\rel_sM$
is a lattice surface. 

\item
$\mu_M$ has type (6) when $\Xi(M)$ is a pair of saddle connections which do not
disconnect $M$, as in case (iv) of \S \ref{subsection: nonempty}. 
\item
$\mu_M$ has type (7) if it does not correspond to one of the previous
cases i.e. $\Xi(M) = \varnothing$ and $M$ is not the result of
applying the $\rel$ flow to a lattice surface. 

\end{itemize}

\end{thm}
\begin{remark}\label{remark: interesting}
It would be interesting to characterize case (5) explicitly. That is, 
give a geometric characterization of those surfaces 
$M \in \EE_D(1,1)$ with $\Xi(M) = \varnothing$, for which there 
is $s \in \R$ such that $\rel_sM$ is a lattice surface. 
\end{remark}

The proof of Theorem \ref{thm: generic in orbit closure} relies on an
analog of the `linearization method', see \cite[\S3.4]{KSS handbook}. 
We need the following
notion.  

\begin{dfn}
The support of a $U$-invariant $U$-ergodic measure on $\EE_D(1,1)$ is
called a {\em sheet.}\index{sheet} The {\em type of a sheet}\index{type of a sheet} of the sheet is the type of the corresponding measure. 
\end{dfn}

Sheets of the same type typically appear in families. It will
be convenient to partition the sheets into families, called beds\index{beds}. 
In the sequel, a {\em bed} will be a measurable set which is a union of
sheets in $\EE_D(1,1)$. We further require the sheets to be of a fixed
type $(j)$, for $j \in \{3,4,5,6,7\}$, or to be of one of the two types (1) or (2). A bed
corresponding to sheets of type (1) or (2) will be called a {\em bed of minimal
  sets}  and a bed corresponding to sheets of type $j \in \{3,
\ldots, 7\}$ will be called a {\em bed of type $(j)$.} 
If $\mathcal{B}$ is a bed corresponding to measures $\mu$ of a certain
type, we say that $\mu$ {\em belongs to $\mathcal{B}$}, and we define
$\Xi(\mathcal{B})$ to be $\Xi(\mu)$ for $\mu$ belonging to $\mathcal{B}$. 

The reason for combining sheets of type (1) and (2) into one bed, is
that an arbitrarily small perturbation of a sheet of 
type (1) can be a sheet of type (2) and conversely (the condition that the
moduli of the cylinders are rationally related is not stable under
small perturbations). 

In our definition of beds we only required them to be measurable but
in fact they can be chosen so that they have a nicer
structure. Continuing with the discussion following Definition
\ref{dfn: types of measures},
one can partition the sheets into beds which
are almost everywhere locally
modeled on affine subspaces of the Lie algebra $\mathfrak{l}$. However  the
corresponding affine subspaces of $\mathfrak{l}$ need not correspond
to Lie subalgebras. Additionally 
beds may have complicated topology (e.g. boundary). This will be
discussed further in \cite{examples}.

For our analysis, the following property of a bed will be helpful:

\begin{dfn}
We will say that a sequence of sets $K_1, K_2, \ldots$ 
{\em exhausts $\BB$}  if the $K_i$ are compact, and $\mu
\left(\bigcup_i K_i \right)=1$ for any measure $\mu$ which belongs to $\BB$. 
\end{dfn}

\begin{prop}\label{prop: beds are submanifolds} 
For each type $j \leq 6$, all sheets in $\EE_D(1,1)$ of type $(j)$ are
contained in a finite or countable union of beds, 
and for each bed
$\BB$ there is a countable sequence of compact sets which exhausts
$\BB$. Explicitly these properties are satisfied by the following
choices: 
\begin{itemize}
\item
For types $j \in \{3,4,6,7\}$,  we let $\BB$ be the union of all sheets of type $(j)$, and we
let $K_i$ be the closure of the set of $M \in
\EE_D(1,1)$ which have no 
saddle connections shorter than $1/i$, and such that 
$\{\delta \in \Xi(M): \| \hol(M, \delta)\| \leq  i \}$ and $
\Xi(\BB)$ are the same (as horizontal data diagrams). 
\item
For types $j \in \{1,2\}$ we let $\BB$ be the union of all sheets of
type 1 or 2, and we define $K_i$ as before. 
\item
If $j=5$, for each closed $G$-orbit $\LL$ we take $\BB = \bigcup_{t
  \in \R} \rel_{t}(\LL)$. Let $\nu_0$ be the Haar measure on 
$\LL$, let $\LL_\infty$ be as in \eqref{eq: definition L0}, let $L_1, L_2,
\ldots$ be a nested sequence of compact
subsets of $\LL_\infty$ with $\nu_0\left(\bigcup_i L_i \right) = 1$, and let 
$$K_i = \bigcup_{|t| \leq i} \rel_t(L_i).$$

\end{itemize} 
\end{prop}

\begin{proof} 
It is clear that the beds listed above contain all sheets
of type $j$. In all cases except $j=5$ there is just one bed. 
In the case
$j=5$ the number of beds is at most countable, since each $\EE_D(1,1)$
contains at most finitely many closed $G$-orbits. 

We now show that in all cases, the sets $K_i$ listed above exhaust
the bed. 
Suppose first that $\Xi(\BB) \neq \varnothing.$
It is clear that each $K_i$ as in the statement of the proposition is
compact. 
The set $\bigcup_i K_i$ contains all $M \in \BB$ for which
$\Xi(M) = \Xi(\BB)$, so Corollary \ref{cor: horizontal relations}
implies that $\mu \left(\bigcup_i K_i \right) =1$ for any measure
$\mu$ which belongs to $\BB$. 

Now suppose $\BB$ is a bed of type (5). Let $\LL, \, \nu_0, \, L_i, \,
K_i$ be as above. Each $K_i$ is compact by Proposition \ref{prop:
  continuous}.  
Also, for any measure $\mu$ belonging to $\BB$, there is $t \in \R$ such that
$\mu = \nu_t \df \rel_{t*} \nu_0$. Since $\nu_0 \left ( \bigcup_i L_i
\right)=1$, we have $\nu_t \left( \bigcup_i \rel_t(L_i) \right)=1$ for
all $t$. For each $t$, since $\bigcup_i K_i$ contains $\bigcup_i \rel_t(L_i)$
for each $t$, we have $\nu_t \left( \bigcup_i K_i
\right)=1$ for each $t$. 
\end{proof}

The following result summarizes a strategy for proving equidistribution results
which we will use repeatedly.

\begin{prop}\label{prop: abstract}
Let $\{\mu_t\}$ be a collection of measures where $t$ ranges over either
the positive integers or non-negative real numbers. Suppose the following
hold:
\begin{itemize}
\item[(a)]
The sequence $(\mu_t)$  has no escape of mass as $t\to\infty$; i.e. for any $\vre>0$
there is a compact $K \subset \EE_D(1,1)$ and $t_0$ such that for all
$t \geq t_0$, $\mu_t(K) \geq 1-\vre.$
\item[(b)]
Any convergent subsequence $(\mu_{t_k})$ of $(\mu_t)$ with $\ t_k
\to \infty$ converges to a measure which is
invariant under a conjugate of $U$ by an element of $G$. 
\item[(c)]
For any bed $\mathcal{B} \varsubsetneq \EE_D(1,1)$ there is a sequence
$K_1, K_2, \ldots$ 
of sets which exhaust $\BB$, and for any $i$ and any
$\vre>0$, 
there is $t_0$ and an open set $\mathcal{U}$ containing $K_i$ such that for
all $t \geq t_0$, $\mu_t(\mathcal{U}) < \vre.$ 
\end{itemize}

Then the sequence $\mu_t$ converges to the flat measure on
$\EE_D(1,1)$ as $t \to \infty$.

\end{prop}

\begin{proof}
It suffices to show that any subsequence  $(\mu_{t_k})$ of $(\mu_t)$, with $t_k
\to \infty$ contains a further subsequence 
converging to the flat measure. So re-indexing, suppose
$(\mu_t)$ is already a subsequence. Since there is no escape of
mass, the set $(\mu_t)$ is precompact with respect to the weak-$*$
topology. So passing to a subsequence we can assume that $\mu_t$
converges weak-$*$ to a probability measure $\nu$ and we need to show that $\nu$ is
the flat measure. By assumption (b), $\nu$ is invariant under a
conjugate $U' = g^{-1}Ug$ of $U$, and hence $g_* \nu$ is $U$-invariant. We need
to show that $g_*\nu$ is the flat measure, as this will imply that
$\nu$ is the flat measure as well. To simplify notation we therefore
replace $g_* \nu$ with $\nu$ to assume that $\nu$ is $U$-invariant. Using the
ergodic decomposition of $\nu$ and the fact that there are countably many beds, we can write 
$\nu = \sum_1^7\nu_j$ where each $\nu_j$ is 
supported on the beds of type $j$,  
and 
(after normalizing $\nu_j$ to be a probability measure) is a
convex combination of the measures belonging to these beds. We must show that $\nu_1 = \cdots = \nu_6
=0$. We derive this from assumption (c), as follows. 

Fix a type $j \leq 6$ and let $\mathcal{B} \varsubsetneq \EE_D(1,1)$
be a bed of type $j$. Let $K_1, K_2, \ldots$ be compact sets as described in assumption
(c). Since $\mu(\bigcup K_i)=1$ for every
$U$-invariant ergodic measure $\mu$ belonging to $\mathcal{B}$,
in order to show $\nu_j=0$
it suffices to show that $\nu_j(\bigcup K_i)=0$.  Suppose by
contradiction that $a = \nu_j(\EE_D(1,1))$ is strictly positive. For
any $i$, and any $\vre>0$, let $t_0$ and $\mathcal{U}$ be as in (c). 
There is a continuous compactly supported 
function $\varphi: \EE_D(1,1) \to [0,1]$ which is identically 1 on $K_i$ and vanishes
outside $\mathcal{U}$. The definition of the weak-$*$ topology and
condition (c) now ensure that $\nu_j(K_i) \leq \int_{\EE_D(1,1)} \varphi
\, d\nu_j \leq \frac{1}{a} \lim_k \int_{\EE_D(1,1)} \varphi \, d\mu_{t_k} \leq \frac{\vre}{a}$. Since
$\vre>0$ was arbitrary we must have $\nu_j(K_i)=0$ for each $i$, and hence
$\nu_j \left( \bigcup K_i \right)=0$. 
\end{proof}

\begin{proof}[Proof of Theorem \ref{thm: generic in orbit closure}]
In this proof we will say that $M$ is of type $(j)$ (where $j \in \{1,
\ldots, 7\}$) if $\Xi(M)$ is topologically equivalent to $\Xi(\mu)$
and $\mu$ is of type $j$. \index{type of a surface} 

{\em Step 1: $M$ is of type (1) through (6).} 
If $\Xi(M)$ decomposes $M$ into horizontal cylinders then the $U$-action on the orbit-closure of $M$ 
is conjugate to an irrational straightline flow on a torus and so every
orbit is equidistributed. This is what happens when $\Xi(M)$ is a complete
separatrix diagram, i.e. in cases (1) or (2). When $\Xi(M) \neq \varnothing$ but $\Xi(M)$ is not a
complete separatrix diagram, the
$U$-action on $\overline{UM}$ is obtained from the $U$-action on a
simpler space via a $U$-equivariant map $\Psi_T$. Namely, by
Proposition 
\ref{prop: tear}, in each of the cases (3), (4) and (6) this simpler space is a
finite volume homogeneous space $G/\Gamma$, and $M = \Psi_T(M_0)$
where $M_0 \in G/\Gamma$ does not lie on a
periodic $U$-orbit, since $\Xi(M)$ is not a complete separatrix
diagram. The equidistribution of $UM_0$ in $G/\Gamma$ 
follows from \cite{DS}, and the equidistribution of $M$ follows from
the fact that $\Psi_T$ is $U$-equivariant and continuous. The same argument applies when $\Xi(M) =
\varnothing$ and $M$ belongs to $\rel_s\LL,$ for a closed $G$-orbit $\LL$
and $s \in \R$. 

{\em Step 2: $M$ is of type (7) and $\Xi(\mathcal{B}) \neq \varnothing$:}
When $M$ is of type (7) we verify the hypotheses of Proposition \ref{prop: abstract} for the
collection of measures 
$\{\mu_t: t>0\}$ defined by averaging along the orbit of $M$; that is,
$\mu_t = \nu(M,t)$, where $\nu(M,t)$ is as
in \eqref{dfn: orbit measure}. 
Hypothesis (a) of Proposition \ref{prop: abstract} follows from Theorem
\ref{thm: from MW} (I) and hypothesis (b) is immediate from
the definition of $\mu_{t}$. To verify (c) we first discuss beds
$\mathcal{B}$ for
which $\Xi(\mathcal{B}) \neq \varnothing, $ adapting an argument of
\cite{EMM}. Let $K_1, K_2, \ldots$ be the 
exhaustion of the bed $\BB$, as in Proposition \ref{prop: beds are
  submanifolds}.
Fix $i$ and let $\vre>0$. Then any surface in $K_i$ contains a
horizontal saddle connection of length at most $i$. 
Let $C(\delta)$ be the open set of surfaces whose shortest saddle
connection is shorter than $\delta$. 
By Theorem \ref{thm: from MW} (II) there is
$\delta >0$ (depending only on the stratum $\HH(1,1)$) such that 
every surface $M'$ without horizontal saddle connections satisfies 
\begin{equation}\label{eq: quant nondiv}{
\limsup_{t \to \infty} \mu_{M',t}(C(\delta)) < \vre
}\end{equation}
(where $\mu_{M',t}=\nu(M', t)$ is as in Definition \ref{dfn: orbit measure}). 
Let 
\begin{equation}\label{eq: condition on t0}{
t_0< 2(\log \delta -\log i),
}\end{equation}
so that if $M_1$ has a horizontal saddle connection of length less
than $i$, then $g_{t_0}M_1$ has a saddle connection of length
less than $\delta$; in other words, $K_i \subset \mathcal{U} \df g_{-t_0}(C(\delta))$.
Since $g_{t_0} u_s g_{-t_0} = u_{e^{t_0}s}, $ 
setting $M' = g_{t_0}M$, we have $g_{t_0*} \mu_{M,t}=
\mu_{M', e^{t_0}t}$ and \eqref{eq: quant nondiv} implies that
$$
\limsup_{t \to \infty} \mu_{M,t}(\mathcal{U}) = \limsup_{t\to \infty}
\mu_{M',t}(C(\delta)) < \vre, 
$$
so (c) is satisfied. 

{\em Step 3: $M$ is of type (7) and $\Xi(\mathcal{B}) = \varnothing.$ }
That is $\mathcal{B}$
is a bed of type (5). The measures belonging to $\mathcal{B}$ are of
the form $\rel_{s*} \mu$, where $s \in \R$ and $\mu$
is the $G$-invariant measure on a closed $G$-orbit $\LL$. Let $\LL_\infty$
be as in \eqref{eq: definition L0}. Set 
\begin{equation}\label{eq: our map}{ F(M,x,y) =M
  \pluscirc (x,y),
}\end{equation}
postponing for the moment questions of domain of definition of
$F$. Note that $F (M,x,0) = \rel_xM,$ and when there is no risk of
confusion, we will write $F(M,x,0)$ simply as $F(M,x)$. 
Let $L_i$ and $K_i$ be as in Proposition \ref{prop: beds are
  submanifolds}, so that the $K_i$ exhaust $\BB$. Note that $L_i
\subset \LL_\infty$ for each $i$ and 
$$
K_i = F\big(L_i  \times [-i, i] \times \{0\}\big). 
$$
We will  verify (c) for the sets $K_i$. Given $i$ and $\vre>0$, choose 
\begin{equation}\label{eq: choice of j}{
j> \frac{(8+\vre)(i+1)}{\vre}.
}\end{equation}
Corollary \ref{cor: needed for linearization} gives us a constant
$\delta>0$, an open set $\mathcal{W} \subset \LL$ containing $L_i$, and
a compact set $\Omega$ containing $\mathcal{W}$ such
that: (1) $F$ is well-defined, continuous and injective on $\Omega \times [-j, j]
\times [-\delta, \delta]$, and (2)
whenever $M' \in \LL$ and $I
\subset \R$ is an interval such that $u_s M' \in \mathcal{W}$ for some
$s \in I$, we have
\begin{equation}\label{eq: using new nondivergence}{
\frac{1}{|I|} \left| \left\{ s \in I : u_sM' \notin \Omega \right\}
\right| < \frac{\vre}{4}. 
}\end{equation}

Now set 
\begin{equation}\label{eq: defn U}{
\mathcal{U}
= F\big(\mathcal{W} \times (-(i+1), i+1) \times (-\delta,
\delta)\big).
}\end{equation}
Then $\mathcal{U}$ is a neighborhood of $K_i$, and we need to show that 
$\mu_{M,t}(\mathcal{U})<\vre$ for all sufficiently large $t$. 
That is, we need to find $t_0$ so that for $t>t_0$, 
\begin{equation}\label{eq: need to show to finish}{
\frac{|\widehat{\mathcal{I} }\cap [0,t] |}{t} <\vre, \ \text{ where we set}
\  
\widehat{\mathcal{I}} = \{ s \in \R : u_{s} M \in \mathcal{U}\}.
}\end{equation}

Whenever $s_0 \in  \widehat{ \mathcal{I}}$ there are $M'  = M'(s_0)\in \mathcal{W}$ and
$(x, y) = (x(s_0), y(s_0)) \in (-(i+1), i+1) \times (-\delta, \delta)$
such that $u_{s_0}M = F(M', x,y)$. Since $M$ is not of type (5) we
have $y(s_0) \neq 0$. We define the following intervals: 
\[ 
\begin{split}
\mathcal{I} = \mathcal{I}(s_0) & \df 
\{s : |x+ sy| \leq i+1 \} \\ 
\mathcal{J} = \mathcal{J}(s_0) & \df \{s : |x+ sy| \leq j \}.
\end{split}
\]
These are nested bounded intervals with common midpoint $\displaystyle{s =
  -\frac{x(s_0)}{y(s_0)}.}$ 
If an interval $I$ contains an endpoint of $\mathcal{J}$ and intersects
$\mathcal{I}$ then it contains one of the two connected
components of $\mathcal{J} \sm \mathcal{I}$. These two
components have equal lengths since the two intervals
share the same midpoint and by 
\eqref{eq: choice of j}, 
\begin{equation}\label{eq: relative lengths}{
\frac{|\mathcal{I} \cap I |}{|\mathcal{J} 
\cap I|}
\leq  \frac{2(i+1)|y_1|}{(j-i-1)|y_1|} < \frac{\vre}{4}.
}\end{equation}

We claim that
if $s \in \mathcal{J} \sm \mathcal{I}$ then either $u_{s} M'
\notin \Omega$ or $u_{s+s_0}M \notin \mathcal{U}$. 
Indeed, suppose $s \in \mathcal{J} \sm \mathcal{I}$ and $u_{s}M'
\in \Omega.$ According to Proposition \ref{prop: distributive}, 
\[
\begin{split} u_{s+s_0}M & =
u_s F(M', x,y) = 
u_s (M' \pluscirc (x,y) ) \\ & = u_s M' \pluscirc (x+sy, y) = F(u_s M',
x+sy, y).
\end{split}
\]
The injectivity of $F$ on $\Omega \times [-j,j] \times [-\delta,
\delta]$ and $|x+sy|>i+1$ imply that $u_{s+s_0}M \notin \mathcal{U}$.
This proves the claim. 

For each $s_0 \in \mathcal{I}$, denote the translates by 
$$\mathcal{J}'(s_0) =
\mathcal{J}(s_0)+s_0, \ \ \ \mathcal{I'}(s_0) = \mathcal{I}(s_0)+s_0.$$
The claim, combined with \eqref{eq: using new nondivergence} 
and \eqref{eq: relative lengths}, imply that if $
[0,t]$ contains an endpoint of $\mathcal{J}'=
\mathcal{J}'(s_0)$, then
$$
\frac{| [0,t] \cap  \mathcal{J}' \cap \widehat{\mathcal{I}}  |}{|[0,t]
  \cap \mathcal{J}'| }
\leq \frac{|[0,t]\cap \mathcal{I}' |}{|[0,t] \cap \mathcal{J}' |} + \frac{|\{s
  \in [0,t] \cap \mathcal{J}' : u_{s-s_0}M' \notin \Omega |}{| [0,t]
  \cap \mathcal{J}'|}< \frac{\vre}{2}.
$$

Since $\widehat{\mathcal{I}}$ is covered by the intervals
$\{\mathcal{J}'(s_0) : s_0 \in \widehat{\mathcal{I}} \}$, a standard
covering argument shows that we can take
a countable subcover $\mathcal{J}'_1, \mathcal{J}'_2, \ldots$
satisfying
$$
s \in \widehat{\mathcal{I}} \ \implies \ 1 \leq 
\# \, \{\ell: s \in
\mathcal{J}'_\ell\} \leq 2. 
$$
In particular $0$ is contained in at most two of the intervals
$\mathcal{J}'(s_{\ell})$ and if we 
take $t_0$ to be larger than the right endpoint of these two
intervals, and $t>t_0$, $[0,t]$ will contain at least one of the endpoints of any $\mathcal{J}'(s_{\ell})$ which intersects
$[0,t]$. Then for any $t>t_0$ we will have:
$$
\left| [0,t] \cap \widehat{\mathcal{I}}\right | \leq \sum_{\ell} \left|[0,t] \cap
\mathcal{J}'_\ell \cap \widehat{\mathcal{I}}\right | \leq \frac{\vre}{2} \sum_{\ell} \left| [0,t] \cap
\mathcal{J}'_{\ell} \right |
\leq \frac{\vre}{2} \, 2 \left|[0,t] \right|
$$
and this proves \eqref{eq: need to show to finish}. 
\end{proof}

\section{Equidistribution results for sequences of measures}\label{sec: more}
In this section we show that several natural sequences of measures equidistribute with respect
to the flat measure on $\EE_D(1,1)$. 
We prove equidistribution
for sequences of periodic horocycles of increasing period in Theorem
\ref{thm: teaser1}, for translates of $G$-invariant measures 
by the $\rel$ flow in Theorem \ref{thm: teaser2}, for circle orbits of
increasing radius in Theorem \ref{thm: Matt absolved}, 
and for pushforwards of invariant measures of minimal sets by the geodesic flow
in Theorem \ref{thm: limits of UZ}.

\begin{proof}[Proof of Theorem \ref{thm: teaser2}]
Let $\LL$ be a closed $G$-orbit in $\EE_D(1,1)$, let $\mu$ be the Haar
measure on $\LL$, and let $\mu_t = \rel_{t*} \mu$. 
We will prove that
as $t \to +\infty$ or $t \to -\infty$, the measure $\mu_t $ converges to the flat measure on
$\EE_D(1,1)$. For definiteness we discuss the case $t \to +\infty$,
the second case being similar. It is enough to show that if $t_n \to
\infty$ is any sequence of real numbers for which $\mu_{t_n}$
converges to some $\nu$ (where $\nu$ is not necessarily a probability
measure), then $\nu$ is the flat measure (and in particular is a
probability measure).

First we show that $\nu$ is a probability measure, i.e. that there is
no escape of mass. This follows from statement (II) of Theorem \ref{thm: from MW} as follows. We
need to show that for any $\vre>0$ there is a compact subset $K_0$ of
$\EE_D(1,1)$ such that for all $t$, $\mu_t(K_0) \geq 1-\vre$. Given
$\vre$ let $K$ be the intersection of $\EE_D(1,1)$ with the compact
set in statement (II) of Theorem \ref{thm: from MW}. Let $\varphi$ be a continuous compactly
supported function on $\EE_D(1,1)$, with values in $[0,1]$, which is identically equal to 1 on
$K$, and let $K_0 = \supp \, \varphi$. Let $M$ be a generic point for
$\mu_t$. According to Proposition \ref{prop: Katok Spatzier},
$M = \rel_t(M')$ for $M' \in \LL$, such that $M'$
is generic for $\mu$. In particular $M'$ has no horizontal saddle
connections, and hence neither does $M$. According to (II), 
$$
\liminf_{T \to \infty} \frac{1}{T} \left| \left \{ s \in [0,T]: u_s M
    \in K\right\} \right|
\geq 1-\vre,
$$
and therefore
\[
\begin{split}
\mu_t(K_0) & \geq \int \varphi \, d\mu_t = \lim_{T \to
  \infty} \frac{1}{T} \int_0^T \varphi(u_sM)ds \\ 
& \geq \liminf_{T \to
  \infty}   \frac{1}{T} \int_0^T 1_{K} (u_sM)ds \geq 1-\vre. 
\end{split}
\]

Now we claim that $\nu(\HH_\infty)=1$, that is $\nu$ gives no mass to the
set of surfaces with horizontal saddle connections. This follows by
again expressing $\mu_t$ as the limit of an integral along a generic
horocycle orbit, and 
repeating the argument of Step 2 of the proof of Theorem \ref{thm:
  generic in orbit closure}. Thus, in view of Claim \ref{claim:
  linchpin} (see the proof of Theorem \ref{thm:
  Calta Wortman revisited2}), it suffices to prove that $\nu$ is invariant under the
`horospherical foliation' $UZ$. It is clear that $\nu$ is $U$-invariant,
and it remains to show that it is invariant under $\rel_{s}$ for any
$s \in \R$. 

Let $\varphi \in C_c \left( \EE_D(1,1) \right)$. Since $\varphi$ is uniformly continuous, for any
$\vre>0$ there is a neighborhood $\mathcal{U}$ of the identity in $G$
so that 
\begin{equation}\label{eq: difference small}{|\varphi(M) - \varphi(gM)|< \vre \ \text{ for
    any } M \in \EE_D(1,1), \ g \in \mathcal{U}.
}\end{equation}
Now define 
$$
\tau(t, s) = 2 \log \left( 1+\frac{s}{t}\right)\text{ and } \ \tau_n = \tau(t_n, s).
$$ 
By a matrix multiplication in $N$, and using Corollary \ref{cor: real
  rel via group} we see that these choices ensure
that for any surface $M$ with no horizontal saddle connections, 
\begin{equation}\label{eq: commutation2}{g_{\tau_n} \rel_{t_n} M = \rel_{t_n+s}
  g_{\tau_n} M.
}\end{equation}
Moreover $g_{\tau_n} \to \mathrm{Id}$ as $n \to \infty$. Then for $n$
large enough, by \eqref{eq: difference small}, 
$$
\left| \int \varphi \, d \rel_{t_n*} \mu  - \int \varphi \circ
  g_{\tau_n} \, d \rel_{t_n*} \mu  \right | < \vre.
$$
On the other hand, using \eqref{eq: commutation2} and the fact that
$\mu$ is supported on $\LL_\infty$, we obtain
\[
\begin{split}
& \int \varphi \circ
  g_{\tau_n} \, d \rel_{t_n*} \mu = \int_{\LL_\infty} \varphi(g_{\tau_n} \rel_{t_n}
  (M)) \, d\mu(M)  \\ 
 = & \int_{\LL_\infty} \varphi(\rel_{t_n+s} g_{\tau_n} M)
  \, d\mu(M) = \int_{\LL_\infty} \varphi(\rel_{t_n+s} M) d\mu(M).
\end{split}
\]
In the last line we used the fact that $\mu$ is invariant under
$g_{\tau}$ for all $\tau$. Putting these together we find that for
sufficiently large $n$, 
$$
\left| \int \varphi \, d \rel_{t_n*} \mu - \int \varphi \, 
  d\left(\rel_{t_n+s} \right)_* \mu \right| < \vre,
$$
and since $\vre$ was arbitrary, 
$$\nu = \lim_{n \to \infty} \rel_{t_n*}
\mu = \lim_{n \to \infty} (\rel_{s+t_n})_* \mu = \rel_{s*} \nu.$$
\end{proof}

Generalizing Theorem \ref{thm: teaser2} we have:
\begin{thm}\label{thm: and more}
In each of the cases (ii), (iii), (iv) of \S \ref{subsection:
  nonempty}, let $T\neq 0$ and let $\Psi_T$ be the map described in \S
\ref{subsection: nonempty}. Let $\mu_T$ be the pushforward of Haar
measure under $\Psi_T$, as in Proposition \ref{prop: tear}. 
Then as $|T| \to \infty$, $\mu_T$ tends to
the flat measure on $\EE_D(1,1)$.  
\end{thm}

The proof of
Theorem \ref{thm: teaser2} goes through almost verbatim. We leave the
details to the reader.

\begin{proof}[Proof of Theorem \ref{thm: teaser1}]
We apply Proposition \ref{prop: abstract} to 
\begin{equation}\label{eq: new mut}{\mu_t \df g_{t*} \mu, }
\end{equation}
where $\int f \, d \mu = \frac{1}{p} \int_0^p f(u_sM) ds$ and $\{u_s M
: s \in [0,p]\}$ is a closed horocycle of period $p$.

To verify that there is no escape of mass we use Theorem \ref{thm:
  from MW}. Let $\rho_0$ be the length of the shortest horizontal saddle
connection on $M$, let $t_0 = -2 \log \rho_0$ and let $t \geq
t_0$. The choice of $t_0$ ensures that the shortest horizontal saddle connection on
$g_tM$ has length at least 1. Since the orbit $Ug_tM$ is periodic,
the measure 
$\mu_t$ is identical to the measure obtained by averaging along any
integer multiple of the period $e^{t}p$ of this orbit. 
The set of saddle connections for $g_tM$ whose length is shorter than
1 is finite and none of these is horizontal. Thus for each such saddle
connection $\delta$, $\hol(u_sg_tM, \delta)$ diverges as $s \to
\infty$. Therefore if we take a
sufficiently large multiple of the period $s_0 = ke^tp, \, k \in \N$, 
any saddle connection on $M$ will have length greater than 1 either
for the surface $g_tM$ or for the surface $u_{s_0}g_t M$. As a
consequence, the hypothesis of Theorem \ref{thm: from MW} is satisfied
with $\rho=1$, and $I=[0, s_0]$ and there is no escape of mass for a
multiple of the period. 
By periodicity, and using the notation of
Definition \ref{dfn: orbit measure}, we have $\mu_t=\nu(g_tM, s_0)$,
and thus there is no escape of mass for 
the measures $\mu_t$ either. This verifies hypothesis (a)
of Proposition \ref{prop: abstract}, and hypothesis (b) is obvious. 

To verify hypothesis (c) we adapt the argument given in the proof of Theorem \ref{thm:
  generic in orbit closure}, retaining the same notation. Namely, in step 2 we verify (c) for beds
with $\Xi(\mathcal{B}) \neq \varnothing$. We fix $i$ and $\vre>0$, and
note that the argument in the preceding paragraph, using 
Theorem \ref{thm: from MW}, implies that there is $\delta>0$ such that  if $t$
is large enough so that $ g_{t}M$ has no horizontal saddle of length
shorter than 1, then 
\begin{equation}\label{eq: quant nondiv2}{
\mu_t(C(\delta)) < \vre.
}\end{equation}
Then we take $t_0$ small enough so that all
surfaces in $g_{t_0} (K_i)$ have horizontal saddle connections
shorter than $\delta$, and set $\mathcal{U} =
g_{-t_0}(C(\delta))$. Then \eqref{eq: quant nondiv2} and $\mu_{t+t_0} =
g_{t_0*}\mu_t$ imply that for all
sufficiently large $t$, 
$
\mu_t(\mathcal{U}) < \vre, 
$
as required. 

Continuing with Step 3, it remains to verify (c) for beds with
$\Xi(\mathcal{B}) = \varnothing.$ 
We define $F, L_i, K_i$ as in the proof of Theorem \ref{thm:
  generic in orbit closure}, recalling that surfaces in $L_i$ have no
horizontal saddle connections. Given $\vre$ and $i$, 
we define $\mathcal{U}$ via \eqref{eq: defn U}, and need to show
that $\mu_t(\mathcal{U})< \vre$ for all sufficiently large $t$. For
this it suffices to prove that 
\begin{equation}\label{eq: this is what we need}{
\left|\widehat{\mathcal{I}} \right| < \vre \, e^tp, \ \ \text{
  where } \widehat{\mathcal{I}} = \{s \in [0, e^tp]: u_s g_t M \in
\mathcal{U} \}.
}\end{equation}

Before proving \eqref{eq: this is what we need}, we claim that 
for all sufficiently large $t$, if $s_0 \in \widehat{\mathcal{I}}$ and
$g_t u_{s_0} M = F(M',x,y)$ with $M' \in \mathcal{W} \subset \LL$, then $y \neq 0$.
Indeed, suppose otherwise, that is there are $t_n \to \infty$, $M'_n \in
\mathcal{W}$, $|x_n| \leq i+1$ and $s_n \in [0,e^{t_n}p]$ 
such that 
$$u_{s_n}g_{t_n}M  = 
\rel_{x_n}(M'_n).$$
Set $s'_n = e^{-t_n} s_n$ so that $g_{t_n}u_{s'_n}M  = \rel_{x_n}(M'_n).
$
which implies 
$$
u_{s'_n}M = \rel_{x'_n} g_{-t_n} M'_n, \ \ \text{ where } x'_n =
e^{-t_n/2}x_n \to 0.
$$
Our hypothesis that $M$ does not belong to a closed $G$-orbit implies
that $x_n \neq 0$. 
Since $U$ commutes with $\rel_{x'_n}$, it
follows that 
\begin{equation}\label{eq: getting closer}{UM = \rel_{x'_n}U  g_{-t_n} M'_n,
}\end{equation}
and hence $Ug_{-t_n} M'_n$ is a closed horocycle of period $p$ on
$\LL$. There are only finitely many such closed horocycles on $\LL$
and their union is a compact set disjoint from the closed orbit
$UM$. Since $x'_n \to 0$, this contradicts \eqref{eq: getting closer},
and proves the claim.  

We now note that the argument given in the proof of Theorem  \ref{thm:
  generic in orbit closure} for proving \eqref{eq: need to show to
  finish} goes through, with the same
notations, and proves \eqref{eq: this is what we need}. Indeed the only information we needed in the
proof of \eqref{eq: need to show to finish} was that the number $y=y(s_0)$ considered in the
proof was nonzero, which is exactly what we have shown in the
preceding paragraph. 
\end{proof}

\begin{remark}
When $D$ is a square, Theorem \ref{thm: teaser1} can also be proved
by exploiting the connection between $\EE_D(1,1)$ and a homogeneous
space $\SL_2(\R) \ltimes \R^2 / \Gamma$ (see \cite{EMS}) and using a theorem of Shah
\cite[Thm. 3.7.6]{KSS handbook}. 
\end{remark}

Generalizing Theorem \ref{thm: teaser1} we have:

\begin{thm}\label{thm: limits of UZ}
Let $\mathcal{O}$ be a minimal set for the $U$-action, and let $\mu$
be the $U$-invariant measure on $\mathcal{O}$. Suppose that
$\mathcal{O}$ is not contained in a closed $G$-orbit. Then $g_{t*} \mu$ tends
to the flat measure on $\EE_D(1,1)$. 
\end{thm}

\begin{proof}
If $\mathcal{O}$ is a closed $U$-orbit this follows from Theorem
\ref{thm: teaser1}. According to Corollary \ref{cor: missing}, in the
remaining case $\dim \mathcal{O}=2$, the cylinder decomposition is of
type (A), and the measure $\mu$ is invariant under $UZ$. Let $\nu$ be
any limit point of $g_{t_n*}\mu$, for $t_n \to \infty$. By repeating
the arguments given in the proof of Theorem \ref{thm: teaser1} we find
that $\nu$ is a probability measure, and gives zero mass to surfaces with horizontal
saddle connections. Also clearly $\nu$ is $UZ$-invariant. Therefore,
by Claim \ref{claim: linchpin}, $\nu$ is the flat measure on
$\EE_D(1,1)$. 
\end{proof}

\begin{proof}[Proof of Theorem \ref{cor: for Giovanni}]
We show the existence of $\nu=\lim_{t \to \pm \infty}
g_{t*}\mu$ and describe this limit explicitly, for each of the
measures $\mu$ in Theorem \ref{thm: Calta Wortman 
  revisited2}. We treat the cases $t \to +\infty$ and $t\to -\infty$
separately, beginning with the case $t \to +\infty$. If $\mu$ is of
type 1, then the limit measure $\nu$ is either the flat measure on
$\EE_D(1,1)$ or a $G$-invariant measure on a closed $G$-orbit, by Theorem \ref{thm:
  teaser1}. If $\mu$ is of type 2, then Theorem \ref{thm: limits of
  UZ} implies that $\nu$ is flat measure. If $\mu$ is of type 3, 4, or
6, then there is some $T \neq 0$ such that $\mu = \mu_T$ is the pushforward of a
$G$-invariant measure under $\Psi_T$, for the map $\Psi_T$ described
in \S \ref{subsection: nonempty}. Note that this map satisfies the
following equivariance rule: for $M \in \mathcal{B}$, $g_t \Psi_T( M )
= \Psi_{e^tT}(g_tM)$. This can be seen by examining the definition of
$\Psi_T$ in each case and using Proposition \ref{prop: distributive}. 
Therefore we have $g_{t*}\mu = \mu_{e^tT}$, and
by Theorem \ref{thm: and more}, the limit $\nu$ is the flat
measure. If $\mu$ is of type 5 then $\mu = \rel_{T*}\mu_0$, where
$\mu_0$ is a $G$-invariant measure on a closed $G$-orbit. If $T=0$
then $\mu = \mu_0$ is $G$-invariant and in particular $\nu =
\mu_0$. If $T \neq 0$ then the relation $g_t \rel_T =
\rel_{e^tT} g_t$ (see Proposition \ref{prop: distributive}) implies
that $g_{t*}\mu = \rel_{e^tT*}\mu_0$ and by 
Theorem \ref{thm: teaser2}, the limit measure $\nu$ is the flat
measure. Finally in case 7 the measure $\mu$ is $G$-invariant and
there is nothing to prove. 

When $t \to -\infty$, for each of the measures of type 1, 2, 3, 4 or 6,
almost every surface in the support of $\mu$ has at least one horizontal saddle
connection of some fixed length. As $t \to -\infty$, the length of
this saddle connection tends to zero and hence $g_{t*}\mu$ diverges in the space of
measures on $\HH(1,1)$.
 If $\mu$ is of type 5 then $\mu = \rel_{T*}\mu_0$, where
$\mu_0$ is a $G$-invariant measure on a closed $G$-orbit. The commutation relation $g_t \rel_T =
\rel_{e^tT} g_t$ implies that $g_{t*}\mu = \rel_{e^tT*}\mu_0$ and
since $t \to -\infty$, the measure tends to the $G$-invariant measure
$\mu_0$. Finally in case 7 the measure $\mu$ is $G$-invariant and
there is nothing to prove. 
\end{proof}

We now collect some results which will be used in the proof of Theorem
\ref{thm: Matt absolved}. Let $B(a,b) = [-a, a] \times [-b,b]$, and  
for any $t>0$ and $v \in \R^2$, set 
$$E_{t,v} = \{g_tr_\theta v : \theta \in [0, 2\pi]\}.$$ The
following is an elementary fact about ellipses whose proof we omit. 
\begin{prop}\label{prop: ellipses in plane}
Given $\vre>0$, suppose $r_1, r_2, \delta_1, \delta_2$ satisfy the
inequalities 
\begin{equation}\label{eq: inequalities for boxes}{
r_1 < \vre r_2, \ \ \delta_1 < \vre \delta_2.}\end{equation}
Then for any 
$t>0$, either $E_{t,v} \subset B(r_2, \delta_2),$ or 
\begin{equation}\label{eq: inequalities for boxes2}{
\left| \left \{\theta \in [0, 2\pi] : g_t r_\theta v \in B(r_1, \delta_1)
  \right \} \right |
\leq  \vre \left| \left \{\theta \in [0,2\pi] : g_t
    r_\theta v \in B(r_2, \delta_2) \right \} \right| .
}\end{equation}
\end{prop}
\qed

We will also need the following analog of Corollary \ref{cor: needed
for linearization}. 
\begin{prop}\label{prop: needed for linearization2}
Given positive $\vre, T, \eta$, and a compact subset $L \subset \LL_\infty$, there
are positive  $t_0, \delta$, a neighborhood $\mathcal{W}$ of $L$ and a compact set $\Omega
\subset \LL$ containing $\mathcal{W}$, such that the map \eqref{eq: defn map 1}
is well-defined, continuous and injective on $\Omega \times [-T, T]
\times [-\delta, \delta]$, and for any $t \geq
t_0$, for any interval $I \subset J \df
\left[\frac{\pi}{2}-\eta,\frac{\pi}{2}+  \eta\right] \cup
\left[\frac{3\pi}{2}-\eta, \frac{3\pi}{2}+ \eta\right]$,
and any $M \in \HH$, if there is $\theta_0 \in I$ such that $g_t r_{\theta_0} g_{-t}M \in
\mathcal{W}$ then 
\begin{equation}\label{eq: what we need2}{
\frac{|\{\theta \in I : g_t r_\theta g_{-t} M \notin \Omega \}|}{|I|} < \vre.
}\end{equation}
\end{prop}

\begin{proof}
Using Proposition \ref{prop: continuous}, we see that it suffices to prove an analog of
Proposition \ref{cor: new nondivergence}; namely, that for any
positive $\eta, \vre, r$ and any compact $L \subset \LL_\infty$, there is a
neighborhood $\mathcal{W}$ of $L$, a compact $\Omega \subset \HH$
containing $\mathcal{W}$, and $t_0$, such that surfaces in $\Omega$ have no
horizontal saddle connections shorter than $r$, and for any $t \geq
t_0$, any interval
$I \subset J$ which contains $\theta_0$ with $g_{t} r_{\theta_0}
g_{-t} M \in \mathcal{W}$, the estimate \eqref{eq: what we need2} holds. 
This statement can be obtained from Proposition \ref{cor: new
  nondivergence} as follows.

A matrix computation shows that 
\begin{equation}\label{eq: matrix shows}{
u_{e^t \tan \theta} = a(t, \theta) g_t r_\theta g_{-t}, \ \ \text{ where } a(t,\theta) \df \left(\begin{matrix} (\cos\theta)^{-1} & 0
    \\ -e^{-t} \sin \theta & \cos \theta \end{matrix}  \right). 
}\end{equation}
We write $a(t, \theta)  = a_2(\theta) \, a_1(t, \theta),$ where  
$$
a_2(\theta) =  \left
( \begin{matrix} (\cos \theta)^{-1} & 0 \\ 0 & \cos
  \theta \end{matrix} \right), \ \ 
a_1(t, \theta ) = \left ( \begin{matrix} 1 & 0 \\ 
-e^t \tan \theta & 1 \end{matrix} \right)
. 
$$
Thus $a_1(t, \theta) \to \mathrm{Id}$ as $t \to \infty$, uniformly for
$\theta \in J$, and $a_2(\theta)$ preserves
horizontal saddle connections, changing their length by a factor of at
most $C_1 \df \max_{\theta \in J} (\cos
\theta)^{-1}.$ Also let $C_2$ be an upper bound on the derivative of
the map $\theta \mapsto \tan \theta$ on the interval $J$.  Let 
$L' = \bigcup_{\theta \in J} a_2 (\theta) L$, which is also a compact
subset of $\LL_\infty$. 
We apply Proposition \ref{cor: new nondivergence} with $C_1(r+1),
\vre/C_2, L'$ in place
of $r, \vre, L$, to obtain a neighborhood $\mathcal{W}'$ of  $L'$ and a compact
set $\Omega'$ containing $\mathcal{W}'$, such that surfaces in $\Omega'$
have no horizontal saddle connections shorter than $C_1(r+1) $, and
\eqref{eq: what we need} holds. We define: 
$$
\Omega'' = \bigcup_{\theta \in J} a_2(\theta)^{-1} \Omega'.
$$
Then $\Omega''$ is a compact set, containing no surfaces with
horizontal saddle connections of length less than $r+1$. Therefore for
$t_0$ sufficiently large, surfaces in
$$
\Omega \df \bigcup_{t \geq t_0} a_1(t,\theta)^{-1} \Omega'' =
\bigcup_{\theta \in J, t \geq t_0} a(t,\theta)^{-1} \Omega'
$$  
have no horizontal saddle connections of length less than
$r$. We will make $t_0$ larger below. We also note that $a(t,\theta)^{-1} \Omega' \subset \Omega$ for every $t
\geq t_0, \theta \in J$. Now we set 
$$\mathcal{W} = \Omega \cap \bigcap_{\theta \in J, t \geq t_0}a(t,\theta)^{-1}
\mathcal{W}',$$
so that if $g_t r_{\theta_0} g_{-t} M \in \mathcal{W}$ for some
$\theta_0 \in J$ and $t \geq t_0$, then by \eqref{eq: matrix shows}, $u_{e^t \tan
  \theta_0}M \in \mathcal{W}'$. Since \eqref{eq: what we need} holds for
$\Omega'$ (with $\vre/C_2$ instead of $\vre$), we obtain that \eqref{eq: what we
need2} holds for $\Omega$. Finally we note that if $t_0$ is chosen sufficiently large, then
$\mathcal{W}$ contains an open set containing $L$.
\end{proof}

\begin{prop} \label{prop: claim}
Suppose $\LL$ is a closed $G$-orbit in $\HH(1,1)$ and $M \notin \LL$. 
Then there are positive $\til \delta, \til t$ so that for all 
$t \geq \til t$, if there is $\theta \in [0, 2\pi]$ such that $g_t r_{\theta} M =
M' \pluscirc (x,y)$ with $M' \in \LL$, then $e^{-t}x^2+e^ty^2
\geq  \til \delta$.  

\end{prop}
\begin{proof}
Assume by contradiction that there are $M_n
\in \LL$, $t_n \to \infty$, $\theta_n \in [0, 2\pi]$ and
$x_n, y_n$ such that 
\begin{equation}\label{eq: for contradiction claim}{
e^{-t_n}x_n^2+e^{t_n}y_n^2 \to 0}\end{equation}
and $g_{t_n}
r_{\theta_n} M = M_n \pluscirc (x_n, y_n)$. 
By Corollary \ref{thm: criterion}, there are neighborhoods $\mathcal{U} \subset \HH(1,1)$
containing $M$ and $\mathcal{V} \subset \R^2$ containing 0, such that the map $(M', v)
\mapsto M' \pluscirc v$ is well-defined on $\mathcal{U} \times
\mathcal{V}$. Set 
$ v_n  = r_{-\theta_n} g_{t_n} (x_n, y_n).$ Then
  \eqref{eq: for contradiction claim} implies that $v_n \to 0$ and
  hence for all sufficiently large $n$, the maps $M' \mapsto M'
  \pluscirc \pm v_n$ are well-defined 
and continuous on $\mathcal{U}$. Set $\til M_n
  = M \pluscirc -v_n,$ so that $\til M_n \to M$. We have 
\[
\begin{split}
M & = r_{-\theta_n} g_{-t_n} \left(M_n \pluscirc (x_n, y_n)
\right) = r_{-\theta_n} g_{-t_n} M_n \pluscirc \left(  r_{-\theta_n}
  g_{-t_n} (x_n, y_n)
\right) \\ & = r_{-\theta_n} g_{-t_n} M_n \pluscirc v_n,  
\end{split}
\]
and this implies that $\til M_n = r_{-\theta_n} g_{-t_n} M_n$. We have
found a sequence in $\LL$ converging to $M$, contrary to the
assumption that $M \notin \LL$. 
\end{proof}

Now we prove Theorem \ref{thm: circle averages}.

\begin{proof}
We will use Proposition \ref{prop: abstract}. Hypothesis (a) was
verified (for any translation surface $M$) in \cite{eskinmasur}. To prove hypothesis
(b), we note that $\mu_t$ is invariant under the
conjugated group $\{g_t
r_{\theta} g_{-t} : \theta \in [0,2\pi]\}$. Conjugating we see that
$$
g_t r_\theta g_{-t} = \left ( \begin{matrix} \cos \theta & -e^t \sin
    \theta \\
e^{-t} \sin \theta & \cos \theta \end{matrix}\right).
$$
For every fixed $s$ and $t$ we set 
$$\theta(s,t) = \arcsin \left( -se^{-t}\right), \text{ so
that }-e^t \sin \theta(s,t) =s.$$
Then $g_t r_{\theta(s,t)} g_{-t}
  \to u_s$ as $t \to \infty$. Therefore the limit measure $\nu$ is
  invariant under each $u_s$.  

We now verify (c) for beds $\BB$ with $\Xi(\BB) \neq
\varnothing$. Let $K_1, K_2, \ldots$ be the sequence filling out the bed
$\BB$, as in Proposition \ref{prop: beds are
  submanifolds}.
Fix $i$ and let $\vre>0$. Then any surface in $K_i$ contains a
horizontal saddle connection of length at most $i$. 
As before, let $C(\delta)$ be the open set of surfaces whose shortest saddle
connection is shorter than $\delta$. 
By \cite{eskinmasur}, there are positive $t_1, \delta$ such that for all $t
\geq t_1$, 
$
\mu_{t}(C(\delta)) < \vre. 
$
Let $t_0$ satisfy \eqref{eq: condition on t0}, so that 
$K_i \subset \mathcal{U} \df g_{-t_0}(C(\delta))$.
Since $\mu_{t+t_0} = g_{t_0*} \mu_t$, for all $t> t_0+t_1$ we have 
$$
\mu_{t}(\mathcal{U}) = \mu_{t_0*} \mu_{t-t_0} (\mathcal{U})  =
\mu_{t-t_0}(C(\delta)) < \vre, 
$$
so (c) is satisfied. 

It remains to verify (c) for beds with $\Xi(\BB) = \varnothing.$ 
Let $F, L_i, K_i = F(L_i \times [-i, i])$ be as in the proof of Theorem \ref{thm: generic in
  orbit closure}. 
Given $i$ and $\vre>0$, choose 
\begin{equation}\label{eq: choice of j2}{
j> \frac{8(i+1)}{\vre}.
}\end{equation}
Choose $\eta>0$ small enough so that 
\begin{equation}\label{eq: choice eta}{
|J| = 
4\eta < \pi \vre, \ \ \text{where } J = \left[ \frac{\pi}{2}-\eta,
  \frac{\pi}{2}+\eta\right] \cup  \left[ \frac{3\pi}{2}-\eta,  \frac{3\pi}{2}+\eta \right].
}\end{equation}
Let $\mathcal{W}, \Omega, t_0, \delta$ be as in Proposition \ref{prop:
  needed for linearization2}, where we apply the proposition with
$\eta$ and with $j,
L_i, \vre/8$ instead of $T, L, \vre$. Let $\til \delta, \til t$ be as in Proposition \ref{prop: claim}. 
Making $\delta$ smaller and $t_0$ larger, we can assume that $t_0 \geq
\til t$ and 
\begin{equation}\label{eq: delta smaller}{
j\delta < \til \delta^2.}\end{equation}
Now let
\begin{equation}\label{eq: choice of delta'}{\delta' \in \left (0,  \frac{\vre
      \delta}{8} \right) 
}\end{equation}
and set
\begin{equation}\label{eq: defn U2}{
\mathcal{U}
= F\big(\mathcal{W} \times (-(i+1), i+1) \times (-\delta',
\delta')\big).
}\end{equation}
Then $\mathcal{U}$ is a neighborhood of $K_i$, and we need to show that 
$\mu_{t}(\mathcal{U})<\vre$ for all $t \geq t_0$. 
That is, we need to verify 
$$
\left|\widehat{\mathcal{I}} \right | < 2\pi \vre, \ \ \text{ where } \
\widehat{\mathcal{I}} = \{ \theta \in [0,2\pi] : g_t r_\theta M \in
\mathcal{U}\}. 
$$
In light of \eqref{eq: choice eta} it suffices to show that $\left|
  \widehat{\mathcal{I}} \cap J \right | < \pi \vre.$ 
Whenever $\theta_0 \in  \widehat{ \mathcal{I}} \cap J$ there are $M'  = M'(\theta_0)\in \mathcal{W}$ and
$(x, y) = (x(\theta_0), y(\theta_0)) \in (-(i+1), i+1) \times (-\delta', \delta')$
such that $g_t r_{\theta_0}M = F(M', x,y)$. We define the following intervals: 
\[ 
\begin{split}
\mathcal{I} = \mathcal{I}(\theta_0) & \df 
\{\theta \in [0, 2\pi] : g_t r_{\theta - \theta_0}g_{-t} (x, y) \in B(i+1, \delta')  \}
\\
\mathcal{J} = \mathcal{J}(\theta_0) & \df \{ \theta \in [0,2\pi] :g_t
r_{\theta - \theta_0} g_{-t} (x, y) \in B(j, \delta) \}.
\end{split}
\]
We think of $[0, 2\pi]$ as a circle by identifying $0$ and $2\pi$, and
think of these subsets as arcs on the circle. 
In view of Proposition \ref{prop: claim} and the choice
of $t_0$, we have $e^{-t}x^2 + e^{t}y^2 \geq \til \delta$. Using
\eqref{eq: delta smaller} and considering the two choices of 
$\theta$ which make $r_{\theta - \theta_0} g_{-t} (x,y)$ horizontal
and vertical, we see that $\mathcal{J}$ does not coincide with the
entire circle. Then \eqref{eq: choice of j2}, \eqref{eq: choice of delta'}
and Proposition \ref{prop: ellipses in plane} imply that 
\begin{equation}\label{eq: relative lengths2}{
\frac{|\mathcal{I} |}{|\mathcal{J} |} < \frac{\vre}{8}. 
}\end{equation}

We claim that
if $\theta \in \mathcal{J} \sm \mathcal{I}$ then either $g_tr_{\theta
  - \theta_0} g_{-t}M'
\notin \Omega$ or $g_t r_{\theta}M \notin \mathcal{U}$. 
Indeed, suppose $\theta \in \mathcal{J} \sm \mathcal{I}$ and
$g_tr_{\theta- \theta_0} g_{-t}M'
\in \Omega.$ According to Proposition \ref{prop: distributive}, 
\[
\begin{split}
 g_tr_{\theta}M & 
= g_tr_{\theta-\theta_0} g_{-t} g_t r_{\theta_0} M =
g_tr_{\theta-\theta_0} g_{-t} (M' \pluscirc (x,y) ) 
\\ & 
= g_tr_{\theta - \theta_0} g_{-t} M' \pluscirc
g_tr_{\theta - \theta_0}  g_{-t}(x, y). 
\end{split}
\]
The injectivity of $F$ on $\Omega \times [-j,j] \times [-\delta, \delta]$
implies that $g_tr_{\theta}M \notin \mathcal{U}$.
This proves the claim.

The claim, combined with \eqref{eq: inequalities for boxes2} 
and \eqref{eq: what we need2}, implies that 
$$
\frac{\left| \widehat{\mathcal{I}} \cap \mathcal{J} \right
  |}{| \mathcal{J}| }
\leq \frac{|\mathcal{I} |}{|\mathcal{J} |} + \frac{|\{ \theta
  \in \mathcal{J} : g_tr_{\theta-\theta_0} g_{-t} M' \notin \Omega |}{|\mathcal{J}|}< \frac{\vre}{4}.
$$

We have covered  $\widehat{\mathcal{I}} \cap J$ by the intervals
$\left \{\mathcal{J}(\theta_0) +\theta_0 : \theta_0 \in
  \widehat{\mathcal{I}} \right \}$, and
we pass to a subcover such that 
$$
\theta \in \widehat{\mathcal{I}} \cap J \ \implies \ 1 \leq 
\# \, \{\ell: \theta \in
\mathcal{J}'_\ell\} \leq 2,
$$
and obtain:
\[
\left| \widehat{\mathcal{I}} \cap J  \right | 
\leq \sum_{\ell} \left|
\widehat{\mathcal{I}} \cap 
\mathcal{J}'_\ell \right | < \frac{\vre}{4} \sum_{\ell} \left|
\mathcal{J}'_{\ell} \right | < 
\pi \vre.
\]
\end{proof}

\begin{proof}[Proof of Theorem \ref{thm: Matt absolved}]
This follows immediately from Theorem \ref{thm: circle averages} by an
argument developed by Eskin and Masur \cite{eskinmasur}. See \cite{EMS, EMM, Matt} for more
details. 
\end{proof}

\bibliographystyle{amsalpha}

\printindex

\end{document}